\documentclass{amsart}

\usepackage[foot]{amsaddr}

\usepackage{amsmath,amssymb}
\usepackage{mathrsfs}

\usepackage{tikz-cd}

\usepackage{bbm}
\usepackage{xltabular}

\usepackage[english]{babel}
\usepackage[utf8]{inputenc}
\usepackage[T1]{fontenc}
\usepackage{mathtools}
\usepackage{lmodern}

\usepackage[top=3cm,bottom=2cm,left=3cm,right=3cm,marginparwidth=2.5cm]{geometry}

\usepackage[colorinlistoftodos]{todonotes}
\usepackage[hypertexnames=false]{hyperref}
\hypersetup{
    colorlinks,
    linkcolor={red!50!black},
    citecolor={blue!50!black},
    urlcolor={blue!80!black},
}

\usepackage[shortlabels]{enumitem}
\setlist[1]{itemsep=2pt}
\usepackage[capitalise]{cleveref}

\usepackage{autonum} 

\DeclarePairedDelimiter\ceil{\lceil}{\rceil}
\DeclarePairedDelimiter\floor{\lfloor}{\rfloor}

\newcommand{\norm}[1]{\left\lVert#1\right\rVert}
\newcommand{\R}{\mathbb{R}}
\newcommand{\C}{\mathbb{C}}
\newcommand{\N}{\mathbb{N}}

\newcommand{\alg}{\mathfrak{g}}	
\newcommand{\gr}{\mathbb{G}}		
\newcommand{\End}{\mathcal{E}}	
\newcommand{\V}{V}				

\newtheorem{theorem}{Theorem}
\newtheorem{proposition}[theorem]{Proposition}

\newtheorem{lemma}[theorem]{Lemma}

\newtheorem{theoremintro}{Theorem}						
\crefname{theoremintro}{Theorem}{Theorems}				

\newtheorem{propositionintro}[theoremintro]{Proposition}			
\crefname{propositionintro}{Proposition}{Propositions}			

\newtheorem{corollaryintro}[theoremintro]{Corollary}			
\crefname{corollaryintro}{Corollary}{Corollaries}			

\theoremstyle{definition}
\newtheorem{definition}[theorem]{Definition}

\theoremstyle{remark} 
\newtheorem{remark}[theorem]{Remark}
\newtheorem{example}[theorem]{Example}

\crefname{construction}{Construction}{Constructions}				

\AddToHook{env/proposition/begin}{\crefalias{theorem}{proposition}}
\AddToHook{env/corollary/begin}{\crefalias{theorem}{corollary}}
\AddToHook{env/lemma/begin}{\crefalias{theorem}{lemma}}
\AddToHook{env/definition/begin}{\crefalias{theorem}{definition}}
\AddToHook{env/remark/begin}{\crefalias{theorem}{remark}}
\AddToHook{env/example/begin}{\crefalias{theorem}{example}}
\AddToHook{env/construction/begin}{\crefalias{theorem}{construction}}
\AddToHook{env/propositionintro/begin}{\crefalias{theoremintro}{proposition}}
\AddToHook{env/corollaryintro/begin}{\crefalias{theoremintro}{corollary}}

\DeclareMathOperator{\op}{op}
\DeclareMathOperator{\graf}{graph}

\DeclareMathOperator{\crit}{Crit}
\DeclareMathOperator{\sym}{Sym}
\DeclareMathOperator{\im}{im}

\title[Sard properties for polynomial maps in infinite dimension]{Sard properties for polynomial maps in infinite dimension}
\author{Antonio Lerario}
\email{\href{mailto:lerario@sissa.it}{lerario@sissa.it}}

\author{Luca Rizzi}
\email{\href{mailto:lrizzi@sissa.it}{lrizzi@sissa.it}}

\author{Daniele Tiberio}
\email{\href{mailto:lrizzi@sissa.it}{dtiberio@sissa.it}}

\address{SISSA, via Bonomea 265, 34136 Trieste, Italy}

\date{\today}

\begin{document}

		\begin{abstract}Sard's theorem asserts that the set of critical values of a smooth map from one Euclidean space to another one has measure zero. A version of this result for infinite-dimensional Banach manifolds was proven by Smale for maps with Fredholm differential. It is well--known, however, that when the domain is infinite dimensional and the range is finite dimensional, the result is not true --  
even under the assumption that the map is ``polynomial'' --  and a general theory is still lacking. Addressing this issue,  in this paper, we provide sharp quantitative criteria for the validity of Sard's theorem in this setting. Our motivation comes from sub--Riemannian geometry and, as an application of our results, we prove the sub--Riemannian Sard conjecture for the restriction of the Endpoint map of Carnot groups to the set of piece--wise real--analytic controls with large enough radius of convergence, and the strong Sard conjecture for the restriction to the set of piece--wise entire controls.
\end{abstract}

\maketitle

\tableofcontents

	\section{Introduction}
	\subsection{The Sard property for polynomial maps on infinite-dimensional spaces}
	The classical Sard theorem \cite{Morse, Sard} states that the set of critical values of a smooth map $f:N\to M$, where $N$ and $M$ are finite dimensional manifolds, has measure zero in $M$. Smale \cite{smale} proved a version of this result in the case both $M$ and $N$ are infinite  dimensional Banach manifolds and the map $f$ is Fredholm: in this case the conclusion is that the set of regular values of $f$ is residual in $N$.
	
	When $N$ is infinite-dimensional and $M$ is finite dimensional, $f$ cannot be Fredholm and Smale's result cannot be applied. In fact there are smooth
surjective maps without regular points from any infinite-dimensional Banach space to $\R^2$, see \cite{Bates}. Even under the hypothesis that $f$ is a \emph{polynomial map}, the set of its critical values can be of positive measure: Kupka \cite{Kupka} constructed an example of a smooth map $f:\ell^2\to \R$, whose restriction to each finite dimensional space is a polynomial of degree $3$, and with the property that the set of its critical values is the segment $[0,1]$. On the other hand, using a quantitative version of the classical Sard theorem \cite{Yomdinnearly}, Yomdin  has shown that if a map $f:\ell^2\to \R$ can be approximated well-enough with finite dimensional polynomials, still the set of its critical values has measure zero, see \cite{YomdinApproxCompl}. (Kupka's map does not have this property.) In Yomdin's proof, the fact that the codomain is one-dimensional is essential and new technical difficulties arise for maps with values in $\R^m$, $m>1$. 
	
	In this paper we investigate the Sard property for polynomial maps defined on Hilbert spaces and with values in $\R^m$, for general $m\geq1$.   

\begin{definition}\label{def:polyH}
Given a Hilbert space $H$ and $d,m \in \N$, we define  the class of maps 
	\begin{equation}
	\mathscr{P}_{d}^m(H):=\bigg\{f: H\to \R^m\,\bigg|\, \dim(E)<\infty \implies f|_{E} \textrm{ is a polynomial map of degree $d$}\bigg\},
	\end{equation}
	where by polynomial map we mean that each component of $f|_E$ is a polynomial of degree $d$. We require that each element of $\mathscr{P}_{d}^m(H)$ is of class $\mathcal{C}^1$ (in the Fréchet sense), and that its differential $Df: H \to \mathcal{L}(H,\R^m)$ is weakly continuous and locally Lipschitz.
\end{definition}
The map constructed by Kupka belongs to this family, as well as the Endpoint maps for horizontal path spaces on Carnot groups (see \cref{sec:introsubriemannian}).
	
		We need to introduce a notion of \emph{quantitative compactness} for subsets of Hilbert spaces.
		More precisely, given $K\subset H$ and $n\in \mathbb{N}$, we denote by $\Omega_n(K, H)$ its \emph{Kolmogorov $n$--width}:
	\begin{equation}\label{eq:defnwidthintro}	
		\Omega_n(K, H):=\inf_{\dim(E)=n} \sup_{u\in K}\inf_{v\in E}\|v-u\|,
	\end{equation}
	where $\|\cdot\|$ denotes the norm of the Hilbert space $H$, and the infimum ranges over all vector subspaces $E\subset H$ of dimension $n$ (see \cref{def:nwidth}).
	A set $K\subset H$ is compact if and only if it is bounded and its $n$--width goes to zero as $n$ goes to infinity.

Our results are stated in terms of \emph{entropy dimension} of the set of \emph{$\nu$-critical values} of a map $f:H\to\R^m$, namely $f(\crit_\nu(f))$, where $\crit_\nu(f)$ is the set of points where $Df$ has rank at most $\nu \leq m-1$. The classical set of critical values corresponds to setting $\nu = m-1$, and it is denoted by $f(\crit(f))$, where $\crit(f)$ is the set of critical points. We refer to \cref{def:entropy} of entropy dimension (denoted by $\dim_e$, and also known as Minkowski dimension), noting that it is larger than the Hausdorff one so that in the forthcoming estimates one can replace the former with the latter for simplicity.

Our first theorem is a sufficient condition for the validity of the Sard property (see \cref{CorollaryOfGeneralSard}). We denote by $\mu$ the Lebesgue measure on $\R^m$.

\begin{theoremintro}[Sard under $n$--width assumptions]\label{CorollaryOfGeneralSard-intro}
Let $H$ be a Hilbert space and $K \subset H$ be a compact set such that, for some $q>1$, it holds
\begin{equation}
\limsup_{n\to \infty}\Omega_n(K,H)^{1/n} \leq q^{-1}.
\end{equation}
Let $d,m \in \N$. There exists $\beta_0 = \beta_0(d,m)>1$ such that for all $f\in \mathscr{P}_{d}^m(H)$ and $\nu \leq m-1$ we have	
\begin{equation}
		\dim_e \bigg(f\big( \crit_{\nu} (f) \cap K \big)\bigg)
		\leq 
		\nu + \frac{\ln \beta_0}{ \ln q}.
	\end{equation}
	In particular, if $q> \beta_0$, then the Sard property holds on $K$:
	\begin{equation}
	\mu\bigg(f\big(\crit (f) \cap K\big)\bigg)=0.
	\end{equation} 
\end{theoremintro}
\begin{remark}\label{rmk:beta0}
The constant $\beta_0(d,m)$ of \cref{CorollaryOfGeneralSard-intro} comes from semialgebraic geometry, and its origin is discussed in \cref{sec:introsemialg}. We sometimes call it the \emph{semialgebraic constant}.
\end{remark}

By building upon Kupka's counterexamples \cite{Kupka}, we prove that the quantitative assumption $q>\beta_0$ in \cref{CorollaryOfGeneralSard-intro} cannot be dispensed of. The following result corresponds to \cref{thm:sintesiexamples}. We denote by $B_X(r)$ the closed ball of radius $r$ in a Banach space $X$.

\begin{theoremintro}[Counterexamples to Sard]\label{thm:sintesiexamples-intro}
Let $d,m \in \N$, with $d\geq 3$, and let $1<q<(d-1)^{1/d}$. There exist a Hilbert space $H$, and $f\in \mathscr{P}_d^m(H)$ such that $K=\crit(f)\cap B_H(r)$ is compact for all $r>0$, with
\begin{equation}
\limsup_{n\to \infty} \Omega_n(K,H)^{1/n} \leq q^{-1},
\end{equation}
and $f : H \to \R^m$ does not verify the Sard property, namely $\mu(f(\crit(f)\cap K)) > 0$. Therefore, the semialgebraic constant $\beta_0(d,m)$ of \cref{CorollaryOfGeneralSard-intro} satisfies
\begin{equation}
\beta_0(d,m) \geq (d-1)^{1/d}, \qquad \forall\, m\in\N,\, d\geq 3.
\end{equation}
\end{theoremintro}

\begin{remark}
The case $d=2$, which is left out from \cref{thm:sintesiexamples-intro}, is special. This is related to the fact that the complexity of semialgebraic sets defined by quadratic equations has a different behaviour compared to the case $d\geq 3$. We currently expect that \emph{any} $f\in \mathscr{P}^m_2(H)$ satisfies the Sard property globally on $H$. This will be the subject of a future work.
\end{remark}

Replacing the entropy dimension with the Hausdorff one (which is smaller than the former, but countably stable) we strengthen \cref{CorollaryOfGeneralSard-intro} as follows (see \cref{thm:main2}).

\begin{theoremintro}[Sard on linear subspaces]\label{thm:main2intro}
Let $H$ be a Hilbert space and $K \subset H$ be a compact set such that, for some $q>1$, it holds
\begin{equation}
\limsup_{n\to \infty}\Omega_n(K,H)^{1/n} \leq q^{-1}.
\end{equation}
Let $d,m \in \N$, let $\beta_0=\beta_0(d,m)>1$ be the same constant of \cref{CorollaryOfGeneralSard-intro}. Consider the (possibly non-closed) linear subspace
\begin{equation}
\V:= \mathrm{span}(K).
\end{equation}
Then for all $f\in \mathscr{P}_{d}^m(H)$ and $\nu\leq m-1$  the restriction $f|_{\V}: \V \to \R^m$ satisfies
\begin{equation}
\dim_{\mathcal{H}}\bigg(f\big(\mathrm{Crit}_\nu(f|_{\V})\big)\bigg)\leq \nu + \frac{\log\beta_0}{\log q}.
\end{equation}
In particular, if $q>\beta_0$, then the restriction $f|_{\V}: \V \to \R^m$ satisfies the Sard property
\begin{equation}
	\mu\bigg(f\big(\mathrm{Crit}(f|_{\V})\big)\bigg)=0.
\end{equation}
In all previous results, $f\big(\mathrm{Crit}_{\nu}(f|_{\V})\big)$ can be replaced with the smaller set $f(\mathrm{Crit}_{\nu}(f)\cap \V)$.
\end{theoremintro}

\begin{remark}\label{rmk:frechetdiffrestr}
The restriction $f|_{\V}: \V \to \R^m$ is a $\mathcal{C}^1$ map in the Fréchet sense from the normed vector space $(\V,\|\cdot\|)$ to $\R^m$, with $D_u (f|_{\V}) = (D_u f)|_{\V}$ for all $u\in \V$.
\end{remark}

\begin{example}
Let $q \in (1,\infty)$ with $q>\beta_0(d,m)$, where $\beta_0(d,m)>1$ is the number in \cref{CorollaryOfGeneralSard-intro,thm:main2intro}, for $d,m \in \N$. Let $H=\ell^2$ with the usual norm. Consider the subset $K\subset B_H(1)$ given by
\begin{equation}
K = \left\{u \in \ell^2 \,\middle\vert\,   \sum_{j=1}^\infty |u_j|^2 q^{2j} \leq 1\right\}.
\end{equation}
One can easily verify that $K$ is compact. Furthermore, the linear space $\V = \mathrm{span}(K)$ is dense in $H$. We can estimate the $n$--width of $K$ as follows. For $n\in \N$ consider the $n$-dimensional subspaces $E_n = \{x\in \ell^2\mid x_j=0,\,\forall\, j\geq n+1\} \subset H$. From \eqref{eq:defnwidthintro}, we obtain
\begin{align}
\Omega_n(K,H)^2 & \leq \sup_{u\in K} \inf_{v\in E_n} \|u-v\|^2 \\
& = \sup_{u\in K} \sum_{j=n+1}^{\infty} |u_j|^2\\
& \leq  q^{-2(n+1)} \sup_{u\in K} \sum_{j=1}^{\infty} |u_j|^2 q^{2j} \leq  q^{-2(n+1)}.
\end{align}
Hence $\limsup_{n} \Omega_n(K,H)^{1/n} \leq q^{-1}$, with $q > \beta_0(d,m)$. Thus, by \cref{thm:main2intro} it holds (in particular)
\begin{equation}
\mu\bigg(f\big(\mathrm{Crit}(f|_{\V})\big)\bigg) = \mu\bigg(f\big(\mathrm{Crit}(f)\cap \V \big)\bigg) = 0, \qquad \forall\,f\in \mathscr{P}_d^m(H).
\end{equation}
Note that even if $\V \subset H$ is dense, the unrestricted map $f :H\to \R^m$ may not have the Sard property (see e.g. the Kupka counterexamples in \cref{sec:point2}).
\end{example}

\Cref{CorollaryOfGeneralSard-intro,thm:mainintro,thm:main2intro} are deduced from a more general result for maps that are ``well-approximated'' by polynomials. This can be regarded as our main result concerning the Sard property, and does not make use of the concept of $n$--width. We report here the statement (see \cref{GeneralSardTheorem}).

\begin{theoremintro}[Sard criterion for well-approximated maps]\label{GeneralSardTheorem-intro}
Let $d,m \in \N$. There exists a constant $\beta_0 = \beta_0(d,m) >1$ such that the following holds. Let $H$ be a Hilbert space, and let $f: H\to \R^m$ be a $\mathcal{C}^1$ map such that its differential $Df : H \to \mathcal{L}(H,\R^m)$ is weakly continuous. Let $B\subset H$ be a bounded set with this approximation property: there exist a sequence $E_n\subset H$ of linear subspaces, $\dim(E_n) = n$, and polynomial maps $f_n: E_n \to \R^m$ with uniformly bounded degree:
\begin{equation}
\sup_{n \in \N } \deg f_n \leq d < \infty,
\end{equation}
such that for some $q>1$, $c\geq 0$, and all large enough $n$ it holds
\begin{equation}\label{ipotesiMainTheorem-intro}
		\sup_{x \in B}
		\bigg(\|f(x)  - f_n \circ \pi_{E_n} (x) \|
		+
		\| \left(D_x f\right)|_{E_n}
		- D_{\pi_{E_n}(x)} f_n  \|_{\op}\bigg)
		\leq 
		c q^{-n}.
	\end{equation}
Then for all $\nu \leq m-1$ it holds
\begin{equation}
		\dim_e\bigg( f\big( \crit_{\nu} (f) \cap B \big)\bigg) 
		\leq 
		\nu + \frac{\ln \beta_0}{ \ln q}.
	\end{equation}
In particular, if $q> \beta_0$, then $f$ satisfies the Sard property on $B$:
	\begin{equation}
	\mu\bigg(f(\crit(f)\cap B)\bigg)=0.
	\end{equation}
\end{theoremintro}

By means of the general \cref{GeneralSardTheorem-intro} we single out a class of maps satisfying the Sard property on the whole domain of definition. This is the content of \cref{thm:sardseries}, that we report here. The case $m=1$ corresponds to \cite[Thm.\ 10.12]{ComteYomdin}. In the statement, we fix a Hilbert basis $\{e_i\}_{i \in \N} \subset H$. For $k \in \N$, we set $E_k :=\mathrm{span}\{e_1, \ldots, e_k\} $ and we denote by $\pi_k: H  \to E_k $ the corresponding orthogonal projection, that is if  $x = \sum_{i} x_i e_i$, then $\pi_k (x) = (x_1 , \dots , x_k)$.

\begin{theoremintro}[Global Sard for special maps]\label{thm:main3intro}
Let $H$ be a separable Hilbert space. For all $k \in \N$, let $p_k: E_k \to \R^m$ be polynomial maps with 
	$\sup_{k \in \N} \deg p_k \leq d$ for some $d \in \N$, and such that
	\begin{equation}\label{eq:boundball1-intro}
\sup_{x \in B_{E_k} (1)} 
\|p_k (x)\|
\leq  q^{-k}, \qquad \forall\, k \in \N,
\end{equation}
for some $q>1$. Then the map $f: H \to \mathbb{R}^m$ defined by
\begin{equation}
f(x):= \sum_{k=1}^{\infty} p_k(x_1 , \dots ,x_k), \qquad \forall \, x \in H,
\end{equation}
is well--defined, $f \in \mathscr{P}_d^m(H)$ (see \cref{def:polyH}), and for all $\nu \leq m-1$ and $r>0$ it holds
\begin{equation}\label{eq:entropyestimatemaps-intro}
	\dim_e \bigg(f\left( \crit_{\nu} (f) 
	\cap 
	B_{H}(r)  \right)\bigg)
	\leq 
	\nu + 
	 \frac{\ln \beta_0}{ \ln q},
\end{equation}
where $\beta_0 =\beta_0(d,m)>1$ is the same constant given by \cref{GeneralSardTheorem-intro}. In particular, if $q> \beta_0$, then $f$ satisfies the Sard property globally on $H$:
\begin{equation}
\mu\bigg(f\big(\crit (f) \big)\bigg)=0.
\end{equation}
\end{theoremintro}

As a consequence of \cref{CorollaryOfGeneralSard-intro,thm:sintesiexamples-intro}, we see that there is a threshold on the $n$--width for the validity of the Sard property. Motivated by this fact, for a compact set $K\subset H$ we define the quantity:
\begin{equation}\label{eq:defomegaKH}
	\omega(K , H):=
	\limsup_{n\to \infty}
	\Omega_n(K, H)^{1/n}
	\in [ 0 , 1].
\end{equation}
Roughly speaking, smaller $\omega(K,H)$ means that $K$ is ``more compact'' in $H$, and better approximated by finite-dimensional subspaces.
We obtain the following statement (see \cref{thm:main}).
	\begin{theoremintro}[Sard threshold on compacts]\label{thm:mainintro}
	For all $d,m\in \mathbb{N}$, there exists $\omega_0(d,m)\in (0,1]$ such that
	\begin{enumerate}[(i)]
	\item \label{item:1mainintro} for every $f\in \mathscr{P}_{d}^m(H)$ and for every compact set $K\subset H$ with $\omega(K,H)<\omega_0(d,m)$,
	\begin{equation}
	\mu\bigg(f\left(\mathrm{crit}(f)\cap K\right)\bigg)=0;
	\end{equation}
	\item \label{item:2mainintro} for every $\omega>\omega_0(d,m)$, with $\omega \in (0,1]$, there exist $f\in \mathscr{P}_{d}^m(H)$ and a compact set $K\subset H$ with $\omega(K,H)=\omega$ and such that
	\begin{equation}
	\mu\bigg(f\left(\mathrm{crit}(f)\cap K\right)\bigg)>0.
	\end{equation}
	\end{enumerate}
	\end{theoremintro}
\begin{remark}
By construction, and \cref{CorollaryOfGeneralSard-intro}, $\omega_0\geq \beta_0^{-1}$. Furthermore, if $d\geq 3$, \cref{thm:sintesiexamples-intro} yields $\omega_0 <1$, so that \cref{item:2mainintro} is non-vacuous in these cases.
\end{remark}

	\subsection{Quantitative semialgebraic geometry}\label{sec:introsemialg}	
	
	The proof of \cref{GeneralSardTheorem-intro} needs some fine properties of semialgebraic sets. The strategy of the proof is conceptually similar to Yomdin's \cite{YomdinApproxCompl}, and uses the theory of \emph{variations}, introduced by Vitushkin \cite{Vit1, Vitbook} and developed in \cite{ComteYomdin}. 

Let us denote by $V_i(S)$ the $i$-th variation of a semialgebraic set, which is a sort of $i$-dimensional volume, (see \cref{defvariation}). Furthermore, recall that for a $C^1$ map $F:\R^n\to \R^m$, and $\Lambda=(\Lambda_1, \ldots, \Lambda_m)\in \R^m_+$ the set of almost-critical values of $F$ is defined by
	\begin{equation}
	C^{\Lambda}(F):=
	\bigg\{ x \in \R^n\,\bigg| \,
	\sigma_i(D_x F ) 
	\leq \Lambda_i, 
	\quad
	\forall\, i=1,\dots,m
	\bigg\},
	\end{equation}
	where $\sigma_1(D_x F)\geq \cdots \geq \sigma_m(D_xF)$ are the singular values of $D_xF$, see \cref{def:singval} (here we assume $n\geq m$). 	In \cite[Cor.\ 7.4]{ComteYomdin} a quantitative estimate on the variations of the almost-critical values of polynomial maps has been obtained. 
	In that estimate the dimensional parameter $n$ does not appear explicitly. We obtain \cref{thm:estimatevariations-intro} below, which makes this dependence explicit (see \cref{thm:estimatevariations}).
	
	\begin{theoremintro}[Quantitative variations estimates]\label{thm:estimatevariations-intro}
	Let $n\geq m$, and $p: \R^n \to \R^m$ be a polynomial map with components of degree at most $d$. For  $ i= 0, \dots, m$, $\Lambda=(\Lambda_1 , \dots , \Lambda_m)\in \R^m_+$ and $r >0$, we have  
				\begin{equation}
		V_i( p ( C^{ \Lambda }( p ) \cap B_{\R^n}(r) ))
			\leq
			\mathrm{cst}(m,r)n^m			\beta_0^n \Lambda_{0} 
			\cdots \Lambda_{i},
		\end{equation}
					where $\beta_0=\beta_0(d,m)>1$ depends only on $d$ and $m$, $\mathrm{cst}(m,r)$ depends only on $m,r$, and we set $\Lambda_0=1$.
	\end{theoremintro}	

A key technical step in the proof of \cref{thm:estimatevariations-intro} is a quantitative substitute for the \emph{definable choice} theorem in semialgebraic geometry, which we prove in \cite{LRT-selection} (see \cref{thm:SemialgebraicSelection}).

	\subsection{Applications to Endpoint maps of Carnot groups}\label{sec:introsubriemannian}
		We discuss now the implications of \cref{thm:mainintro} and \cref{thm:main2intro} in the context of sub-Riemannian geometry.
		
Recall that an $m$-dimensional \emph{Carnot group} of step $s \in \N$ is a connected and simply connected Lie group $(\gr, \cdot)$ of dimension $m$, whose Lie algebra $\mathfrak{g}$ admits a stratification of step $s$, that is 
	\begin{equation}
		\alg 
		= 
		\alg_1 \oplus \dots \oplus \alg_s,
	\end{equation}
	where 
	$\alg_i \neq \{0\}$, 
	$\alg_{i+1} = [\alg_1,\alg_i]$ 
	for all $i=1,\dots,s-1$ and $[\alg_1,\alg_s]= \{0\}$. The group exponential map yields an identification $\gr\simeq \R^m$, and the first stratum $\alg_1$ of the Lie algebra defines a smooth, totally non-holonomic distribution $\Delta\subseteq T\R^m$ of rank $k:=\dim\alg_1$. 
	
	Fix a global trivialization of $\Delta$:
	\begin{equation}
	\Delta=\mathrm{span}\{X_1, \ldots , X_k\},
	\end{equation}
	where each $X_i$ is a left-invariant vector field. Let $I:=[0,1]$ be the unit interval and $H:=L^2(I,\R^k)$.  We say that a curve $\gamma: I \to \gr$ is \emph{horizontal} if it is absolutely continuous and there exist $u\in H$, called \emph{control} such that for a.e.\ $t\in I$ it holds
	\begin{equation}\label{eq:cauchyintro}
	\dot\gamma(t)=\sum_{i=1}^k u_i(t)X_i(\gamma(t)),
	\end{equation}
Furthermore, for any given $u\in  H$ there exists a unique $\gamma_u : I\to \gr$ satisfying \eqref{eq:cauchyintro} and such that $\gamma_u(0)=e$ (the identity $e\in \gr$, identified with $0\in \R^m$). Note that the class of horizontal curves does not depend on the choice of the global trivialization of $\Delta$.
	
The \emph{Endpoint map} is the map that sends a control $u$ to the the corresponding final point $\gamma_u(1)$:
	\begin{equation}
	\End: L^2(I,\R^k) \to \gr.
	\end{equation}

The \emph{Sard conjecture} is the conjecture that Endpoint map has the Sard property:
\begin{center}
\textbf{Sard conjecture:} the set $\End(\mathrm{Crit}(\End))$ has zero measure.
\end{center}
The conjecture was introduced by Zhitomirskii and Montgomery, see \cite[Sec.\ 10.2]{Montgomerybook} for general sub-Riemannian structures, where the Endpoint map can be defined in a similar way, on a suitable domain of horizontal paths, which has the structure of a Hilbert manifold. For the specific case of Carnot groups, it is known to be true for step $\leq 2$, see \cite{AGL-pathspace,LDMOPV-Sardprop}. Furthermore, the conjecture has been verified in \cite{BNV-Filiform} for filiform Carnot groups, in \cite{BV-Dynamical} for Carnot groups of rank 2 and step 4, and rank 3 and step 3, and for a handful of other specific examples described in \cite{LDMOPV-Sardprop}.

Assume that $\Delta$ is also endowed with a left-invariant norm. This choice induces a length structure on the space of horizontal curves, and correspondingly, a distance, which gives $\gr$ the structure of \emph{sub-Finsler Carnot group} (or \emph{sub-Riemannian Carnot group} if the norm is induced by a scalar product). In this case, within the set of critical points of $\End$ (also called singular horizontal paths), the energy-minimizing ones play a significant role. At the corresponding critical values, the distance loses regularity: it is never smooth and it can even lose local semiconcavity \cite[Sec.\ 4.2]{BR-Inv}. For this reason, the understanding of the ``abundance'' of such critical values is crucial. Denoting by $H_{\min}\subseteq L^2(I,\R^k)$ the set of minimizing horizontal paths starting from the identity, the \emph{minimizing Sard conjecture} can be formulated as follows:
\begin{center}
\textbf{Minimizing Sard conjecture:} the set $\End(\mathrm{Crit}(\End)\cap H_{\min})$ has zero measure.
\end{center}
The most general result to date is that the above set is a closed nowhere dense set \cite{agrasmoothness,RT-MorseSard}, which of course does not imply the minimizing Sard conjecture. The above conjecture is one of the main open problems in sub-Riemannian geometry \cite[Prob.\ 3]{A-openproblems}, \cite[Conj.\ 1]{RT-MorseSard}. In general, the problem is settled in the following cases:
\begin{itemize}
\item[-] Thanks to the work of Agrachev-Sarychev \cite{AS-Morse} and Agrachev-Lee \cite{AAPL-OT}, in absence of so-called Goh singular minimizing paths, all sub-Riemannian minimizing paths (resp.\ sub-Finsler, under suitable smoothness assumptions on the norm) are solution of a finite-dimensional Hamiltonian flow so that the corresponding set of critical values has zero measure by the finite-dimensional Sard theorem. 
Recently, Rifford proved that, more generally, the minimizing Sard conjecture holds for structures where all non-trivial Goh paths have Goh-rank $\leq 1$ almost everywhere, see \cite{R-subdiff}. For Carnot groups, these results can be applied for example when the distribution is \emph{pre-medium fat}, see \cite[(1.2)]{R-subdiff}. In passing, we remark that for generic (so, typically not Carnot) sub-Riemannian structures with distribution of rank $\geq 3$, there are no non-trivial singular Goh minimizing paths, see \cite[Cor.\ 2.5]{CJT-generic} and \cite[Thm. 8]{AG-subanalitic}. As a consequence the minimizing Sard conjecture holds true. Thanks to \cite[Cor.\ 1.4]{R-subdiff}, the latter result extends the case of rank $\geq 2$.
\item[-] The minimizing Sard conjecture is true also for Carnot groups of step $\leq 3$, see \cite{LDMOPV-Sardprop}. In this case, the problem is reduced (in a non-trivial way) to a finite-dimensional one. First, by noting that singular minimizing curves are non-singular in some proper subgroup, and then exploiting the fact that Carnot subgroups are parametrized by a finite-dimension manifold. 
\end{itemize}

Finally, we refer to \cite{ZZ-Rigid,BR-DukeMartinet, BFPR-StrongSardInventiones, BPR-Sardpreprint1, BPR-Sardpreprint3, BPR-Sardpreprint2, OV-codimensionCarnotSard, RT-MorseSard} for further works on various forms of the Sard conjecture on general sub-Riemannian structures.

The following result (\cref{RestrictionEndpointPolynomial}) connects this framework to the previous sections.
	\begin{propositionintro}[Polynomial properties of the Endpoint map]\label{RestrictionEndpointPolynomial-intro}
	Let $\gr$ be a Carnot group of topological dimension $m$, step $s$, and rank $k$. Then the Endpoint map $\End\in \mathscr{P}_s^m(H)$, for $H=L^2(I,\R^k)$
	\end{propositionintro}

As a consequence of \cref{RestrictionEndpointPolynomial-intro}, we can apply \cref{thm:main2intro} to the study of Endpoint maps of Carnot groups. We record here a first immediate corollary.

\begin{corollaryintro}[Sard criterion for Endpoint maps]\label{RisultatoEndpointGenerale-intro}
Let $\gr$ be a Carnot group of topological dimension $m$, step $s$, and rank $k$. Let $K \subset L^2(I,\R^k)$ be a compact subset and let $q>1$ be such that
\begin{equation}\label{eq:RisultatoEndpointGenerale-intro}
\limsup_{n\to\infty} \Omega_n (K , L^2(I, \R^k) )^{1/n}\leq q^{-1}.
\end{equation}
Then, letting $\V:= \mathrm{span}(K)$, the restriction $\End|_{\V}: \V \to \gr$ satisfies
\begin{equation}
\dim_{\mathcal{H}}\bigg(\End\big(\mathrm{Crit}(\End|_{\V})\big)\bigg)\leq m-1 + \frac{\log\beta_0}{\log q},
\end{equation}
where $\beta_0=\beta_0(s,m)>1$ is the same constant of \cref{CorollaryOfGeneralSard-intro}. In particular, it holds:
\begin{enumerate}[label=(\roman*)]
\item\label{i:RisultatoEndpointGenerale-intro_1} if $q>\beta_0$, the restriction $\End|_{\V}: \V \to \gr$ has the Sard property, namely
\begin{equation}
	\mu\bigg(\End\big(\mathrm{Crit}(\End|_{\V})\big)\bigg)=0.
\end{equation}
\item\label{i:RisultatoEndpointGenerale-intro_2} if $q=\infty$, the restriction $\End|_{\V}: \V \to \gr$ has the so-called strong Sard property, namely
\begin{equation}
\dim_{\mathcal{H}}\bigg(\End\big(\mathrm{Crit}(\End|_{\V})\big)\bigg)\leq m-1.
\end{equation} 
\end{enumerate}
In all previous results, $\End\big(\mathrm{Crit}(\End|_{\V})\big)$ can be replaced with the smaller set $\End\big(\mathrm{Crit}(\End)\cap \V\big)$.
\end{corollaryintro}
 We can apply \cref{RisultatoEndpointGenerale-intro} to find large sets on which the Sard property holds for the Endpoint map of Carnot groups. This is a more functional-analytic approach to the Sard conjecture that, unlike previous ones, does not resort to reduction to finite-dimensional cases.


An interesting case in which the assumptions of \cref{RisultatoEndpointGenerale-intro} can be effectively checked is given by sets of controls that are ``sufficiently regular''. We introduce the notation to state our result, see \cref{sec:Carnotrealanal} for additional details. Given $r>0$, and a closed interval $I\subset \R$, consider the set:
	\begin{equation}
	\mathcal{C}^\omega(I,\R^k;r):=\bigg\{u:I\to \R^k\,\bigg|\, 
	\textrm{$u$ is real-analytic with radius of convergence} > r \bigg\},
	\end{equation}
	endowed with the sup norm over the $r$-neighbourhood of $I$ in $\C$. More generally, let $\mathcal{C}^\omega(I,\R^k;r,\ell)$ be the set of piecewise analytic controls:
\begin{equation}
\mathcal{C}^\omega(I,\R^k;r,\ell) = \bigg\{u: I \to \R^k \,\bigg|\,  u|_{I_j} \in \mathcal{C}^\omega(I_j,\R^k;r),\,\text{ for all $j=1,\dots,\ell$} \bigg\},
\end{equation}
where $I_j = \inf I+ \left[\frac{(j-1)|I|}{\ell} ,\frac{j |I|}{\ell}\right]$, for all $j=1,\dots,\ell$, which we equip with a suitable norm (see \cref{sec:Carnotrealanal}). 
In \cref{WidthAnalitiche} we prove that the unit ball $K$ of $\mathcal{C}^\omega(I,\R^k;r,\ell)$, with $I=[0,1]$, is compact in $H$ and its $n$--width satisfies:
\begin{equation}
			\Omega_n
			( 
			K, H)
			\leq
			\frac{
				(k\ell)^{1/2} 
									}
				{  \ln r
									}
			\left(
			\frac{1}{r}
			 \right)^
			 {\floor*{\frac{n}{k\ell}}} , \qquad \forall\, n\in \N,\quad r>1.
\end{equation}
By applying \cref{RisultatoEndpointGenerale-intro} we can therefore deduce the following result (see \cref{thm:main4}).

\begin{theoremintro}[Sard property on piecewise real--analytic controls]\label{thm:main4intro}
Let $\gr$ be a Carnot group of topological dimension $m$, step $s$, and rank $k$. Let $I=[0,1]$. For $\ell\in\N$, there exists $r=r(m,s,k,\ell)>1$ such that
\begin{equation}
\mu\big(\End(\mathrm{Crit}(\End)\cap \mathcal{C}^\omega(I,\R^k;r,\ell))\big)= \mu\big(\End(\mathrm{Crit}(\End|_{\mathcal{C}^\omega(I,\R^k;r,\ell)})\big) =0.
\end{equation}
Namely, the Sard property holds on the space of piecewise real--analytic controls with radius of convergence $>r$ and with $\ell$ pieces. 
Furthermore it holds
\begin{equation}
\dim_{\mathcal{H}}\big(\End(\mathrm{Crit}(\End)\cap \mathcal{C}^\omega(I,\R^k;\infty,\ell))\big) \leq \dim_{\mathcal{H}}\big(\End(\mathrm{Crit}(\End|_{\mathcal{C}^\omega(I,\R^k;\infty,\ell)})\big) \leq m-1.
\end{equation}
Namely, the strong Sard property holds on the space of piecewise entire controls with $\ell$ pieces.
\end{theoremintro}
\begin{remark}[Stronger versions]
By inspecting the proof of \cref{thm:main4intro}, one actually obtains more general estimates for the Hausdorff dimension of the sets of critical values as in \cref{RisultatoEndpointGenerale-intro}. Furthermore, all these results hold in a stronger form, replacing $\mathcal{C}^\omega(I,\R^k;r,\ell)$ with the larger space	
\begin{equation}
\mathrm{span}\left( \overline{\{ u \in \mathcal{C}^\omega(I,\R^k;r,\ell) \mid \|u\|_{\mathcal{C}^\omega(I,\R^k;r,\ell)} \leq 1\}}\right),\qquad r \in (0,\infty],
\end{equation}
where $\|\cdot\|_{\mathcal{C}^\omega(I,\R^k;r,\ell)}$, defined in \eqref{eq:normComegaell}, is a sort of Hardy $\infty$-norm, while the closure is w.r.t.\ $L^2(I,\R^k)$.
\end{remark}
\begin{remark}[Regularity and Sard]
Sussmann established that for real--analytic sub-Riemannian structures (such as Carnot groups), minimizing horizontal paths are real--analytic on an open and dense set of their interval of definition \cite{Suss-realanal}. In light of that result, \cref{thm:main4intro} hints at an unexpected link between the regularity problem of sub-Riemannian geodesics and the minimizing Sard conjecture.
\end{remark}
\begin{remark}[Surjectivity of the Endpoint map]\label{rmk:surj}
Of course a subset of controls $\V\subset L^2(I,\R^k)$ must satisfy $\End(\V)=\gr$ to be relevant for the theory. By the Rashevskii-Chow theorem, $\End:L^2(I,\R^k)\to \R^m$ is surjective. We observe that $\End$ is surjective also on $\mathcal{C}^\omega(I,\R^k;r,\ell)$ if $\ell \geq 4m$, since from the proof of the Rashevskii-Chow theorem any pair of points can be joined by a concatenation of $4m$ horizontal paths with constant controls, see  \cite[Sec.\ 3.2]{ABB}. Actually, the Endpoint map is surjective when restricted on any set that is dense in $L^2(I,\R^k)$, see e.g.\ \cite[Sec.\ 2.5]{Bellaiche}. It follows that $\End$ is surjective on the sets of \cref{thm:main4intro} since in any case they contain all (restrictions on $I$ of) polynomial controls.
\end{remark}

Finally, we prove in \cref{thm:polysurj} that the Endpoint map on Carnot groups is surjective already when restricted to polynomial controls of large enough degree (depending on $\gr$). This is a particular case of the following more general result (see \cref{thm:contdensurj}).
\begin{propositionintro}[Quantitative surjectivity]\label{thm:contdensurj-intro}
Let $\gr$ be a Carnot group with topological dimension $m$ and rank $k$. Let $S \subset L^2(I,\R^k)$ be a dense set. Then there exist $u_0,u_1,\dots,u_m \in S$ such that
\begin{equation}
\End(\mathrm{span}\{u_0,u_1,\dots,u_m\}) = \gr.
\end{equation}
\end{propositionintro}
In other words, one can find an $(m+1)$-dimensional vector subspace $E\subset L^2(I,\R^k)$ of controls (depending on $\gr$) such that the restriction of the Endpoint map to $E$ is surjective. Furthermore, $E$ can be assumed to have generators in any prescribed dense set $S\subset L^2(I,\R^k)$. \cref{thm:contdensurj-intro} can be seen as a more refined version (for Carnot groups) of the well--known fact that the Endpoint map is surjective when restricted on any set of controls that is dense in $L^2(I,\R^k)$, see e.g.\ \cite[Sec.\ 2.5]{Bellaiche}.

\subsection{Closing thoughts on the Sard conjecture}
We do not venture in a guess in favor of Sard conjecture on the whole $H=L^2(I, \R^k)$, instead we propose a direction of future investigation that can be approached with our techniques. Recall that $H_{\mathrm{min}}\subset H$ is the set of energy-minimizing controls. It is known that $H_{\min}$ is boundedly compact in $H$ i.e., the intersection $H_{\min}\cap B_H(r)$ is compact for all $r>0$, see \cite[Thm.\ 8.66]{ABB}, \cite{AAA-Rendiconti}. In fact, from the very recent result by Lokutsievskiy and Zelikin \cite[Cor.\ 1 $(\mathfrak{G})$]{Lev}, one can deduce the following estimate for the $n$--width on Carnot groups
\begin{equation}\label{eq:mini}
\Omega_{2n+1}(H_{\min}\cap B_H(1), H)\leq C n^{-\frac{1}{2s}}, \qquad \forall\, n \in \N,
\end{equation}
for some $C>0$ depending on $\gr$. Estimate \eqref{eq:mini} can be understood as a quantitative compactness property for $H_{\min}$. Unfortunately, the polynomial decay (w.r.t.\ $n$) of \eqref{eq:mini} is too weak to apply \cref{CorollaryOfGeneralSard-intro}, since the latter requires an exponential one.

If one could prove that, for some $q^{-1} < \omega_0(s,m)$, it holds
\begin{equation}\label{eq:ifonecould}
\Omega_n(H_{\min}\cap B_H(1),H) \leq  q^{-n},\qquad \text{as } n\to\infty,
\end{equation}
then \cref{thm:mainintro} would settle the minimizing Sard conjecture on Carnot groups. 

We also note that it would be sufficient to prove \eqref{eq:ifonecould} for the set of the so-called \emph{strictly abnormal} energy-minimizing controls $H_{\min}^{\mathrm{str.abn}}\subsetneq H_{\min}$.  Unfortunately, we were not able to produce a direct estimate of the $n$--width of (bounded subsets of) $H_{\min}$ or $H_{\min}^{\mathrm{str.abn}}$, so that we have no further evidence to support this idea towards a proof of the minimizing Sard conjecture.


In the opposite direction, and to conclude, our results can be used to identify spaces where the (general) Sard conjecture can be violated. If, for a given Carnot group $\gr$, there is a linear subset of controls $\V\subset H$ such that $\mu(\End(\mathrm{Crit}(\End)\cap \V))>0$, then the following necessary condition must hold:
\begin{equation}
\limsup_{n\to \infty} \Omega_n(\V \cap B_H(r),H)^{1/n} \geq \omega_0(s,m).
\end{equation}

\subsection{Acknowledgements}

This project has received funding from (i) the European Research Council (ERC) under the European Union's Horizon 2020 research and innovation programme (grant agreement GEOSUB, No. 945655); (ii) the PRIN project ``Optimal transport: new challenges across analysis and geometry'' funded by the Italian Ministry of University and Research; (iii) the Knut and Alice Wallenberg Foundation. The authors also acknowledge the INdAM support. The authors also wish to thank Yosef Yomdin for helpful discussions that improved the presentation of the paper.


\section{Table of notations}

\begin{xltabular}{\textwidth}{p{0.19\textwidth} X}
$\R^m_+$ \dotfill & $m$-tuples of positive real numbers \\
$B_X(x,r)$ \dotfill &  closed ball of radius $r>0$ centered at $x$ of a Banach space $X$ \\
$B_X(r)$ \dotfill &  closed ball of radius $r>0$ centered at $0$ of a Banach space $X$ \\
$B_X$ \dotfill &  closed unit ball centered at $0$ of a Banach space $X$ \\
$\mathcal{L}(X,Y)$ \dotfill & linear continuous maps between banach spaces $X$, $Y$ \\
$\|\cdot\|_{\op}$ \dotfill &  operator norm on $\mathcal{L}(X,Y)$, see \eqref{eq:operatornorm} \\
$\|\cdot\|_X$ \dotfill &  norm for a Banach space $X$ (the $X$ is omitted when there is no confusion) \\
$\mathcal{B}(X,Y)$ \dotfill & continuous bounded maps between Banach spaces $X$, $Y$ \\
$\mathcal{C}^\omega(I,\R^k;r,\ell)$ \dotfill &  piece-wise real--analytic functions, \cref{sec:Carnotrealanal} \\
$\mu$ \dotfill & Lebesgue measure \\
$\Omega_n(K,H)$ \dotfill &  Kolmogorov $n$--width of $K\subset H$, \cref{def:nwidth} \\
$\omega(K,H)$ \dotfill &  asymptotic Kolmogorov width of $K\subset H$, \cref{eq:defomegaKH} \\
$\mathcal{P}_d^m(H)$ \dotfill & set of well-approximable maps, \cref{def:polyH} \\
$\beta_0(d,m)$ \dotfill & semialgebraic constant, \cref{rmk:beta0} \\
$\omega(d,m)$ \dotfill & threshold constant, \cref{thm:mainintro}\\
$\mathcal{H}^i(A)$ \dotfill & $i$-dimensional Hausdorff measure of a rectifiable Borel set $A$ \\
$V_i(A)$ \dotfill & Vitushkin variation of a bounded semialgebraic set $A$, \cref{defvariation} \\
$b_0(S)$ \dotfill & number of connected components of a set $S$ \\
$\sigma_k(L)$ \dotfill & $k$-th singular value of a linear map $L: H\to\R^m$ on a Hilbert space, ordered in decreasing order $\sigma_1(M)\geq \dots \geq \sigma_m$, \cref{def:singval} \\
$\crit(f)$ \dotfill & critical  points of a map, \cref{eq:nucritical}  \\
$\crit_\nu(f)$ \dotfill & $\nu$-critical  points of a map, \cref{eq:nucritical}  \\
$C^{\Lambda}(f)$ \dotfill & $\Lambda$-critical (or almost-critical) points of a map, \cref{def:lambdacrit} \\
$M(\varepsilon,S)$ \dotfill & $\varepsilon$-entropy of a set $S \subset \R^m$, \cref{def:entropy} \\
$\dim_e(S)$ \dotfill & entropy dimension of a set $S \subset \R^m$, \cref{def:entropy} \\
$\mathcal{U}_{\varepsilon}(S)$ \dotfill & $\varepsilon$-neighbourhood of a subset $S\subset \R^n$ \\
$\mathrm{dist}_H$ \dotfill & Hausdorff distance \\
$\beta_{\mathrm{tm}}$ \dotfill & Thom-Milnor bound constant, \cref{thm:milnorbound} \\
$\beta_{\mathrm{dc}}$ \dotfill & definable choice constant, \cref{thm:SemialgebraicSelection} \\
$\beta_{\mathrm{yc}}$ \dotfill & Yomdin-Comte constant, \cref{BehaviourVariations1} \\
$D(S)$ \dotfill & diagram of a semialgebraic set $S$, \cref{def:diagram} \\
$\gr$ \dotfill & Carnot group (of step $s$, rank $k$, dimension $m$), \cref{CarnotGroups}\\
$\End$ \dotfill & Endpoint map, \cref{def:endpoint} \\
\end{xltabular}

\section{Quantitative variations estimates}\label{sec:quantitativesemialgebraic}

\subsection{Semialgebraic sets and maps} 

The goal of this section is to prove \cref{thm:estimatevariations-intro}.
 We begin by recalling some basic notions from semialgebraic geometry.
\begin{definition}[Semialgebraic sets and maps]\label{def:semialgebraic}
We say that a set $S\subset \R^n$ is \emph{semialgebraic} if
\begin{equation}\label{eq:defA} 
S=\bigcup_{i=1}^a\bigcap_{j=1}^{b_i}\left\{x\in \R^n\mid \mathrm{sign}(p_{ij}(x))=\sigma_{ij}\right\},
\end{equation}
for some finite set of polynomials  $p_{ij}\in \R[x_1, \ldots, x_n]$ and $\sigma_{ij}\in \{0, +1, -1\}$, where
\begin{equation}
\mathrm{sign}(r):=\begin{cases}+1&r>0,\\
-1&r<0,\\
0&r=0.\end{cases}
\end{equation}
We call the description \eqref{eq:defA} a \emph{representation} of $S$. 
 
A map $f:A\to B$ between semialgebraic sets $A\subset \R^n, B\subset \R^\ell$ is said to be \emph{semialgebraic} if its graph is a semialgebraic set in $\R^n\times \R^\ell$.
\end{definition}

The representation of a semialgebraic set $S$ as in \eqref{eq:defA} is not unique. However, a representation is useful to quantify the complexity of $S$, using the following notion.
\begin{definition}[Diagram of a semialgebraic set]\label{def:diagram}
Let $S\subset \R^n$ be a semialgebraic set represented as in \eqref{eq:defA}. We will say that the triple
\begin{equation}\label{eq:diagram}
\left(n, a \cdot \max_{i}\{b_i\}, \max_{i,j}\{\deg(p_{ij})\}\right)\in \N^3
\end{equation}
is a \emph{diagram} for $S$. Below, the equation ``$D(S)=(m, c, d)$'' will mean that there exists a representation of $S$ as in \eqref{eq:defA} with $n\leq m$, $ a\cdot\max_i\{{b_i}\}\leq c$ and $\max_{i,j}\{\deg(p_{ij})\}\leq d$.
\end{definition}
\begin{remark}
There are several  ways of quantifying the complexity of a semialgebraic set, depending on the way it is presented. Essentially, these notions should contain three pieces of information: the number of variables, the ``combinatorics'' of the presentation, and the degrees of the defining polynomials. For instance, every semialgebraic set $S\subset \R^n$ can be presented as
\begin{equation}S=\bigcup_{i=1}^{a}\bigcap_{j=1}^{b_i}E_{ij},
\end{equation}
where each $E_{ij}$ has the form $\{ x \in \R^n \mid p_{ij}(x) \leq 0\}$ or $\{ x \in \R^n \mid p_{ij}(x) < 0\}$,
where $p_{ij}$ is a polynomial with $\deg(p_{ij})\leq d$ for all $i,j$.
Given a presentation of this type, it is immediate to see that $S$ also admits a diagram $D(S)=(n, c, d)$ (i.e. it can be presented as in \eqref{eq:defA} with $c= c(a,b_1,\dots,b_a)$. 
\end{remark}


\subsubsection{Dimension and stratifications}\label{item1useful}
 Every semialgebraic subset of $\R^n$ can be written as a finite union of semialgebraic sets, each of them semialgebraically homeomorphic to an open cube $(0,1)^m\subset \R^m$, for some $m \leq n$ (\cite[Thm.\ 2.3.6]{BCR}). This allows to define the \emph{dimension} of a semialgebraic set as the maximum of the dimensions $m$ of these cubes. Every semialgebraic set can be written as a finite union of smooth, semialgebraic disjoint manifolds called \emph{strata} (\cite[Prop.\ 9.1.8]{BCR}). 

Notice that the dimension of a semialgebraic set is preserved by semialgebraic homeomorphisms and behaves naturally under product structures: 
\begin{equation}\label{eq:dimension}
\dim(A\times B)=\dim(A)+\dim(B).
\end{equation}
Moreover, if $f:A\to B$ is a continuous semialgebraic map, then
\begin{equation}\label{eq:dimension2}
\dim(f(A))\leq \dim(A).
\end{equation}

\subsubsection{Semialgebraic triviality}\label{item2useful}
Continuous semialgebraic maps $f:A\to B$ are ``piecewise'' trivial fibrations: there exists a partition of $B$ into finitely many semialgebraic sets
\begin{equation}
B=\bigsqcup_{j=1}^b B_j
\end{equation}
and, for every $j=1, \ldots, b$ there exist fibers $F_j:=f^{-1}(b_j)$, with $b_j\in B_j$ and a semialgebraic homeomorphism $\psi_j:f^{-1}(B_j)\to B_j\times F_j$ that make the following diagram commutative: 
\begin{equation}
\begin{tikzcd}
B_j\times F_j \arrow[rr, "\psi_j"] \arrow[rd, "p_1"'] &     & f^{-1}(B_j) \arrow[ld, "f"] \\
                                                      & B_j &                            
\end{tikzcd}
\end{equation}
This result is called \emph{semialgebraic triviality}, see \cite[Thm.\ 5.46]{BasuPoRoyBook}.

\subsubsection{Projections}\label{item:projectionsaugata}
The image of a semialgebraic set $S\subset \R^n\times \R^\ell$ under the projection map $\pi:\R^n\times \R^\ell\to \R^n$ is semialgebraic (\cite[Thm. 2.2.1]{BCR}). Moreover, it follows from \cite[Thm.\ 14.16, Notation 8.2]{BasuPoRoyBook}\footnote{Note that the role of $k$ and $\ell$ for us is swapped with respect to \cite[Thm.\ 14.16]{BasuPoRoyBook}: we eliminate $\ell$ variables from a list of $n+\ell$, they eliminate $k$ variables from a list of $k+\ell$.} that
 there exists $C_1>0$ such that  if $(n+\ell, c, d)$ is a diagram for $S$,  we can write
 \begin{equation}
\pi(S)=\bigcup_{i\in I}\bigcap_{j\in J_i}\bigcup_{k\in N_{ij}}\left\{x\in \R^n \mid \mathrm{sign}(p_{ijk}(x))=\sigma_{ijk}\right\},
\end{equation}
with 
\begin{equation}\label{eq:boundindexsets}
\#I\leq c^{(n+1)(\ell+1)}d^{C_1\ell n}, \quad\#J_i\leq c^{\ell+1}d^{C_1\ell},\quad \#N_{ij}\leq d^{C_1\ell},
\end{equation} 
and with
\begin{equation}\label{eq:bounddegree}
\deg(p_{ijk})\leq d^{C_1\ell}.
\end{equation}
In particular, there exists $C_2>0$ such $\pi(S)$ admits a diagram with:
\begin{equation}\label{eq:bounddiagramproj}
D(\pi(S))=\left(n, (cd)^{C_2\ell n}(d^{C_2\ell})^{(cd)^{C_2\ell}}, d^{C_2\ell}\right).
\end{equation}
In particular, the diagram of the projection of a semialgebraic set is controlled explicitly by the diagram of the original set.

\subsubsection{Thom-Milnor bound}\label{item:milnorbound}

We will need the following qualitative bound on the number of connected components of a semialgebraic set, that follows from \cite[Thm.\ 7.50]{BasuPoRoyBook}.

\begin{theorem}[Thom-Milnor]\label{thm:milnorbound}
Let $c,d\in \N$. There exists a constant $\beta_{\mathrm{tm}}=\beta_{\mathrm{tm}}(c,d)>1$, such that for any semialgebraic set $S\subset \R^n$ with diagram $D(S)=(n, c, d)$ it holds
\begin{equation}\label{eq:milnorbound}
b_0(S) \leq \beta_{\mathrm{tm}}^n,
\end{equation}
where $b_0(S)$ denotes the number of connected components of $S$.
\end{theorem}

\subsection{Definable choice}

We recall from \cite{LRT-selection} the following \emph{quantitative approximate definable choice} result, which will be a key tool in the proof of \cref{thm:estimatevariations-intro}. 

Let us explain the name and the context of this result. Given a polynomial map $F:\R^n\to \R^\ell$ and a compact semialgebraic set $A\subset \R^n$, for every $y\in F(A)$ one would like to choose a point $x(y)\in F^{-1}(y)\cap A$ so that the resulting set $C:=\{x(y)\}_{y\in F(A)}$ is semialgebraic and \emph{with controlled dimension}. In particular, the ``choice'' of such $x(y)$ has to be ``definable''. In this way $C$ would be of dimension $\ell$. The fact that this can be done  follows from semialgebraic triviality (see \cref{item2useful}). However, in this way there is no control on the geometry of $C$  (in particular, the number of its connected components) in terms of the data defining $A$ (namely, its diagram, see \cref{def:diagram}).

For many geometric applications (as in this paper) one does not really need that $C$ is a choice over $A$, but it is enough that it is close to it. More precisely, given $\varepsilon>0$, denote by $\mathcal{U}_\varepsilon(A)$ the Euclidean $\varepsilon$--neighbourhood of $A$. Then one can relax the requirements for the definable choice and ask, given $\varepsilon>0$, for a set 	$C_\varepsilon\subset \mathcal{U}_\varepsilon(A)$, of dimension $\leq \ell$, whose image under $F$ is $\varepsilon$-close to $F(A)$ (in the Hausdorff metric, denoted by $\mathrm{dist}_H$), and with a control the geometry of $C_\varepsilon$.

\begin{theorem}[Quantitative approximate definable choice \cite{LRT-selection}]\label{thm:SemialgebraicSelection}
	For every $ c, d, \ell \in \N$ there exists  $\beta_{\mathrm{dc}}>1$ satisfying the following statement. Let $n\in \N$ and let $A\subset \R^n$ be a closed semialgebraic set contained in the ball $B_{\R^n}(r)$ with diagram
	\begin{equation}
	D(A)=(n,c, d).
	\end{equation}
	Let also $F=(F_1, \ldots, F_\ell):\R^n\to \R^\ell$ be a polynomial map with components of degree bounded by $d$. Then for every $\varepsilon \in (0 , r)$ there exists a closed semialgebraic set $C_\varepsilon\subset \R^{n}$ such that
	\begin{enumerate}[(i)]
		\item \label{item:thm1} $\dim(C_\varepsilon)\leq \ell$;
		\item  \label{item:thm2} $C_\varepsilon \subset \mathcal{U}_\varepsilon(A)$;
		\item  \label{item:thm3} $\mathrm{dist}_{H}(F(C_\varepsilon), F(A))\leq L(F, r) \cdot \varepsilon$, where
	$L(F, r):=2+\mathrm{Lip}(F, B_{\R^n}(2 r))$;
		\item  \label{item:thm4} for every $e=1, \ldots, n$ and every affine space $\R^e\simeq E\subset \R^n$, the number of connected components of $E\cap C_\varepsilon$ is bounded by

		\begin{equation}b_0(E\cap C_{\varepsilon})\leq \beta_{\mathrm{dc}}^e.\end{equation}
	\end{enumerate}
\end{theorem}

\subsection{Variations and their behaviour under polynomial maps}
We recall the definition of \textit{variations}, introduced by Vitushkin  \cite{Vit1, Vitbook} and developed in the semialgebraic context by Comte and Yomdin \cite{ComteYomdin}.
The variations encode the size of a set, through an integral-geometric approach.
Given a polynomial map $p:\R^n \to \R^m$ and a semialgebraic set $A \subset \R^n$, it is possible to estimate the variations of the set $p(A)$, in a quantitative way, see
\cite[Thm.\ 7.2]{ComteYomdin}.
However, these estimates are not quantitative with respect to the dimensional parameter $n$.
The main result of this section is \cref{BehaviourVariations1}, where we prove a version of \cite[Thm.\ 7.2]{ComteYomdin}, where the dependence on $n$ is explicit, thanks to the use of \cref{thm:SemialgebraicSelection}.


For $i\leq n$, we denote by $G_i(\R^n )$ the Grassmannian of all $i$-dimensional linear subspaces of $\R^n$. 
We endow it with the standard probability measure $\gamma_{ i , n}$ defined through the action of the orthogonal group of $\R^n$. For any $Z \in G_i(\R^n )$ we the denote the orthogonal projection on $Z$ by 
\begin{equation}
\pi_Z : \R^n \to Z.
\end{equation}

\begin{definition}[Vitushkin variations]\label{defvariation}
Let $A\subset \R^n$ be a bounded semialgebraic set. We define the $0$--th \emph{variation} of $A$ as
\begin{equation}
V_0(A):=b_0(A),
\end{equation}
where $b_0(A)$ denotes the number of connected components of $A$. For $i=1, \ldots, n$ we define the $i$--th \emph{variation} of $A$ as
\begin{equation}\label{eq:defvariation}
	V_i(A) 
	:=
	c(n,i)
	\int_{G_i (\R^n)} 
	\left(
		\int_Z
		b_0 ( A \cap \pi_Z^{-1}(x) )
		\,\mathcal{H}^i(dx)
				\right)
	\,\gamma_{i,n}(dZ) ,
\end{equation}
where $\mathcal{H}^i$ denotes the $i$-dimensional Hausdorff measure.
\begin{equation}\label{exactvalueconstant}
	c(n,i)
	:=
	\frac{\Gamma(\frac{1}{2})\Gamma(\frac{n + 1}{2})}{\Gamma(\frac{i+1}{2})
		\Gamma(\frac{n-i+1}{2})}.
\end{equation}
We note that the integrand in \eqref{eq:defvariation} is measurable, see \cite[Sec.\ 3]{ComteYomdin}.
\end{definition} 

A list of properties of variations is given in  \cite[page 35]{ComteYomdin}. Here we prove the following one.

\begin{lemma}\label{VolumeAndVariationsSemialg}
	If $A \subset \R^n$ is a bounded semialgebraic set of dimension $\ell$, then the $\ell$-variation coincides with the $\ell$-dimensional Hausdorff volume:
	\begin{equation}
		V_\ell (A)= 	\mathcal{H}^\ell (A)  .
	\end{equation}
We also have $V_i(A ) = 0  \text{ for all  } i > \ell$.
\end{lemma}
\begin{proof}
We recall a classical result in Integral Geometry, called Cauchy--Crofton formula \cite[Thm.\ 5.14]{FedererTransAmerMathSoc}: if $A$ is a $i$-rectifiable and Borel subset of $\R^n$, for $i \leq n$, then it holds
\begin{equation}\label{eq:federer}
	\mathcal{H}^i (A)
	=
	c(n,i)
	\int_{G_i (\R^n)} 
	\left(
	\int_Z
	 \# ( A \cap \pi_Z^{-1}(x) )
	\,\mathcal{H}^i(dx)
	\right)
	\,\gamma_{i,n}(dZ).
\end{equation}
If $A\subset \R^n$ is semialgebraic, the set $A$ can be written as a finite
union of smooth, semialgebraic disjoint manifolds called strata (see \cref{item1useful}). We can relabel the strata  so that
	 \begin{equation}
	 A = \bigsqcup_{j=0}^\ell A_j,
	 \end{equation}
	  where each $A_j$ is a finite union of smooth $j$-dimensional manifolds.
	For any fixed $Z \in G_\ell (\R^N)$ and for almost every $x \in Z$
	the set $\pi_Z^{-1}(x)$ does not intersect\footnote{We note that in this step of the proof it is enough to use the Semialgebraic Sard's theorem \cite[Thm.\ 9.6.2]{BCR}, which is logically independent from the result in the smooth category.} $A_j $ for $j <\ell$.
	Hence, 
	$V_0 (A \cap \pi_Z^{-1}(x)) = 
	V_0 (A_\ell \cap \pi_Z^{-1}(x))$ 
	for almost every $ x \in Z$.
	Furthermore, for almost every $x \in Z$ the intersection
	$A_\ell \cap \pi_Z^{-1}(x)$ is transverse, and $\dim A + \dim \pi_Z^{-1}(x) =n$, so that $A_\ell \cap \pi_Z^{-1}(x)$ has dimension zero. Thus, its cardinality coincides with its number of connected components. We conclude by \eqref{eq:federer}.
\end{proof}

Let $M$ be a real $m \times m$ symmetric and positive semidefinite matrix. We denote by 
\begin{equation}
\lambda_1(M) \geq \dots \geq \lambda_m(M)\geq 0,
\end{equation}
 its eigenvalues, ordered in decreasing order. 
\begin{definition}[Singular values]\label{def:singval}
	Let $H$ be a Hilbert space with
	$\dim H \geq m$, and let
	$L:H \to \mathbb{R}^m$ 
	be a linear and continuous map. For any $k=1,\dots,m$ the $k$-th \emph{singular value} of $L$ is
	\begin{equation}
		\sigma_k(L)
		= 
		\lambda_k(LL^{\top})^{\frac{1}{2}}  .
	\end{equation}
\end{definition}
In finite dimension, the following Weyl inequality holds  for the singular values of the difference of two matrices (see \cite[Ex.\ 1.3.22 (iv)]{TaoRMT2012}): for $n \geq m$ and linear maps $L_1, L_2: \mathbb{R}^n \to \mathbb{R}^m$ we have
\begin{equation}\label{WeylFiniteDim}
	|\sigma_k( L_1 ) - \sigma_k( L_2 )| 
	\leq   
	\| L_1 - L_2 \|_{\op} ,\qquad \forall\, k=1, \dots , m.
\end{equation}
The Weyl inequality generalizes to infinite-dimensional spaces. Since we are not able to find a reference, we provide below a self-contained statement and its proof in the form that we will need later.
\begin{lemma} \label{SingolarValuesLipschitz}
	If $L_1, L_2 \in \mathcal{L}(H , \mathbb{R}^m)$, then 
	$|\sigma_k(L_1) - \sigma_k(L_2)| \leq \|L_1-L_2\|_{\op}$ for all 
	$k= 1 , \dots , m$.
\end{lemma}
\begin{proof}
	Let $L:H \to \R^m$ be linear and continuous. Let $W\subset H$ a finite dimensional subspace such that $(\ker L)^\perp \subseteq W$, and denote by $L_W: W \to \mathbb{R}^m$ the restriction of $L$ to $W$. More precisely, $L_W=L \circ i_W$ where $i_W : W\hookrightarrow H$ is the inclusion.
	We now prove that for all $k \leq m$
	\begin{equation}\label{InfDimWeyl_1}
		\sigma_k(L) = \sigma_k(L_W)  .
	\end{equation}
	Indeed, since the adjoint $i_W^\top $ coincides with the orthogonal projection 
	$ \pi_W: H \to W \subset H$, we have
	\begin{equation} 
	\lambda_k(L_W \circ L_W^\top) 
	=
	\lambda_k(L \circ i_W \circ i_W^\top \circ L^\top)  
	=
	\lambda_k(L \circ \pi_W \circ L^\top)  = \lambda_k(L\circ L^\top).
	\end{equation}
In the last equality, we used the fact that
	$L \circ \pi_W  = L$, 
	indeed
	$L(v) 
	= 
	L(\pi_W v + \pi_{W^\perp} v)
	=
	L(\pi_W v)
	$
	since $W^\perp \subseteq \ker L$.
	Now given $L_1,L_2 \in \mathcal{L}(H,\R^m)$ let $W:=(\ker L_1)^\perp + (\ker L_2) ^\perp$, which is a finite dimensional linear subspace of $H$.
	By \eqref{InfDimWeyl_1} and the finite dimensional Weyl inequality, we have 
	\begin{equation}
		|\sigma_k(L_1) - \sigma_k(L_2)|
		=
		|\sigma_k((L_1)_W) - \sigma_k((L_2)_W)|
		\leq 
		\|(L_1)_W - (L_2)_W\|_{\op}
		\leq
		\|L_1-L_2\|_{\op},
	\end{equation}
	concluding the proof.
\end{proof}

In the next result, given a semialgebraic set $A\subset \R^n$ and a polynomial map $p:\R^n\to \R^m$, we bound the $i$-th variation $V_i(p(A))$ in terms of the singular values of $p$ and the diagram of $A$. Its proof follows the blueprint of \cite[Thm.\ 7.2]{ComteYomdin}. The main novelty (which is key for our work) is that the dependence on the dimension $n$ is explicit.

	\begin{theorem} \label{BehaviourVariations1}
	For every $c,d,m\in \N$ there exists 
	$\beta_{\mathrm{yc}}=\beta_{\mathrm{yc}}(c,d,m) > 1$ with the following property.
	For $n\geq m$, let $A \subset \R^n$ be a closed and bounded semialgebraic set with $A\subset B_{\R^n}(r)$ and diagram $D(A)=(n,c,d)$. Let
	$p: \R^n \to \R^m$ be a polynomial map with components of degrees at most $d$. For all $i = 1 , \dots , m$ set
	\begin{equation}\label{eq:prodsigmaass}
	\overline{\sigma}_i:=\sup_{x \in A} \Big({\sigma}_1(D_xp)\cdots{\sigma}_i(D_xp) \Big), \qquad \overline{\sigma}_0:=1.
	\end{equation}
	Then, for all $i=0,\dots,m$ it holds
	\begin{equation}
		V_i(p(A))
		\leq  \mathrm{cst}(m,r) n^m \beta_{\mathrm{yc}}^n  \overline{\sigma}_i
	\end{equation}
	where $\mathrm{cst}(m,r)>0$ is a constant that depends only on $m,r$.
	\end{theorem}

	\begin{remark}[Origin of $\beta_{\mathrm{yc}}$]\label{rmk:origin}
An inspection of the proof shows that it holds $\beta_{\mathrm{yc}}= \beta_{\mathrm{tm}}^2 \beta_{\mathrm{dc}}$ where $\beta_{\mathrm{tm}} = \beta_{\mathrm{tm}}(c(m+1),d)$ is the constant from the Thom-Milnor-type bound of \cref{thm:milnorbound}, and $\beta_{\mathrm{dc}} = \max_{i\leq m} \beta_{\mathrm{dc}}(c,d,i)$ is the definable choice constant from \cref{thm:SemialgebraicSelection}.
	\end{remark}
	\begin{remark}\label{rmk:comteyomdinimproved}
In the course of the proof, see \eqref{eq:estVipA}, we obtain the following estimate, reminiscent of an area formula for variations, which is an improved version of \cite[Thm.\ 7.2]{ComteYomdin}, with the explicit dependence of the constants on $n$: there exist a constant $\beta_{\mathrm{yc}} = \beta_{\mathrm{yc}}(c,d,m)$, such that
\begin{equation}
	V_i(p(A)) 
	\leq 
	c(m,i) \beta_{\mathrm{yc}}^n \overline{\sigma}_i
	V_i(A), \qquad i=1,\dots,m,
\end{equation}
where $c(m,i)$ is the constant in \eqref{exactvalueconstant}.
	\end{remark}

\begin{proof}
By \cref{defvariation}, $V_0 = b_0$. Furthermore, $b_0(p(A))\leq b_0(A)$. Thus, \cref{thm:milnorbound} yields the thesis of the theorem for the case $i=0$. By \cref{defvariation}, for $i=1,\dots,m$, we have
	\begin{equation}
		V_i(p(A))
		=
		c(m,i)
		\int_{G_i(\R^m)}\left(
		\int_{Z}
		b_0
		( p(A) \cap \pi^{-1}_Z (x) )
		\,	\mathcal{H}^i(dx) \right) \gamma_{i,m}\left(dZ\right)   .
	\end{equation}
	For any $x\in Z$, the set $A\cap (\pi_Z\circ p)^{-1}(x)$ is semialgebraic with controlled diagram (more precisely, $D(A\cap (\pi_Z\circ p)^{-1}(x)) = (n,c(1+m),d)$. Using \cref{thm:milnorbound} we obtain
	\begin{equation}\label{eq:b0betatm}
			b_0
		( p(A) \cap \pi^{-1}_Z (x) )
		\leq
		b_0
		( A \cap (\pi_Z\circ p)^{-1}(x) ))
		\leq \beta_{\mathrm{tm}}^n,
	\end{equation}
	where $\beta_{\mathrm{tm}}=\beta_{\mathrm{tm}}(c(1+m),d)$ depends only on $c,d,m$.
	 Denoting by 
	$\mathbbm{1}_{S}$ the characteristic function of a set $S$,
	we get
	\begin{align}
		V_i(p(A)) 
		& \leq 
		c(m,i)
		\beta_{\mathrm{tm}}^n
		\int_{G_i(\R^m)}
		\left(\int_Z
		\mathbbm{1}_{ \pi_Z \circ p  (A) } (x)
		\, \mathcal{H}^i(dx)\right)
		\, \gamma_{i,m}(dZ) \\
		& =
		c(m,i)
		\beta_{\mathrm{tm}}^n
		\int_{G_i(\R^m)}
		\mathcal{H}^i (\pi_Z \circ p (A))
		\, \gamma_{i,m}(dZ).\label{IGAF_0}
	\end{align}
We would like to estimate $\mathcal{H}^i(\pi_Z \circ p (A) ) $ applying area formula to the map 
	\begin{equation}
	p_Z:=\pi_Z \circ p:\R^n\to \R^i.
	\end{equation}
	 Exactly as in \cite[Thm.\ 7.2]{ComteYomdin}, we cannot apply directly the area formula to  $A$, since the dimension of $A$ can be larger than $i$.
	In \cite[Thm.\ 7.2]{ComteYomdin} this problem is solved using \cite[Ex.\ 4.11]{ComteYomdin}. However, the dependence on $n$ of the constants appearing there is not explicit. It could in be principle be made so, using the Tarski-Seidenberg theorem, but the dependence obtained in this way would be doubly exponential in $n$, which is not enough for our purposes. We overcome the obstacle using \cref{thm:SemialgebraicSelection}.

For sufficiently small $\varepsilon>0$ let $C_{\varepsilon}$ the semialgebraic set obtained by applying \cref{thm:SemialgebraicSelection} to the set $A$ and the polynomial map $F= p_Z : \R^n \to \R^i$. We then proceed using the area formula to estimate $\mathcal{H}^i(\pi_Z \circ p(C_{\varepsilon}))$. (We will see at the end of the proof how $\mathcal{H}^i(\pi_Z \circ p(C_{\varepsilon})) \to \mathcal{H}^i(\pi_Z \circ p(A))$ as $\varepsilon\to 0$). By \cref{item:thm1} of \cref{thm:SemialgebraicSelection}, $\dim(C_\varepsilon)\leq i$. We can assume without loss of generality that $C_\varepsilon\subset \R^n$ is a smooth embedded submanifold with dimension equal to $i$, since the other strata give zero contribution to $\mathcal{H}^i$. Denoting by $\bar{p}_Z : C_\varepsilon \to \R^i$ the restriction, we obtain
\begin{align}
\mathcal{H}^i(p_Z(C_\varepsilon)) & \leq \int_{C_\varepsilon} |\det(D_x \bar{p}_Z )|\mathcal{H}^i(dx) & \text{by the area formula }\\
& \leq \int_{C_\varepsilon} \sigma_1(D_x \bar{p}_Z )\cdots  \sigma_i(D_x \bar{p}_Z )\mathcal{H}^i(dx). & \text{by definition of sing.\ values}
\end{align}
Note that, denoting by $j: C_\varepsilon \hookrightarrow \R^n$ the inclusion, we have
\begin{equation}
D_x \bar{p}_Z: T_x C_\varepsilon \xrightarrow{D_x j} \R^n \xrightarrow{D_x p} \R^m \xrightarrow{D_{p(x)} \pi_Z} \R^i.
\end{equation}
Thus, since $D_{p(x)} \pi_Z$ is an orthogonal projection and $D_x j$ is a linear inclusion, we have
\begin{equation}
\langle z , (D_x \bar{p}_Z) (D_x \bar{p}_Z)^\top z \rangle \leq \langle z ,(D_x p)(D_x p)^\top z \rangle, \qquad \forall z \in Z = \R^i.
\end{equation}
It follows that $\sigma_k(D_x \bar{p}_Z) \leq \sigma_k(D_x p)$ for all $k=1,\dots,i$. Continuing the above inequalities we obtain
\begin{align}
\mathcal{H}^i(p_Z(C_\varepsilon)) & \leq \int_{C_\varepsilon} \sigma_1(D_x p )\cdots  \sigma_i(D_x p )\mathcal{H}^i(dx) & \\
& \leq \sup_{x\in \mathcal{U}_\varepsilon(A)} \sigma_1(D_x p )\cdots  \sigma_i(D_x p ) \mathcal{H}^i(C_\varepsilon) & \text{by \cref{item:thm2} of \cref{thm:SemialgebraicSelection}}\\
& \leq \left(\overline{\sigma}_i + \eta(\varepsilon)\right)\mathcal{H}^i(C_\varepsilon), & \text{by continuity of sing.\ values}
\end{align}
where $\overline{\sigma}_i$ is defined in \eqref{eq:prodsigmaass}, and $\eta(\varepsilon) \to 0$ as $\varepsilon \to 0$. Since $C_\varepsilon$ is a bounded semialgebraic set of dimension $i$, then $\mathcal{H}^i(C_\varepsilon) = V_i(C_\varepsilon)$ by \cref{VolumeAndVariationsSemialg}. Using \cref{item:thm4,item:thm2} of \cref{thm:SemialgebraicSelection}, we obtain
\begin{align}
\mathcal{H}^i(C_\varepsilon) & = c(n,i)
		\int_{G_i(\R^n)}\left(
		\int_{Z}
		b_0
		( C_\varepsilon \cap \pi^{-1}_Z (x) )
		\,	\mathcal{H}^i(dx) \right) \gamma_{i,n}\left(dZ\right)  \\
		& \leq c(n,i) \beta_{\mathrm{dc}}^{n-i} \int_{G_i(\R^n)}\left(
		\int_{Z}
		\mathbbm{1}_{\pi_Z(C_\varepsilon)}(x)
		\,	\mathcal{H}^i(dx) \right) \gamma_{i,n}\left(dZ\right) 
		\\
		&  \leq c(n,i) \beta_{\mathrm{dc}}^{n-i} \int_{G_i(\R^n)}\left(
		\int_{Z}
		\mathbbm{1}_{\pi_Z(\mathcal{U}_\varepsilon(A))}(x)
		\,	\mathcal{H}^i(dx) \right) \gamma_{i,n}\left(dZ\right) 
		\\
		& \leq c(n,i) \beta_{\mathrm{dc}}^{n-i} \int_{G_i(\R^n)}\left(
		\int_{Z}
		b_0(\mathcal{U}_\varepsilon(A)\cap \pi_Z^{-1}(x))
		\,	\mathcal{H}^i(dx) \right) \gamma_{i,n}\left(dZ\right) \\
		& = \beta_{\mathrm{dc}}^{n-i} V_i(\mathcal{U}_\varepsilon(A)),
\end{align}
where $\beta_{\mathrm{dc}} = \beta_{\mathrm{dc}}(c,d,i)>1$ is the constant from \cref{thm:SemialgebraicSelection}, which (up to taking the maximum for $i\leq m$) depends only on $c,d,m$. Summing up, we have proved that for all $i=1,\dots,m$ it holds
\begin{equation}\label{eq:limit}
\mathcal{H}^i(p_Z(C_\varepsilon)) \leq (\overline{\sigma}_i + \eta(\varepsilon)) \beta_{\mathrm{dc}}^{n} V_i(\mathcal{U}_\varepsilon(A)).
\end{equation}
In order to take the limit in \eqref{eq:limit}, we recall two properties of semialgebraic sets. The first one is the continuity of the Lebesgue measure in the Hausdorff topology in the semialgebraic category (see \cite[Thm.\ 5.10]{ComteYomdin}): if $\{S_\varepsilon\}_{\varepsilon\geq 0}\subset \R^i$ is a one-parameter family of bounded semialgebraic sets then
\begin{equation}\label{eq:thm510}
\lim_{\varepsilon \to 0}\mathrm{dist}_H(S_\varepsilon,S_0) =0 \qquad \Longrightarrow \qquad \lim_{\varepsilon \to 0}\mathcal{H}^i(S_\varepsilon) = \mathcal{H}^i(S_0).
\end{equation}
The second property is the behaviour of variations for $\varepsilon$-neighbourhoods (see \cite[Thm.\ 5.11]{ComteYomdin}): if $A\subset \R^n$ is a bounded semialgebraic set, then for all $i=1,\dots,n$ it holds
\begin{equation}\label{eq:thm511}
\lim_{\varepsilon \to 0} V_i(\mathcal{U}_\varepsilon(A)) \leq V_i(A).
\end{equation}
Therefore, using \eqref{eq:thm510} (and recalling that by \cref{item:thm3} of \cref{thm:SemialgebraicSelection} it holds $p_Z(C_\varepsilon) \to p_Z(A)$ in the Hausdorff topology) and \eqref{eq:thm511} we can pass to the limit in \eqref{eq:limit} and obtain
\begin{equation}
\mathcal{H}^i(p_Z(A)) \leq \overline{\sigma}_i \beta_{\mathrm{dc}}^{n} V_i(A).
\end{equation}
Hence from \eqref{IGAF_0} we get for all $i=1,\dots,m$
\begin{equation}\label{eq:estVipA}
	V_i(p(A)) 
	\leq 
	c(m,i) \beta_{\mathrm{tm}}^n\beta_{\mathrm{dc}}^n \overline{\sigma}_i
	V_i(A) .
\end{equation}
To conclude, we estimate $V_i(A)$. Similarly as in \eqref{IGAF_0}, since $A\subset B_{\R^n}(r)$ we have
\begin{align}
V_i(A) \leq c(n,i) \beta_{\mathrm{tm}}^n \int_{G_i(\R^n)} \mathcal{H}^i(\pi_Z(A)) \gamma_{i,n}(dZ) \leq c(n,i) \beta_{\mathrm{tm}}^n \mathcal{H}^i(B_{\R^i}(r)) = c(n,i) \beta_{\mathrm{tm}}^n \frac{\pi^{i/2}}{\Gamma\left(\frac{i}{2}+1\right)}  r^i.\label{eq:estViA}
\end{align}
Putting together \eqref{eq:estVipA} and \eqref{eq:estViA}, we obtain
\begin{equation}
V_i(p(A)) 
	\leq \left[ c(m,i)c(n,i) \frac{\pi^{i/2}}{\Gamma\left(\frac{i}{2}+1\right)} \right] (\beta_{\mathrm{tm}}^2\beta_{\mathrm{dc}})^n \overline{\sigma}_i r^i.
\end{equation}
Set $\beta_{\mathrm{yc}} = \beta_{\mathrm{tm}}^2\beta_{\mathrm{dc}}$. Using the form of the constants $c(m,i), c(n,i)$ in \eqref{exactvalueconstant} and elementary estimates we obtain that for all $i=0,\dots,m$ there exists a constant $\mathrm{cst}(i)>0$ such that
\begin{equation}
V_i(p(A)) \leq  \mathrm{cst}(i) n^i \beta_{\mathrm{yc}}^n  \overline{\sigma}_i r^i.
\end{equation}
Taking the maximum over $i=1,\dots,m$ we obtain the result with $\mathrm{cst}(m,r) = \displaystyle \max_{i=0,\dots,m} \mathrm{cst}(i) r^i$.
\end{proof}

	\subsection{Variations of almost-critical values of polynomial maps}	\label{sec:variations}
	In this section we recall the notion of almost-critical points $C^{\Lambda} (p)$ of 
	a polynomial map $p: \R^n \to \R^m$. 
	The set $p( C^{\Lambda} (p) )$ is the set of its
	almost-critical values.
	The main result of this section is \cref{thm:estimatevariations}, which provides
	a version of 
	\cite[Cor.\ 7.4]{ComteYomdin} where the dependence of the parameters with respect to $n$ is explicit.
	This result is based on \cref{BehaviourVariations1} and on the specific semialgebraic structure of  almost-critical points of polynomial maps, that we prove in \cref{DiagramLambdaCrit}.
	\begin{definition}\label{def:lambdacrit}
		Let $f:\R^n \to \R^m$ be a $C^1$ map with $n \geq m$. 
		Given $\Lambda=(\Lambda_1 ,\dots, \Lambda_m) \in \R^m_+$, the set of \emph{almost-critical points} of $f$ is defined as
		\begin{equation}\label{NearCriticalPoints}
			C^{\Lambda}(f)
			:=\bigg
			\{ x \in \R^n\bigg| 
			\sigma_i(D_x f ) 
			\leq \Lambda_i,
			\,
			\forall\, i=1,\dots,m
			 \bigg\}  .
		\end{equation}
	\end{definition}
The following lemma is fundamental to study the semialgebraic structure of the set of almost-critical points of a polynomial map.
\begin{lemma}\label{SublevelsSemialgebraicFunc}
	Let $A \subset \R^s$ be a semialgebraic set and let
	$f :A \to \R$ be a semialgebraic function.
	Then, for any $t \in \R$, the sublevel set
	$\{ x \in A \mid f(x) \leq t \}$ admits a semialgebraic description whose
	diagram does not depend on $t$.
	\end{lemma}
\begin{proof}
The graph of $f$ is semialgebraic, hence we can write it as
\begin{equation}
\graf (f) = 
\bigcup_{i=1}^{a} \bigcap_{j=1}^{b_i} E_{ij},
\end{equation} 
where $E_{ij}$ is of the form 
$\{p_{ij} < 0 \}$ or 
$\{ p_{ij} \leq 0 \}$
 and 
$p_{ij}: \R^s \times \R \to \R$ is a polynomial. Denoting by $\pi_1: \R^s \times \R \to \R^s$ the projection, we have
\begin{equation}\label{eq:projection}
	\{ x \in A \mid f(x) \leq t \} = 
	\pi_1 
	\Big( \graf (f)  \cap 
	\{ (x,y) \in \R^s \times \R 
	\mid y -t \leq 0\} \Big)  .
\end{equation}
The set $\graf (f)  \cap 
\{ (x,y) \in \R^s \times \R 
\mid y -t \leq 0\}$ is semialgebraic, with a description given by
\begin{equation}\label{DescrizioneSemialgGraf}
	\graf (f)  \cap 
	\{ (x , y)
	\mid y - t \leq 0\}
	=
	 \bigcup_{i=1}^{a} \bigcap_{j=1}^{b_i +1} 
	E_{i,j} ,
\end{equation}
where  
$E_{i , b_i + 1} = 
\{ (x , y) \in \R^s \times \R 
\mid y - t \leq 0\} $.
In particular, this set admits a diagram that does not depend on $t$, since the degree of the polynomial $y-t$ is equal to $1$. 
As explained in \cref{item:projectionsaugata}, the projection of a semialgebraic set is semialgebraic and its diagram is controlled explicitly by the diagram of the original set only. In particular, the diagram of the projection \eqref{eq:projection} admits a description with a diagram $ ( s , I , J)$, that does not depend on $t$.
\end{proof}

Let $\sym_m^+$ denote the set of positive semidefinite matrices of size $m$. This is clearly a semialgebraic subset of all $m \times m$ real matrices. The following lemma is elementary and we omit its proof.
\begin{lemma}\label{AutovaloriSemialg}
The functions $\lambda_i : {\sym}_m^+ \to \R$ are continuous and semialgebraic. 
\end{lemma}
In the following proposition we provide  an estimate on the diagram of a description of the (semialgebraic) set of almost-critical points of polynomial maps.
\begin{proposition}\label{DiagramLambdaCrit}
Let $d,m\in \N$. Then there are $c'=c'(d,m)$ and $d'=d'(d,m)$ such that for all $n\geq m$ and any polynomial map	 $p: \R^n \to \mathbb{R}^m$ with components of degrees at most $d$, and all $\Lambda= (\Lambda_1 , \dots , \Lambda_m) \in \R^m_+$, the set $C^{\Lambda}(p)$ is closed, semialgebraic, with
	\begin{equation}
		D( C^{\Lambda}(p) ) = (n , c' , d').
	\end{equation}
	(In particular, the diagram does not depend on $\Lambda$ and it depends on $n$ only in the number of variables.)
\end{proposition}
\begin{proof}
By \cref{AutovaloriSemialg}, for every $k=1, \ldots, m$, the function 
$\lambda_k : {\sym}_m^+ \to \R$ is semialgebraic, hence from \cref{SublevelsSemialgebraicFunc} applied to $A= \sym_m^+$ and the semialgebraic functions $\lambda_k:A\to \R$, we deduce the following fact:
for any $t \in \R$ we can write
\begin{equation}
\{ M \in \sym_m^+ \mid \lambda_k (M) \leq t^2 \} =
\bigcup_{i=1}^a 
\bigcap_{j=1}^{b_i} 
\{ M \in \R^{m\times m} \mid P_{k , i, j , t} (M) \leq 0\}
\end{equation}
for some polynomials 
$P_{k ,i, j , t}: \R^{m \times m} \to \R$, $a,b_1,\dots,b_a \in \N$, and with 
$ a \max_i b_i $  
and 
$\max_{i,j} \deg 
( P_{k , i,j, t} )$ depending only on $m$ and on $k$ (and not on $t$).
We can assume that $ a $ does not depend on $k$, and all the $ b_i $'s do not depend on $a,k$, simply by adding empty or trivial sets.

Now, for any $k= 1 , \dots , m$
\begin{equation}
\{N \in \R^{m \times n} \mid \sigma_k(N) \leq t\} 
=
\{N \in \R^{m \times n} \mid \lambda_k (NN^\top ) \leq t^2\} ,
\end{equation}
therefore we can write
\begin{equation}
	\{ x \in \R^n \mid \sigma_k( D_xp ) \leq t\} 	
	=
	\bigcup_{i=1}^a 
	\bigcap_{j=1}^{b_i} 
	\{x \in \R^n \mid P_{k , i,j , t} (D_x p D_xp^\top) \leq 0\} .
\end{equation}
Observe that $N \to NN^\top$ is polynomial with components of degree $2$.
Therefore, the map
$
q: \R^{n } \to 
\R^{m \times m},$
given by $q( x ) := D_x p D_xp^\top$ is polynomial, with components of degree at most $2d$. 
Hence, the polynomial
\begin{equation}
P_{k ,i,j , t} \circ q : 
\R^{n} \to \R
\end{equation}
has degree at most 
$2 d \max_{i,j} \deg P_{k,i,j, t}$, and in particular its degree has a bound that depends on only $d,m,k$ (and not on $t$).
The set of almost-critical points of $p$ can be described as
\begin{equation}
C^{\Lambda} (p) = 
\bigcap_{k =1}^m 
\{ x \in \R^n\mid 
\sigma_k (D_x p ) 
\leq \Lambda_k
\}
= \bigcap_{k=1}^m \bigcup_{i=1}^a \bigcup_{j=1}^{b_{i}} \{ x \mid P_{k,i,j,\Lambda_k}(D_x p D_x p^\top)\leq 0\}.
\end{equation}
It is clear that $C^\Lambda(p)$ is closed. Note that
$d':=\max_{k,i,j}  \deg 
(P_{k , i, j ,\Lambda_k} \circ q )$  depends only on $d,m$. Using the distributivity of intersection, we can rewrite $C^{\Lambda}(p)$ in the form \eqref{eq:defA}, where the number of unions and intersections depends only on $m$, and the maximum degree of the polynomials is $c' = c'(d,m)$, concluding the proof.
\end{proof}

Now we state and prove the main result of this section (\cref{thm:estimatevariations-intro} in the Introduction).
	\begin{theorem}\label{thm:estimatevariations}
	Let $n\geq m$, and $p: \R^n \to \R^m$ be a polynomial map with components of degree at most $d$. For  $ i= 0, \dots, m$, $\Lambda=(\Lambda_1 , \dots , \Lambda_m)\in \R^m_+$ and $r >0$, we have  
				\begin{equation}\label{estimatevariations2}
		V_i( p ( C^{ \Lambda }( p ) \cap B_{\R^n}(r) ))
			\leq
			\mathrm{cst}(m,r)n^m			\beta_0^n \Lambda_{0} 
			\cdots \Lambda_{i},
		\end{equation}
							where $\beta_0=\beta_0(d,m)>1$ depends only on $d$ and $m$, $\mathrm{cst}(m,r)$ depends only on $m,r$, and we set $\Lambda_0=1$.
						\end{theorem}
	\begin{proof}
	By \cref{DiagramLambdaCrit} the set $A = C^{\Lambda}(p) \cap B_{\R^n}(r)$ is a closed and bounded semialgebraic set whose  diagram is $D(A) =(n,c',d')$ where $c'=c'(d,m)$, $d'=d'(d,m)$ (which in particular do not depend on $n$, $\Lambda$). Furthermore, for the polynomial map $p: \R^n \to \R^m$ it holds, by construction
	\begin{equation}
	\overline{\sigma}_i:=\sup_{x \in A} \Big({\sigma}_1(D_xp)\cdots{\sigma}_i(D_xp) \Big) \leq \Lambda_0\cdots \Lambda_i.
	\end{equation}
We can then apply \cref{BehaviourVariations1}, yielding \eqref{estimatevariations2} with $\beta_0 = \beta_{\mathrm{yc}}(c',\max\{d,d'\},m) >1$.
	\end{proof}
	
	\section{Sard properties for infinite-dimensional maps}\label{sec:Sardinfinite}
	
Let $H$ be a Hilbert space. We consider a map $f: U \to \R^m$ where $U$ is an open subset of $H$. If $f$ is $C^1$ we denote by $\crit(f)$ the set of its critical points. For a fixed $\nu \in \N$, we also consider the set
\begin{equation}\label{eq:nucritical}
	\crit_{\nu}(f) 
	=
	\{x \in U \mid \mathrm{rank}(D_xf) \leq \nu \}  .
\end{equation}

Given a relatively compact $S \subset \R^m$, for any $\varepsilon >0$ we denote by 
$M(\varepsilon , S) $ the \emph{$\varepsilon$-entropy} of $S$, which is the minimum number of closed balls of radius $\varepsilon$ that we need to cover $S$.
\begin{definition}[Entropy dimension]\label{def:entropy}
	The \emph{entropy dimension} of $S$ is defined as
	\begin{equation}
		\dim_e (S) = 
		\limsup_{\varepsilon \to 0^+}
		\frac{\ln M(\varepsilon , S) }{
			\ln(\frac{1}{\varepsilon}) }  .
	\end{equation}
\end{definition}
In \cite[Ch.\ 2]{ComteYomdin}, it is proved that
\begin{equation}
	\dim_{\mathcal{H}} (S) \leq \dim_e (S) .
\end{equation}
The inequality can be strict, since there are sequences of real numbers with positive entropy dimension. We also note that the entropy dimension is stable under \emph{finite} unions of sets, namely
\begin{equation}
\dim_e\left( \bigcup_{n=1}^N S_n\right) = \max_{n\leq N} \dim_e(S_n),
\end{equation}
while it is not even sub-additive with respect to \emph{countable} unions. In fact, the entropy dimension of a point in $\R$ is equal to zero, while the entropy dimension of $\mathbb{Q}\subset \R$ is equal to one. (This is a major difference with respect to the Hausdorff dimension which is countably stable.)

For Banach spaces $X$, $Y$ and linear continuous maps $L \in \mathcal{L}(X,Y)$, we denote by 
\begin{equation}\label{eq:operatornorm}
\|L\|_{\op}:= \sup_{\|x\|_{X}\leq 1} \|L x\|_{Y}
\end{equation}
the operator norm. With some abuse of notation the same symbol $\|L\|_{\op}$ is used for different $X$ and $Y$, but there should be no confusion as the domain and codomain of $L$ are clear from the context.  We also use without risk of confusion the symbol $\|\cdot\|$ to denote the usual Euclidean norm of $\R^m$.

\subsection{Sard-type theorem for well approximated maps}

The purpose of this section is to prove the following theorem, which is our main Sard-type result (\cref{GeneralSardTheorem-intro} in the Introduction).

\begin{theorem}\label{GeneralSardTheorem}
Let $d,m \in \N$. There exists a constant $\beta_0 = \beta_0(d,m)> 1$ such that the following holds. Let $H$ be a Hilbert space, and let $f: H\to \R^m$ be a $\mathcal{C}^1$ map such that its differential $Df : H \to \mathcal{L}(H,\R^m)$ is weakly continuous. Let $B\subset H$ be a bounded set with this approximation property: there exist a sequence $E_n\subset H$ of linear subspaces, $\dim(E_n) = n$, and polynomial maps $f_n: E_n \to \R^m$ with uniformly bounded degree:
\begin{equation}
\sup_{n \in \N } \deg f_n \leq d < \infty,
\end{equation}
such that for some $q>1$, $c\geq 0$, and all large enough $n$ it holds
\begin{equation}\label{ipotesiMainTheorem}
		\sup_{x \in B}
		\bigg(\|f(x)  - f_n \circ \pi_{E_n} (x) \|
		+
		\| \left(D_x f\right)|_{E_n}
		- D_{\pi_{E_n}(x)} f_n  \|_{\op}\bigg)
		\leq 
		c q^{-n}.
	\end{equation}
Then for all $\nu \leq m-1$ it holds
\begin{equation}
		\dim_e\bigg( f\big( \crit_{\nu} (f) \cap B \big)\bigg) 
		\leq 
		\nu + \frac{\ln \beta_0}{ \ln q}.
	\end{equation}
In particular, if $q> \beta_0$, then $f$ satisfies the Sard property on $B$:
	\begin{equation}
	\mu\bigg(f(\crit(f)\cap B)\bigg)=0.
	\end{equation}
\end{theorem}
\begin{remark}
As it will be clear from the proof, the upper bound $c q^{-n}$ in the r.h.s.\ of \eqref{ipotesiMainTheorem} can be replaced by any $C_n \geq 0$, such that $\limsup_{n\to\infty} C_n^{1/n} = q^{-1}$, yielding the same results.
\end{remark}

For the benefit of the reader we provide a general outline of the strategy of the proof. Our goal is to estimate the ``size'' of the set 
	$f( \crit_{\nu} (f ) \cap B )$.
	More precisely, we bound the \emph{$\varepsilon$-entropy} of 
	$f( \crit_{\nu} (f ) \cap B )$
	for any
	$\varepsilon >0$, and then pass to the limit $\varepsilon \to 0$.

	We observe that the $\varepsilon$-entropy of 
	$f( \crit_{\nu} (f ) \cap B )$ is controlled with the one of the set of 
	almost-critical values $f( C^{\Lambda} (f ) \cap B )$, see \cref{def:singval}. In turn, this is controlled by the $\varepsilon$-entropy of suitable almost-critical values of the approximating maps $f_n$, namely $f( C^{\Lambda_\varepsilon} (f_n ) \cap B )$, see \eqref{RelazioneEntropie}. 
	
To proceed, we use the relation between the $\varepsilon$-entropy of sets in $\R^m$ and their Vitushkin variations (\cref{defvariation}), provided by 
	\cite[Thm. 3.5]{ComteYomdin}.

	Finally, we estimate the Vitushkin variations of $f_n( C^{\Lambda_\varepsilon} (f_n ) \cap B )$ applying \cref{thm:estimatevariations-intro}, exploiting the hypothesis that the maps $f_n$ are polynomial. This is the connection with quantitative semialgebraic geometry from \cref{sec:quantitativesemialgebraic}. The crucial point of \cref{thm:estimatevariations-intro} is that the bound is quantitative with respect to parameters $\Lambda=\Lambda_\varepsilon$ and $n$. This allows to chose a large enough $n=n_\varepsilon$, see \eqref{eq:defneps}, to ensure convergence when passing to the limit $\varepsilon \to 0$.
	\begin{proof}

\textbf{Step 1}. We show that for any  
		$\nu= 0 ,\dots , m-1$
		the set $\crit_{\nu} (f) \cap B$ is contained in a suitable set of almost-critical points of $f$. By assumption, $B$ is weakly pre-compact. Since $Df : H \to \mathcal{L}(H,\R^m)$ is weakly continuous, and
		by \cref{SingolarValuesLipschitz} the function
		$\sigma_i : \mathcal{L}(H,\R^m)\to \R$ is continuous for every $i=1,\dots, m$. Therefore we have
		\begin{equation}
				\Sigma_i
				:=
				\sup_{B} 
				\sigma_i (D_x f) 
				< \infty  .
		\end{equation}
		Hence, defining $\Lambda:
		= 
		(\Sigma_1, \dots, \Sigma_{\nu}, 0 \dots,0) \in \R^m$,
		we get
		\begin{equation}\label{ranghi}
			\crit_{\nu} (f) \cap B
			\subset 
				C^{\Lambda}(f) \cap B .
		\end{equation}

\textbf{Step 2}. We now relate critical points of $f$ to almost-critical points of the approximating polynomials $f_n$. Fix $\varepsilon >0$. Recall that, by \cref{SingolarValuesLipschitz}, the singular values are $1$-Lipschitz with respect to the operator norm. Then for any sufficiently large
		$n \in \N$ such that 
		$ c q^{-n} \leq \varepsilon$ and for all $x \in B$ we have
		\begin{equation}\label{relationSVD}
				| 
				\sigma_i
				\left (D_x f |_{E_n} 
				\right) 
				- 
				\sigma_i
				(
				D_{\pi_{E_n} (x)} f_n
				) | 
				\leq 
				\| \left (D_x f \right )|_{E_{n}}
				- D_{\pi_{E_{n}} (x)} f_{n}  \|_{\op} \leq
				\varepsilon  .
		\end{equation}
		Hence, for all $x \in B \cap C^{\Lambda}(f)$	 
		we have the following estimate for all $i=1,\dots,m$
		\begin{equation}\label{SVDfinitodim}
				\sigma_i
				\left (
				D_{\pi_{E_n}(x)} f_n
				\right)
				\leq 
				\varepsilon
				+
				\sup_{y \in C^{\Lambda}(f) \cap B}
				\sigma_i
				\left( 
				D_y f |_{E_n}
				\right)
				\leq
				\varepsilon
				+
				\sup_{y \in C^{\Lambda}(f) \cap B }
				\sigma_i
				\left( 
				D_y f 
				\right)
				\leq 
				\varepsilon
				+\Lambda_i ,
		\end{equation}
	 where, in the second inequality, we used  that singular values are monotone with respect to restriction to subspaces. Hence, defining
		\begin{equation}
				\Lambda^\varepsilon
				:= 
				( \Sigma_1+ \varepsilon, \dots , \Sigma_{\nu} + \varepsilon, \varepsilon, \dots ,\varepsilon )   ,
		\end{equation}
		by \eqref{SVDfinitodim} we get
		\begin{equation}\label{inclusioncritn}
				\pi_{E_n}
				\left( C^{\Lambda}(f) 
				\cap B \right)
				\subset
				C^ { \Lambda^\varepsilon }(f_n)
				\cap 
				\pi_{E_n} (B)
				 \subset
				C^ { \Lambda^\varepsilon }(f_n)
				\cap 
				B_{E_n} (r),
		\end{equation}
		where $r>0$ is such that $B\subseteq B_H(r)$, and thus $B_{E_n} (r) = B_H(r) \cap E_n$ denotes the ball in $E_n$ defined by the restriction of the Hilbert norm.

\textbf{Step 3}.
		We use the inclusion \eqref{inclusioncritn} to estimate the entropy dimension of 
		$f( \crit_{\nu} (f) \cap B) $ 
		with the one of almost-critical values of $f_n$.		
		By \eqref{ipotesiMainTheorem}, for sufficiently large $n$, we have
		\begin{equation}\label{eq:dissup}
				\sup_{x \in B}
				\| f(x) - f_n \circ \pi_{E_n} (x)  \|
				\leq 
				\varepsilon  .
		\end{equation}
Finally, we get
		\begin{align}
					M
					( 2\varepsilon, f( \crit_{\nu} (f) \cap B) 
				    & \leq
					M (   2\varepsilon, f (C^{\Lambda}(f) \cap B ) ) & \text{by \eqref{ranghi}} \\
					& \leq 
					M \left( 2\varepsilon,  
					\mathcal{U}_{\varepsilon} 
					\left (f_n \circ \pi_{E_n}
					\left (
					C^{\Lambda}(f) \cap B \right )
					\right ) \right ) & \text{by \eqref{eq:dissup}} \\
					& \leq 
					M \left(\varepsilon,  
					f_n \circ \pi_{E_n}
					\left(C^{\Lambda}(f) \cap B \right )
					\right ) & \text{by definition of $\mathcal{U}_\varepsilon$} \\
					& \leq  M( \varepsilon , 
				f_n
				( C^{\Lambda^\varepsilon } (f_n)
				\cap
				B_{E_n}(r)
				)  ), & \text{by \eqref{inclusioncritn}} \label{RelazioneEntropie}
		\end{align}
		 where we recall that $M(\varepsilon,A)$ is the minimal number of closed balls of $\R^m$ of radius $\varepsilon$ needed to cover $A\subset \R^m$, and $\mathcal{U}_\varepsilon(A)$ denotes the $\varepsilon$ neighbourhood of $A$.
		This concludes the proof of step 3.

		The rest of the proof consists in estimating \eqref{RelazioneEntropie}. We will use the results from \cref{sec:variations}, for this reason in the next step we discuss how to bring the problem to the Euclidean space $\R^{n}$.

\textbf{Step 4}.
 		We fix a basis for $E_n$ which is orthonormal with respect to the inner product of $H$. This yields a linear isometry of Hilbert spaces $\ell_n :\R^n \to E_n$. In particular, for the transpose, it holds $\ell_n^\top =\ell_n^{-1}$. Consider then the polynomial map $\tilde{f}_n : \R^n \to \R^m$ defined by $\tilde{f}_n:=f_n\circ \ell_n$, $n\geq m$. We have
		 \begin{equation}
			\sigma_i(D_v \tilde{f}_n)
			=
			\lambda_i 
			( D_{\ell_n (v)} f_n \circ \ell_n \circ \ell_n^\top \circ D_{\ell_n (v)} f_n^\top  )^{1/2}
			=
			\sigma_i ( D_{\ell_n(v)} f_n ) , \qquad \forall\, v\in\R^n.
		 \end{equation}
	 	It follows that,  for sufficiently large $n$,
		 \begin{equation}\label{SingularValuesAndHilbertBasis}
		 	\ell_n (C^{\Lambda^\varepsilon} (\tilde{f}_n ))
		 	=
		 	C^{\Lambda^\varepsilon} (f_n )  .
		 \end{equation}
	 Hence, 
		 we have
		 \begin{equation}\label{NearCttiPointsInCoordinates}
			f_n
				( C^{\Lambda^\varepsilon } (f_n)
				\cap
				B_{E_n}(r)
					)
			=  
			f_n
		( \ell_n (C^{\Lambda^\varepsilon } (\tilde{f}_n))
		\cap
		B_{E_n}(r)
		)	
		=
		\tilde{f}_n
		( C^{\Lambda^\varepsilon } (\tilde{f}_n)
		\cap
		B_{\R^{n} } (r)
		)   .
	 \end{equation}
		By \cite[Thm.\ 3.5]{ComteYomdin} we have
		\begin{equation}\label{GeneralSard_2}
			M(  \varepsilon, 
			\tilde{f}_n
			( C^{\Lambda^\varepsilon } (\tilde{f}_n)
			\cap
			B_{\R^{n} } (r)
			)  
						)
			\leq
			C(m)
			\sum_{i=0}^m
			\frac{1}{\varepsilon^i}
			V_i\bigg(
			\tilde{f}_n
			( C^{\Lambda^\varepsilon } (\tilde{f}_n)
			\cap
			B_{\R^{n} } (r)
			)\bigg)  ,
		\end{equation}
		where $C(m)>0$ is a constant that depends only on $m$. We now apply the results from \cref{sec:variations} in order to estimate the right-hand side of the last inequality.
	 
\textbf{Step 5}. Take $n$ sufficiently large so that it also satisfies $n\geq m$. Then, by \cref{thm:estimatevariations-intro}, for every $i=0, \dots ,m$ we have 
		\begin{equation}\label{GeneralSard_1}
			V_i\bigg(
			\tilde{f}_n
			( C^{\Lambda^\varepsilon } (\tilde{f}_n)
			\cap
			B_{\R^{n} } (r)
			)\bigg)
			\leq
			\mathrm{cst}(m,r)n^m \beta_0^n \Lambda_{0}^\varepsilon 
			\cdots \Lambda_{i}^\varepsilon 
			 ,
		\end{equation}
where we recall $\beta_0=\beta_0(d,m)$ and $\Lambda_0^\varepsilon=1$ by definition.	Now, we have for all $i=1,\dots,m$
		\begin{align}
				\Lambda^\varepsilon_1 \cdots \Lambda^\varepsilon_i
				& =
				(\Lambda_1 +
				\varepsilon) 
				\cdots (\Lambda_{i} + \varepsilon) \\
				& = 
				\sum_{h=0}^i 
				\varepsilon^{i-h}
				\sum_{0 = j_0 < j_1 < 
					\dots < j_h \leq i}
				\Lambda_{j_0} \cdots \Lambda_{j_h} \\
				& \leq 
				\sum_{h=0}^i 
				\varepsilon^{i-h}
				\sum_{0 = j_0 < j_1 < 
					\cdots < j_h \leq i}
				\Lambda_0 \cdots \Lambda_h \\
				& = 
				\binom{i}{h}
				\sum_{h=0}^i 
				\varepsilon^{i-h}
				\Lambda_0 \cdots \Lambda_h\\  
				& \leq 
				i!
				\sum_{h=0}^i 
				\varepsilon^{i-h}
				\Lambda_0 \cdots \Lambda_h   .\label{GeneralSard_0}
		\end{align}
		Therefore, by \eqref{GeneralSard_2}, \eqref{GeneralSard_1}  and \eqref{GeneralSard_0}, for large $n$ we have
				\begin{align}
				M(  \varepsilon, 
				\tilde{f}_n
				( C^{\Lambda^\varepsilon } (\tilde{f}_n)
				\cap
				B_{\R^{n} } (r)
				)  
				)
				& \leq
				\mathrm{cst}(m,r) C(m)n^m \beta_0^{n} \sum_{i=0}^m \frac{	
				i!}{\varepsilon^i}  
				\sum_{h=0}^i 
				\varepsilon^{i-h}
				\Lambda_0 \cdots \Lambda_h   \\
				& \leq 
				\widetilde{\mathrm{cst}}(m,r)n^m
				\beta_0^{n}
				\sum_{h=0}^\nu
				\frac{\Lambda_0 \cdots \Lambda_h}{\varepsilon^h}  ,
		\end{align}
 where we note that, by definition, $\Lambda_h =0$ for all $h >\nu$, and $\widetilde{\mathrm{cst}}(m,r)$ denotes a constant that depends only on $m,r$.
		From this we deduce that
		\begin{equation}\label{stimaFinaleEntropiaPolinomi}
				\ln 
				M(  \varepsilon, 
				\tilde{f}_n
				( C^{\Lambda^\varepsilon } (\tilde{f}_n)
				\cap
				B_{\R^{n} } (r)
				)  
				)
				\leq 
				\ln ( \widetilde{\mathrm{cst}}(m,r) ) 
				+
				n
				\ln \beta_0 
				+
				m \ln n
				+
				\ln \left(    
				\sum_{h=0}^\nu 
				\frac{\Lambda_0 \cdots \Lambda_h }{\varepsilon^h}  
				  \right ),
		\end{equation}
	for all $\varepsilon >0$ and $n \in \N$ such that $ c q^{-n } \leq \varepsilon$ and $n \geq m$.

\textbf{Step 6}. We show how \eqref{stimaFinaleEntropiaPolinomi} yields the estimate on	$\dim_e(f(\crit_{\nu}(f) \cap B))$. For $\varepsilon>0$ we choose $n=n_\varepsilon$ where
\begin{equation}\label{eq:defneps}
n_\varepsilon:= \ceil*{  \log_q\frac{ c }{\varepsilon}},
\end{equation}
in such a way that \eqref{RelazioneEntropie},  \eqref{NearCttiPointsInCoordinates} and \eqref{stimaFinaleEntropiaPolinomi} hold when $\varepsilon \to 0$. By \eqref{RelazioneEntropie} and \eqref{NearCttiPointsInCoordinates} we have 
	\begin{equation}
		\dim_e f( \crit_{\nu} (f) \cap B )
		\leq 
		\limsup_{\varepsilon \to 0}
		\frac{
			\ln 
			M(  \varepsilon, 
			\tilde{f}_{n_{\varepsilon}}
			( C^{\Lambda^\varepsilon } (\tilde{f}_{n_{\varepsilon}})
			\cap
			B_{\R^{ {n_{\varepsilon}}} } (r)
			)  
			)
		}  
		{\ln \frac{1}{ 2 \varepsilon}}  .
		\end{equation}
	Hence now we estimate the right-hand side through \eqref{stimaFinaleEntropiaPolinomi}.
	As $\varepsilon $ goes to zero we have
	\begin{equation}
		\limsup_{\varepsilon \to 0}\frac{ n_{\varepsilon} \ln \beta_0}{\ln \frac{1}{2\varepsilon}}
		=  \lim_{\varepsilon \to 0}
		 \frac{	\ceil*{ 
				\log_q \frac{ c }{\varepsilon} }   
			\ln \beta_0}
		{\ln\frac{1}{2\varepsilon}}
		=		\frac{\ln \beta_0}{ \ln q} .
	\end{equation}
	We also have, taking into account \eqref{eq:defneps}, that
	\begin{equation}
\lim_{\varepsilon\to 0}		\frac{m \ln n_\varepsilon}{\ln ( \frac{1}{2 \varepsilon} )}
		= 0.
			\end{equation}
	Furthermore
\begin{equation}
 \lim_{\varepsilon \to 0} \frac{1 }{ \ln (\frac{1}{2 \varepsilon})}
\ln 
\left (\sum_{h=0}^{\nu} 
\frac{\Lambda_0 \cdots \Lambda_h}{\varepsilon^h}
\right ) 
	= \nu  .
\end{equation}
Therefore,
\begin{equation}
\dim_e f( \crit_{\nu} (f) \cap B )
\leq 
\nu + \frac{\ln \beta_0}{ \ln q},
\end{equation}		
concluding the proof.
	\end{proof}

\subsection{Sard-type theorems and Kolmogorov \texorpdfstring{$n$}{n}-width}

In this section, we recall the definition of Kolmogorov $n$-width, and we show its role in
Sard-type theorems. 

\begin{definition}[Kolmogorov $n$-width]\label{def:nwidth}
	Let $X$ be a normed space and $S \subset X$ be a subset. For $n \in \N$, the \emph{Kolmogorov $n$-width} of $S$ in $X$ is
	\begin{equation}
		\Omega_n(S,X)
		=
		\inf_{\dim E = n}
		\sup_{ u \in S}
		\inf_{ y \in E }
		\| 
		u - y
		\|  ,
	\end{equation}
	where the infimum is taken over all $n$-dimensional linear subspaces $E$ of $X$.
\end{definition}
The asymptotics of $n$-width measures in a quantitative way the compactness of a set, in fact the following holds (see \cite[Prop.\ 1.2]{Pinkus}).
\begin{proposition}\label{PropoPinkus}
	$S$ is precompact if and only if
	$S$ is bounded and $\lim_{n \to \infty}\Omega_n(S) = 0$.
\end{proposition}
For various properties of the $n$-width we refer the reader to \cite{Pinkus}; 
we recall here those that we will need in the sequel, see \cite[Thm.\ 1.1]{Pinkus}.
\begin{proposition}\label{thmPinkus}
Let $X$ be a normed space and $S \subset X$. Then for all $n \in \N$:
\begin{enumerate}[(i)]
\item $\Omega_n(S,X) = \Omega_n(\bar{S},X)$, where $\bar{S}$ is the closure of $S$;
\item For every $\alpha \in \R$
\begin{equation}
\Omega_n(\alpha S,X) = |\alpha| \Omega_n(S,X);
\end{equation}
\item Let $\mathrm{co}(S)$ be the convex hull of $S$. Then
\begin{equation}
\Omega_n(\mathrm{co}(S),X)=\Omega_n(S,X);
\end{equation}
\item Let $\mathrm{b}(S) = \{\alpha x \mid x \in S,\, |\alpha|\leq 1\}$ be the balanced hull of $S$. Then
\begin{equation}
\Omega_n(\mathrm{b}(S),X)=\Omega_n(S,X).
\end{equation}
\end{enumerate}
\end{proposition}
\begin{remark}
As a consequence, in all our results, up to enlarging $S$, one can assume without loss of generality that $S$ is convex and centrally symmetric, without changing its $n$-width.
\end{remark}
We establish now a connection between $n$-width of compact sets and locally Lipschitz functions.
 
\begin{lemma}\label{LocaleLipschitzCompattezza}
		Let $H$ be a Hilbert space, $(Y,\|\cdot\|_Y)$ a Banach space, and let $f: H \to Y$ be locally Lipschitz.  Let $K \subset H$ be a compact subset. Then there exists $n_0 \in \N$ such that for all $n \geq n_0$ there exist a linear subspace  $E_n \subset H$
	of dimension $n$
	such that
	\begin{equation}
		\sup_{x \in K}
		\| 	f(x)-f( \pi_{E_n}(x))	\|_Y		\leq 
		c(f,K)
		\Omega_n (K,H)   .
	\end{equation}
\end{lemma}
\begin{proof}
Let $\{B_H(x_i,r/2)\}_{i \in I}$ a finite cover of $K$ by balls, with centers $x_i \in K$ and radii $r/2>0$. We can assume that $f$ is $L$-Lipschitz on each $B_H(x_i,r)$, for some $L>0$. Since $K$ is compact $\Omega_n(K,H) \to 0$ as $n \to \infty$, hence there exists $n_0 \in \N$ such that for all $n\geq n_0$ we have $\Omega_n(K,H)< r/4$. By definition of $n$-width, for all $n \geq n_0$ there exists a $n$-dimensional subspace $E_n \subset H$ such that
\begin{equation}\label{LemmaLocLipscEQUATION2}
\sup_{x \in K} \|x- \pi_{E_n}(x)\| < 2\Omega_n(K,H) < \frac{r}{2}.
\end{equation}
Let $x \in K$. Then by construction $x \in B_H(x_i,r/2)$ for some $i \in I$. By \eqref{LemmaLocLipscEQUATION2} we have $\pi_{E_n}(x) \in B_H(x_i,r)$. Since $f$ is $L$-Lipschitz on every $B_H(x_i,r)$ we have
\begin{equation}
\sup_{x \in K} \|f(x) - f(\pi_{E_n}(x))\|_Y \leq L \sup_{x \in K} \|x- \pi_{E_n}(x)\| \leq 2L \Omega_n(K,H),
\end{equation}
where we used again \eqref{LemmaLocLipscEQUATION2}. This proves the result, with $c(f,K) = 2L$.
\end{proof}
We can now prove the following Sard-type theorem, corresponding to \cref{CorollaryOfGeneralSard-intro} in the Introduction.
\begin{theorem}\label{CorollaryOfGeneralSard}
Let $H$ be a Hilbert space and $K \subset H$ be a compact set such that, for some $q>1$, it holds
\begin{equation}
\limsup_{n\to \infty}\Omega_n(K,H)^{1/n} \leq q^{-1}.
\end{equation}
Let $d,m \in \N$. There exists $\beta_0 = \beta_0(d,m)>1$ such that for all $f\in \mathscr{P}_{d}^m(H)$ and $\nu \leq m-1$ we have	
\begin{equation}
		\dim_e \bigg(f\big( \crit_{\nu} (f) \cap K \big)\bigg)
		\leq 
		\nu + \frac{\ln \beta_0}{ \ln q}.
	\end{equation}
	In particular, if $q> \beta_0$, then the Sard property holds on $K$:
	\begin{equation}
	\mu\bigg(f\big(\crit (f) \cap K\big)\bigg)=0.
	\end{equation} 
\end{theorem}
\begin{proof}
Let $q_\varepsilon \in (0,q)$. By assumption, $\Omega_n (K , H) \leq q^{-n}_\varepsilon$ for sufficiently large $n$.
	We prove that $f$ satisfies the hypothesis of \cref{GeneralSardTheorem-intro}. Since $f \in \mathscr{P}^m_d(H)$ (see \cref{def:polyH}), the map
	\begin{equation}
		(f,Df) : H \to 
		\R^m \times \mathcal{L}(H , \R^m)
	\end{equation}
is locally Lipschitz. The codomain $Y=\R^m \times \mathcal{L}(H , \R^m)$ is a Banach space equipped with the norm $\|(v,A)\|_Y = \|v\| + \|A\|_{\op}$ for $v \in \R^m$ and $A \in \mathcal{L}(H,\R^m)$. Hence we can apply \cref{LocaleLipschitzCompattezza} to 
get $n_0 \in \N$ such that for all $n \geq n_0$ there exists a $n$-dimensional linear subspace 
	$E_n \subset H$ such that
	\begin{equation}
		\sup_{x \in K}
		\|f(x) - f( \pi_{E_n} (x))\|
		+
		\|D_xf - 
			D_{ { \pi_{ E_n(x) } } } f \|_{\op}
		\leq 
		c(f, K) q^{-n}_\varepsilon,
	\end{equation}
where $\pi_{E_n}: H \to E_n$ denotes the orthogonal projection. Let $f_n : E_n \to \R^m$ defined by the restriction $f_n:=f|_{E_n}$, which are polynomial maps with degree uniformly bounded above by $d$. It holds
	\begin{equation}
	\|
	 (D_x f)|_{E_n}
	-
	D_{\pi_{E_n} (x)} f_n
	\|_{\op}
=
	\|
	 (D_x f)|_{E_n}
	-
	(D_{\pi_{E_n} (x)} f )|_{E_n}
	\|_{\op}
	\leq
	\|
	D_x f
	-
	D_{\pi_{E_n} (x)} f
	\|_{\op}.
\end{equation}
As a consequence of the above two estimates, assumption \eqref{ipotesiMainTheorem-intro} of \cref{GeneralSardTheorem-intro} holds, yielding
\begin{equation}
		\dim_e \bigg(f\big( \crit_{\nu} (f) \cap K \big)\bigg)
		\leq 
		\nu + \frac{\ln \beta_0}{ \ln q_\varepsilon}.
	\end{equation}
Letting $q_\varepsilon \uparrow q$ we obtain the thesis.
\end{proof}

\subsection{Counterexamples to the Sard theorem in infinite dimension}\label{sec:point2}

In this section we provide examples of maps $f \in \mathscr{P}^m_d(H)$ as in \cref{CorollaryOfGeneralSard-intro} 
for which there exists a compact set $K\subset H$ 
such that the set
$f( \crit(f) \cap K)$ has not measure zero, however it holds
$\Omega_n (K,H) \leq c q^{-n}$. This shows the necessity of the quantitative assumption $q > \beta_0 (m , d)$.
As a consequence, we get lower bounds on the semialgebraic constant $\beta_0$.

We start with the following example, which is a minor modification of Kupka's one \cite{Kupka}. 
\begin{example}[Kupka revisited]\label{ExampleOptimality1}
Let $q>1$. Let $f:\ell^2 \to \R$ defined by
\begin{equation}
	f(x) = \sum_{k=1}^\infty\frac{1}{2^k}\phi(q^{k-1} x_k),
\end{equation}
where $\phi$ is the Kupka polynomial as in 
\cite[Sec.\ 10.2.3]{ComteYomdin}. Namely $\phi$ has degree $3$, with  $\phi(0)=\phi'(0)=\phi(1)-1=\phi'(1)=0$. If we assume
$ q^3 / 2 <1 $
then $f$ is $C^\infty$. All critical points have the form
\begin{equation}
	\crit(f) 
	= 
	\left\lbrace
	\left( x_1,\frac{x_2}{q},\dots,\frac{x_k}{q^{k-1}},\dots \right)\,\middle\vert\, x_i \in \{0,1\} \right\rbrace,
\end{equation}
and the set of critical values is $[0,1]$. 
Furthermore, given
$y=
\left( 
x_1,\frac{x_2}{q},\dots,\frac{x_k}{q^{k-1}},\dots \right) \in \crit(f)$, 
we consider its projection $y_n$ on 
$\mathrm{span} \{e_1 , \dots , e_n\}$, that is
\begin{equation}
	y_n
	=
	\left( 
	x_1,\frac{x_2}{q},\dots,\frac{x_n}{q^{n-1}},0,0,0, \dots \right) .
\end{equation} 
We have
\begin{equation}
\norm{ y - y_n }_{\ell^2}^2 \leq
\sum_{k= n+1}^{\infty } \frac{ 1 }{ q^{ 2 ( k-1 ) } }
\leq
\int_{n}^{ \infty} \frac{1}{q^{ 2t }} = 
\frac{q^{-2n}}{2 \ln q}  .
\end{equation}
Hence,
\begin{equation}
\Omega_n(\crit(f) , \ell^2) 
\leq 
\sup_{ y \in \crit(f)}
\norm{ y - y_n }_{\ell^2}
\leq 
\frac{q^{-n}}{\sqrt{2 \ln q}} \qquad \Longrightarrow \qquad \limsup_{n\to\infty}\Omega_n(\crit(f),\ell^2) \leq q^{-1}.
\end{equation}
Since $f$ and $\crit(f)$ satisfy all the hypothesis of \cref{CorollaryOfGeneralSard-intro} (provided that $q^3/2<1$), we get a consequence that
\begin{equation}
\beta_0 ( 1 , 3 ) \geq 2^{1/3}.
\end{equation}
\end{example}

We now get the lower bound on $\beta_0(1 , d)$ for $d\geq 3$ by improving on the above construction.

\begin{example}[Kupka revisited, degree $d\geq 3$]\label{ExampleOptimality2}
We need the following classical fact, which follows from \cite{realpoly2}: for any $d \geq 2$ there exists a polynomial	$\psi :\R \to \R$ of degree $d$ such that
	\begin{equation}\label{PolinomioValoriCritici}
		\psi (\crit (\psi))= \{0, \dots , d-2\}.
	\end{equation}
(See also \cite[Thm.\ 1]{realpoly} which gives  a version of this result taking into account also possible multiplicities.)
We fix $d\geq 3$ and we consider a polynomial $\psi$ of degree $d$ such that \eqref{PolinomioValoriCritici} holds. For $q >1$ we consider the map $f_d: \ell^2 \to \R$
\begin{equation}
f_d(x) := 
\sum_{k=1}^{\infty}
\frac{1}{ (d-1)^k }
\psi ( q^{k-1} x_k )  .
\end{equation}
If $q \in (1 , (d-1)^{1/d} ) $, then the map $f_d$ is $C^{\infty}$. The set of critical points of $f_d$ is 
\begin{equation}
\crit(f_d) =
\left \{ 
\left (
y_1 ,  \frac{y_2}{q},\dots , \frac{y_k}{q^{k-1} }  , \dots \right)
\in 
\ell^2
\,\middle\vert\,
\psi'(y_k)=0,\, \forall\, k \in \N
\right \}  .
\end{equation}
Hence we have
\begin{equation}
f_d (\crit (f_d)) =
\left \{ 
\sum_{k=1}^{\infty}
\frac{1}{(d-1)^k} c_k \, \middle \vert \,
c_k \in \{0, \dots, d-2\}
\right \}
=
[0,1]  .
\end{equation}
Let $\zeta = \max\{ |z| \mid \psi'(z) =0\}$. Exactly as in \cref{ExampleOptimality1} we get for every $n \in \N$
\begin{equation}
	\Omega_n(\crit(f_d) , \ell^2) 
	\leq 
	\frac{\zeta}{\sqrt{2 \ln q}} 
	q^{-n} \qquad \Longrightarrow \qquad \limsup_{n\to\infty}\Omega_n(\crit(f_d) , \ell^2)  \leq  q^{-1}.
\end{equation}
For any $q \in (1 , (d-1)^{1/d})$ the function $f_d$ and the set $K = \crit(f_d)$ satisfy the hypothesis of \cref{CorollaryOfGeneralSard}, but $f_d(\crit(f_d))=[0,1]$. It follows that
\begin{equation}
\beta_0 (1 , d) \geq (d-1)^{1/d},\qquad \forall\, d\geq 3.
\end{equation}
\end{example}

In the next example we extend the construction to arbitrary dimension of the codomain.
\begin{example}[Kupka revisited, $d\geq 3$, arbitrary codomain]\label{ExampleOptimality3}
Let $f_d$ be as in \cref{ExampleOptimality2}, and consider 
$g_d: \R^{m-1} \times \ell^2 \to \R^m$ defined as
\begin{equation}
	g_d(x , y) := (x , f_d (y) ).
\end{equation}
We have 
\begin{equation}
	\crit( g_d )
	= 
	\R^{m-1} \times \crit ( f_d )  .
\end{equation}
Taking $E_{n+m-1} = \mathrm{span}\{(e_i,0),\, (0,e_j) \mid i=1,\dots,m-1,\, j=1,\dots,n\} \subset \R^{m-1}\times \ell^2$, for all $n\in \N$, we can estimate the $n$-width of $\crit(g_d)$ as
\begin{equation}
	\Omega_{n+m-1} (\crit (g_d) , \R^{m-1} \times \ell^2 )
	\leq
	\frac{\zeta}{\sqrt{2 \ln q} } q^{-n}\qquad \Longrightarrow \qquad \limsup_{n\to\infty}\Omega_n(\crit(g_d), \R^{m-1} \times \ell^2) \leq q^{-1}.
\end{equation}
We observe, provided that $q\in (1,(d-1)^{1/d})$, the function $g_d$ satisfies all the hypothesis of \cref{CorollaryOfGeneralSard}, but the set
$g_d(\crit(g_d)) = \R^{m-1} \times [0,1]$ has not measure zero. Hence
\begin{equation}
\beta_0(d,m) \geq (d-1)^{1/d},\qquad \forall\, m\in \N,\, d\geq 3.
\end{equation}
\end{example}

We collect the content of \cref{ExampleOptimality1,ExampleOptimality2,ExampleOptimality3} in a unified statement (see \cref{thm:sintesiexamples-intro}).

\begin{theorem}\label{thm:sintesiexamples}
Let $d,m \in \N$, with $d\geq 3$, and let $1<q<(d-1)^{1/d}$. There exist a Hilbert space $H$, and $f\in \mathscr{P}_d^m(H)$ such that $K=\crit(f)\cap B_H(r)$ is compact for all $r>0$, with
\begin{equation}
\limsup_{n\to \infty} \Omega_n(K,H)^{1/n} \leq q^{-1},
\end{equation}
and $f : H \to \R^m$ does not verify the Sard property, namely $\mu(f(\crit(f)\cap K)) > 0$. Therefore, the semialgebraic constant $\beta_0(d,m)$ of \cref{CorollaryOfGeneralSard-intro} satisfies
\begin{equation}
\beta_0(d,m) \geq (d-1)^{1/d}, \qquad \forall\, m\in\N,\, d\geq 3.
\end{equation}
\end{theorem}

In all the examples above we have taken $K$ to be set of all critical points of a given function in a ball.
We now show a different construction, where the set $K$ is strictly contained in the set of critical points, more precisely is the set of points where the differential has rank zero.

\begin{example}[Rank zero counterexample to Sard]
We consider $f_d:\ell^2 \to \R$ as in \cref{ExampleOptimality2}, and $m\in \N$. We define
 $h: ( \ell^2 )^m \to \R^m$ as
\begin{equation}
h(x^1 , \dots , x^m) := (f_d(x^1), \dots , f_d(x^m) )  .
\end{equation}
For any $x =(x^1 , \dots , x^m) \in (\ell^2)^m$ 
we have
\begin{equation}
\im (D_x h) =
\im (D_{x^1} f_d) \times \dots \times \im (D_{x^m} f_d)   .
\end{equation}
Hence $x \in \crit_{\nu}(h)$ (the set of points where the rank of the differential is $\leq \nu$) if and only if 
at least $m-\nu$ components of $x$ are critical points of $f_d$.
In particular, we consider the compact set $K$ of all critical points of $h$ of rank zero, namely
\begin{equation}
K := \crit_0 (h) = \crit(f_d) ^ m.
\end{equation}
We have $K \subsetneq \crit(h)$, and
\begin{equation}
h(\crit(h) \cap K) = h(\crit_0 (h) ) = [0,1]^m  .
\end{equation}
As in the previous examples, $K$ is compact with exponential $n$-width. By taking finite-dimensional spaces $E_{nm}= \mathrm{span}\{\underbrace{(0,\dots,e_i,\dots,0)}_{j\text{-th component}} \in (\ell^2)^m \mid i =1,\dots,n,\, j=1,\dots,m\}$, for $n\in \N$ we have
\begin{equation}
	\Omega_{nm} (K , (\ell^2)^m ) 
	\leq 
	m 
	\Omega_{n}
	(\crit(f_d) , \ell^2 )
	\leq
	\frac{m \zeta}{\sqrt{2 \ln q}}
	q^{-n} \qquad \Longrightarrow\qquad \limsup_{n\to\infty}\Omega_n(K,(\ell^2)^m) \leq q^{-1/m},
\end{equation}
where $\zeta$ is the same constant appearing in \cref{ExampleOptimality2}.
\end{example}

\subsection{Sard on linear subspaces} 

Next, we prove \cref{thm:main2intro}. We will need the following lemma.

\begin{lemma}[Critical points under restrictions]\label{lem:restrandclosures}
Let $(X,\|\cdot\|)$ be a normed space and let $f: X \to \R^m$ be a $\mathcal{C}^1$ map. For any linear subspace $Y\subset X$, consider the  restriction $f|_{Y}: Y \to \R^m$. Note that the latter is a $\mathcal{C}^1$ map in the Fréchet sense from the normed vector space $(Y,\|\cdot\|)$ to $\R^m$, with $D_u (f|_{Y}) = (D_u f)|_{Y}$ for all $u\in Y$. Then, for all $\nu =0,\dots,m-1$ it holds
\begin{equation}\label{eq:critintvsrestr}
\mathrm{Crit}_{\nu}(f|_{Y}) \supseteq \mathrm{Crit}_\nu(f)\cap Y,
\end{equation}
with equality if $Y$ is dense in $X$.
\end{lemma}
\begin{proof}
The inclusion is proven by noting that, for all $u\in Y$, it holds
\begin{equation}
\mathrm{rank} \left(D_u (f|_{Y}) : Y \to \R^m\right) \leq \mathrm{rank} \left(D_u f : X\to \R^m\right).
\end{equation}
Assume that $Y$ is dense in $X$. Thus  if $\lambda \in (\R^m)^* \setminus \{0\}$ is such that 
		$0\equiv 
		\lambda \circ D_u(f|_{Y})=
		(\lambda \circ D_u f)|_{Y}$, since $Y$ is dense in $X$, it follows that 
		$\lambda \circ D_u f\equiv 0$. This means that for all $u\in Y$
		\begin{equation}
	\mathrm{rank}\left( D_u f : X\to\R^m\right) \leq \mathrm{rank}\left(D_u (f|_{Y}) :Y \to \R^m\right).
	\end{equation}
	It follows that $\mathrm{Crit}_{\nu}(f|_{Y}) \subseteq \mathrm{Crit}_{\nu}(f)\cap Y$, proving equality of the two sets.
\end{proof}

We can now prove \cref{thm:main2intro}, of which we recall the statement.

\begin{theorem}\label{thm:main2}
Let $H$ be a Hilbert space and $K \subset H$ be a compact set such that, for some $q>1$, it holds
\begin{equation}
\limsup_{n\to \infty}\Omega_n(K,H)^{1/n} \leq q^{-1}.
\end{equation}
Let $d,m \in \N$, let $\beta_0=\beta_0(d,m)>1$ be the same constant of \cref{CorollaryOfGeneralSard-intro}. Consider the (possibly non-closed) linear subspace
\begin{equation}
\V:= \mathrm{span}(K).
\end{equation}
Then for all $f\in \mathscr{P}_{d}^m(H)$ and $\nu\leq m-1$  the restriction $f|_{\V}: \V \to \R^m$ satisfies
\begin{equation}\label{eq:thm:main2-result}
\dim_{\mathcal{H}}\bigg(f\big(\mathrm{Crit}_\nu(f|_{\V})\big)\bigg)\leq \nu + \frac{\log\beta_0}{\log q}.
\end{equation}
In particular, if $q>\beta_0$, then the restriction $f|_{\V}: \V \to \R^m$ satisfies the Sard property
\begin{equation}
	\mu\bigg(f\big(\mathrm{Crit}(f|_{\V})\big)\bigg)=0.
\end{equation}
In all previous results, $f\big(\mathrm{Crit}_{\nu}(f|_{\V})\big)$ can be replaced with the smaller set $f(\mathrm{Crit}_{\nu}(f)\cap \V)$.
\end{theorem}
\begin{proof}
	Let $X$ denote the closure of $\V$ in $H$. Observe first that
	\begin{equation}
	\Omega_n(K, H)=\Omega_n(K, X).
	\end{equation}
	The inequality $\Omega_n(K, H)\leq \Omega_n(K, X )$ is clear since every $n$-dimensional subspace of $X$ is also a subspace of $H$. For the other inequality, 
	 since $ X $ is closed in $H$, we can write $H= X \oplus X^\perp$ and denote by $\pi:H\to X$ the orthogonal projection. Then, using that $K\subset V \subset X$, we obtain
			 \begin{align}
	 \Omega_n(K,H)
		&=
		\inf_{\substack{Y\subset H \\ \dim Y = n}}
		\sup_{ u \in K}
		\inf_{ y \in Y }
		\| 
		u - y
		\|  \geq \inf_{\substack{Y\subset H \\ \dim Y = n}}
		\sup_{ u \in K}
		\inf_{ y \in Y }
		\| 
		u - \pi(y)
		\| \\
		& =\inf_{\substack{Y\subset H \\ \dim Y = n}}
		\sup_{ u \in K}
		\inf_{ z \in \pi( Y ) }
		\| 
		u - z
		\| \\
		&=
		\inf_{\substack{Z\subset X \\ \dim Z \leq n}}
		\sup_{ u \in K}
		\inf_{ z \in Z }
		\| 
		u - z
		\| \\
		&
		= \inf_{\substack{Z\subset X \\ \dim Z =n}}
		\sup_{ u \in K}
		\inf_{ z \in Z }
		\| 
		u - z
		\|=\Omega_n(K, X).
		\end{align}
We can apply \cref{CorollaryOfGeneralSard-intro} to the map $\tilde{f}\in\mathscr{P}_{d}^m( X)$, defined by $\tilde{f} :=f|_{X}$, with $f\in\mathscr{P}_{d}^m(H)$, and the compact set $K\subset X$ obtaining for all $\nu \leq m-1$
\begin{equation}
		\dim_e \bigg(\tilde{f}\big( \crit_{\nu} (\tilde{f}) \cap K \big)\bigg)
		\leq 
		\nu + \frac{\ln \beta_0}{ \ln q}.
	\end{equation}		
		By applying \cref{lem:restrandclosures} (with $Y=\V$) and noting that, on $K\subset V$, we have $ \tilde{f}=f|_{\V}$, we obtain:
\begin{equation}\label{eq:measurezeroK}
\dim_e \bigg(f\big( \crit_{\nu} (f|_{\V}) \cap K \big)\bigg) = \dim_e \bigg(\tilde{f}\big( \crit_{\nu} (\tilde{f}) \cap K \big)\bigg)\leq 
		\nu + \frac{\ln \beta_0}{ \ln q}.
\end{equation}
Let $K'$ be the balanced and convex hull of $K$. It holds:
\begin{equation}
\V=\mathrm{span}(K) = \mathrm{span}(K') = \bigcup_{j\in \N} j K'.
\end{equation}
By \cref{thmPinkus} the $n$-width of the compact, convex, centrally symmetric sets $j K'$ satisfy
\begin{equation}
\Omega_n(j K',H) = j \Omega_n(K',H)=j \Omega_n(K,H).
\end{equation}
Therefore we can repeat the above argument, obtaining \eqref{eq:measurezeroK} for $j K'$ in place of $K$, for any $j\in \N$. Furthermore, we can replace the entropy dimension with the (smaller) Hausdorff dimension, which has the advantage of being countably stable. Taking the supremum over $j\in \N$ we obtain \eqref{eq:thm:main2-result}.

To conclude, we observe that by \cref{lem:restrandclosures} we have $f(\mathrm{Crit}_\nu(f|_{\V})) \supseteq f(\mathrm{Crit}_\nu(f)\cap \V)$, which yields the last claim of the statement of the theorem.
\end{proof}

\subsection{A class of maps with the global Sard property}\label{sec:class}
Using \cref{GeneralSardTheorem-intro}, we study a class of maps from a Hilbert space to $\R^m$ for which the Sard property holds true. The case $m=1$ is not new, and it was proved in \cite{YomdinApproxCompl}, see also \cite{ComteYomdin}.
We need the following preliminary facts.
\begin{lemma}[Markov inequality]\label{MarkovLemma}
Let $p: \mathbb{R}^n \to \mathbb{R}^m$ be a polynomial map with components of degree at most $d$. Then for every $r>0$ it holds
	\begin{equation}
		\sup_{x \in B_{\R^n}(r)}
		\|D_xp\|_{\op} 
		\leq 
		\frac{\sqrt{m} d^2}{r} 
		\sup_{x \in B_{\R^n}(r)} 
		\|p (x)\| .
	\end{equation}
\end{lemma}
\begin{proof}
From \cite[Thm.\ VI]{Kellogg1928} we get the thesis for $r=1$. Then, we conclude considering the rescaled polynomial 
$x \to p(rx)$.
\end{proof}
\begin{lemma}[Estimates on balls at different radii]\label{lem:1tor}
Let $p : \R^n \to \R^m$ be a polynomial map with components of degree at most $d$. Then for all $r>0$ it holds
\begin{equation}
\sup_{x \in B_{\R^n}(r)} \|p(x)\| \leq \alpha(d,m) n^d (1+r)^d \sup_{x \in B_{\R^n}(1)} \|p(x)\|,
\end{equation}
where $\alpha(d,m)=m^{d/2} d^{2d}(d+1)$ is a constant depending only on $d$ and $m$.
\end{lemma}
\begin{proof}
We can write the polynomial map as
\begin{equation}
p(x) = \sum_{|\alpha| \leq d} c_\alpha x^\alpha,
\end{equation}
where $c_\alpha \in \R^m$, and the sum is over the multi-indices $\alpha \in \N^n$ with $|\alpha| = \sum_{i} \alpha_i \leq d$. Since $|\alpha|!c_\alpha = \partial_{x^\alpha} p|_{x=0}$, we can iterate Markov inequality (for $r=1$, see \cref{MarkovLemma}) and obtain
\begin{equation}
\|c_\alpha\| \leq \frac{1}{|\alpha|!} m^{|\alpha|/2} d^{2|\alpha|} \sup_{x \in B_{\R^n}(1)} \|p(x)\| \leq m^{|\alpha|/2} d^{2|\alpha|} \sup_{x \in B_{\R^n}(1)} \|p(x)\|.
\end{equation}
Therefore, for $r>1$, we have
\begin{align}
\sup_{x \in B_{\R^n}(r)} \|p(x)\| & = \sup_{x \in B_{\R^n}(1)} \|p(r x)\| \\
& = \sup_{x \in B_{\R^n}(1)} \left\| \sum_{i=1}^d r^i \left( \sum_{|\alpha|=i} c_\alpha x^\alpha \right)\right\| \\
& \leq (1+r)^d \left(\sum_{|\alpha| \leq d} m^{|\alpha|/2} d^{2|\alpha|}\right) \sup_{x \in B_{\R^n}(1)} \|p(x)\| \\
& \leq (1+r)^d m^{d/2} d^{2d}(d+1)n^d \sup_{x \in B_{\R^n}(1)} \|p(x)\|,
\end{align}
where we used the fact that the number of multi-indices $\alpha \in \N^n$ with $|\alpha|\leq d$ is equal to $(1 + n + n^2 + \dots + n^d) \leq (d+1)n^d$
\end{proof}

For the rest of this section, we assume that the Hilbert space $H$ is separable, and we fix a Hilbert basis $\{e_i\}_{i \in \N} \subset H$. For $k \in \N$, we set $E_k :=\mathrm{span}\{e_1, \ldots, e_k\} $ and we denote by $\pi_k: H  \to E_k $ the corresponding orthogonal projection, that is if  $x = \sum_{i} x_i e_i$, then $\pi_k (x) = (x_1 , \dots , x_k)$. We now define a special class of maps in this setting.

\begin{proposition}[Construction of special maps]\label{RegolarSeriePolinomiali}
Let $H$ be a separable Hilbert space. For all $k \in \N$, let $p_k: E_k \to \R^m$ be polynomial maps with 
	$\sup_{k \in \N} \deg p_k \leq d$ for some $d \in \N$,
and such that
	\begin{equation}\label{eq:seriesconv}
	\sum_{k=1}^{\infty} \sup_{x \in B_{E_k} (r)} \|p_k(x)\| < 	 \infty , \qquad \forall\, r>0.
	\end{equation}
Then the map $f: H \to \mathbb{R}^m $ given by 
	\begin{equation}\label{eq:deffserie}
	f(x):= \sum_{k=1}^{\infty} p_k(x_1 , \dots,  x_k), \qquad \forall\, x \in H,
	\end{equation}
	is well--defined, $f\in \mathscr{P}_{d}^m(H)$ (see \cref{def:polyH}) and its differential is given by
	\begin{equation}\label{DifferentialeSeriePolinomi}
		D_xf=\sum_{k=1}^{\infty} D_{x}(p_k \circ \pi_k) .
	\end{equation}
\end{proposition}
\begin{proof}
For a Banach spaces $(X,\|\cdot\|_X)$, $(Y,\|\cdot\|_Y)$, and a closed set $U\subseteq X$, we denote by
$\mathcal{B}
(
U, Y)$
the space of all continuous and bounded functions $f: U \to Y$. It is a Banach space when endowed with the norm $\sup_{x \in U}
\norm{f(x)}_{Y}$.

We first note that the assumption \eqref{eq:seriesconv} implies that the sequence of maps $f_n:H \to \R^m$ given by $f_n:= \sum_{k=1}^n p_k\circ \pi_k$ is Cauchy in $\mathcal{B}(B_{H}(r),\R^m)$ for all $r>0$. In particular the right hand side of \eqref{eq:deffserie} converges locally uniformly. For any $n \in N$ we have 
	\begin{equation}\label{eq:dxfn}
		D_x f_n = 
		\sum_{k=1}^n 
		D_{x}( p_k \circ \pi_k ) .
	\end{equation}
	Furthermore,
	$x \mapsto D_x f_n \in 
	\mathcal{B}
	(
	B_{H}(r);  
	\mathcal{L} (H   , \R^m )
	)$ for all $n \in \N$, $r>0$. We now prove that it is a Cauchy sequence.
	Indeed, using \cref{MarkovLemma}, for any 
	$n_2 \geq n_1$
	\begin{equation}
		\sup_{x \in B_{H} (r) } 
		\norm{D_x f_{n_2} 
			- D_x f_{n_1}}_{\op}
	 \leq
		\frac{\sqrt{m} d^2}{r} 
		\sum_{k>n_1}
		\sup_{x \in B_{E_k}(r)}  
		\|p_k (x) \|	
		\to 0 ,
	\end{equation}
	as $n_2, n_1 \to \infty$. Therefore, $\{D f_n\}_{n \in \N}$ has a limit in $\mathcal{B}
	(
	B_{H}(r),  
	\mathcal{L} (H , \R^m )
	)$, for all $r>0$. It follows that $f$ is $\mathcal{C}^1$ and \eqref{DifferentialeSeriePolinomi} holds, see e.g.\ \cite[Thm.\ 9.1]{LangRealandFunct}.
	
	Thanks to \eqref{DifferentialeSeriePolinomi}, we prove that $Df : H \to \mathcal{L}(H,\R^m)$ is weakly continuous. 
	Let us consider a sequence $\{x^j\}_{j\in \N}$ weakly convergent to $x \in B_{H } (r) $; we prove that $D_{x^j} f \to D_x f$ as $j \to \infty$.
	Indeed, given $\varepsilon >0$ we fix $N_{\varepsilon} \in \mathbb{N}$ such that $\sum_{k>N_\varepsilon} \sup_{B_{E_k} (r)} \|p_k\| \leq  \varepsilon$ and, by \cref{MarkovLemma}, we get
	\begin{align}
			\norm{D_{x^j}f - D_xf}_{\op}
			& \leq
			\sum_{k=1}^{\infty}
			\norm{D_{x^j} (p_k \circ \pi_k) - D_{x} (p_k \circ \pi_k) }_{\op} \\
			& =
			\sum_{k=1}^{\infty}
			\norm{(D_{\pi_k (x^j)} p_k) \circ \pi_k - (D_{\pi_k (x)} p_k) \circ \pi_k }_{\op} \\
			& =
			\sum_{k=1}^{\infty}
			\norm{ D_{\pi_k (x^j)} p_k - D_{\pi_k (x)} p_k}_{\op} \label{eq:terzaoperator}\\
			& \leq
			\sum_{k=1}^{N_{\varepsilon}}
			\norm{ D_{\pi_k (x^j)} p_k - D_{\pi_k (x)} p_k}_{\op} 
			+
			\frac{\sqrt{m} d^2}{r} 
			2\varepsilon,
	\end{align}
	where with some abuse of notation we denoted by the same symbol $\|\cdot\|_{\op}$ either the norm in $\mathcal{L}(H,\R^m)$ or the one for the finite-dimensional subspaces $\mathcal{L}(E_k,\R^m)$. Since $\pi_k(x^j) \to \pi_k(x)$ as $j \to \infty$ for all $k \in \N$, we conclude by the continuity of $D p_k$ on $E_k$.

	{
	We prove now that $Df : H\to \mathcal{L}(H,\R^m)$ is locally Lipschitz. We will prove that the sequence $\{D f_n\}_{n\in \N}$ defined by \eqref{eq:dxfn} is uniformly Lipschitz on balls. For $k\in \N$, fix $z,w \in E_k$. Then it holds
	\begin{equation}\label{eq:shift}
	D_z p_k - D_w p_k = D_z p_k - D_z \tilde{p}_k  = D_z (p_k - \tilde{p}_k),
	\end{equation}
	where $\tilde{p}_k : E_k \to \R^m$, defined by $\tilde{p}_k(\cdot) := p_k(\cdot - z+ w)$, is polynomial depending on the fixed choices of $z,w$, with degree bounded by $d$. Thus for all fixed $z,w \in B_{E_k}(r)$, $r>0$, it holds
	\begin{align}
		\|D_z p_k - D_w p_k\|_{\op} & \leq \sup_{a \in B_{E_k}(r)} \| D_a (p_k - \tilde{p}_k)\|_{\op} & \text{by \eqref{eq:shift}} \\
		& \leq \frac{\sqrt{m}d^2}{r}\sup_{a\in B_{E_k}(r)} \|p_k(a) - \tilde{p}_k(a)\| & \text{by \cref{MarkovLemma}}\\
		& \leq \frac{\sqrt{m}d^2}{r}\sup_{a\in B_{E_k}(r)} \|p_k(a) - p_k(a-z+w)\|  & \text{by definition of $\tilde{p}_k$} \\
		& \leq \left(\frac{\sqrt{m}d^2}{r}\sup_{a\in B_{E_k}(3r)} \|D_a p_k\|_{\op}\right) \|z-w\| & \text{$p_k$ is Lipschitz}  \\
		& \leq \left(\frac{m d^4}{r^2}\sup_{a\in B_{E_k}(3r)} \|p_k(a)\|\right) \|z-w\|. & \text{by \cref{MarkovLemma}}\label{eq:lipestimate}
		\end{align}
		Note that if $x \in B_H(r)$ then $\pi_k(x) \in B_{E_k}(r)$. Thus, for all $x,y\in B_H(r)$ it holds
\begin{align}
\|D_x f_n - D_y f_n\|_{\op} & \leq \sum_{k=1}^n \|D_x (p_k\circ \pi_k)- D_y (p_k\circ \pi_k)\|_{\op} & \text{by \eqref{eq:dxfn}}\\
& \leq \sum_{k=1}^n \|D_{\pi_k(x)} p_k - D_{\pi_k(y)} p_k\|_{\op} & \text{by Leibniz rule} \\
& \leq   \frac{m d^4}{r^2} \sum_{k=1}^n  \sup_{a\in B_{E_k}(3r)} \|p_k(a)\| \|\pi_k(x)-\pi_k(y)\| & \text{by \eqref{eq:lipestimate}}\\
& \leq \frac{m d^4}{r^2}  \|x-y\|  \sum_{k=1}^\infty \sup_{a\in B_{E_k}(3r)} \|p_k(a)\|& \text{$\pi_k$ is $1$-Lipschitz}\\
& \leq c(r) \|x-y\|, 
\end{align}
where $c(r)>0$ does not depend on $n$, and in the last inequality we used \eqref{eq:seriesconv} for the convergence of the series. Therefore the sequence $\{D f_n\}_{n\in \N}$, is uniformly Lipschitz on every ball $B_H(r)$. Thus its limit $D f : H \to \mathcal{L}(H,\R^m)$ is Lipschitz on any ball $B_H(r)$, and a fortiori locally Lipschitz.
}
	
	Finally, we show that for every finite dimensional space $E\subset H$, the map $f|_E$ is a polynomial map with components of degree bounded by $d$. Let $n=\mathrm{dim}(E)$ and fix a linear isometry $L:\R^n\to E$. The statement is equivalent to show that $f\circ L$ is a polynomial map with components of degree bounded by $d$. Denoting by $(t_1, \ldots, t_n)$ the standard coordinates in $\R^n$, we have
	\begin{align}
	f(L(t_1, \ldots, t_n))& =\sum_{k= 1}^\infty p_k(x_1(t_1, \ldots, t_n) , \dots, x_k(t_1, \ldots, t_n))\\
	&=\sum_{k= 1}^\infty \tilde{p}_k(t_1, \ldots, t_n),
	\end{align}
	where each $\tilde{p}_k$ is a polynomial in $(t_1, \ldots, t_n)$ of degree at most $d$ (since the $x_i$ depend linearly on the $t_j$). Condition \eqref{eq:seriesconv} guarantees that, for each monomial $t_1^{\alpha_1}\cdots {t_n}^{\alpha_n}$, its coefficient in the expansion $\sum_{k\leq s}\tilde{p}_k$ converges as $s\to\infty$. Note that there are only a finite number of such coefficients, determined by the upper bound $d$ on the degree and the dimension $n$. We conclude by noting that the uniform limit of polynomials with degree  bounded by $d$ is a polynomial with degree bounded by $d$.
\end{proof}

With the following we exhibit a special class of maps, obtained via the construction of \cref{RegolarSeriePolinomiali}, for which the Sard property holds true \emph{globally}. This is the content of \cref{thm:main3intro} in the Introduction, of which we recall the statement here.

\begin{theorem}\label{thm:sardseries}
Let $H$ be a separable Hilbert space. For all $k \in \N$, let $p_k: E_k \to \R^m$ be polynomial maps with 
	$\sup_{k \in \N} \deg p_k \leq d$ for some $d \in \N$, and such that
	\begin{equation}\label{eq:boundball1}
\sup_{x \in B_{E_k} (1)} 
\|p_k (x)\|
\leq  q^{-k}, \qquad \forall\, k \in \N,
\end{equation}
for some $q>1$. Then the map $f: H \to \mathbb{R}^m$ defined by
\begin{equation}
f(x):= \sum_{k=1}^{\infty} p_k(x_1 , \dots ,x_k), \qquad \forall \, x \in H,
\end{equation}
is well--defined, $f \in \mathscr{P}_d^m(H)$ (see \cref{def:polyH}), and for all $\nu \leq m-1$ and $r>0$ it holds
\begin{equation}\label{eq:entropyestimatemaps}
	\dim_e \bigg(f\left( \crit_{\nu} (f) 
	\cap 
	B_{H}(r)  \right)\bigg)
	\leq 
	\nu + 
	 \frac{\ln \beta_0}{ \ln q},
\end{equation}
where $\beta_0 =\beta_0(d,m)>1$ is the same constant given by \cref{GeneralSardTheorem-intro}. In particular, if $q> \beta_0$, then $f$ satisfies the Sard property globally on $H$:
\begin{equation}
\mu\bigg(f\big(\crit (f) \big)\bigg)=0.
\end{equation}
\end{theorem}
\begin{proof}
We show first that the map $f$ is well--defined. The upper bound \eqref{eq:boundball1}, together with \cref{lem:1tor} imply that for all $r>0$
\begin{equation}\label{eq:upperboundallr}
\sup_{x \in B_{E_k} (r)} 
\|p_k (x)\|
\leq  \alpha(d,m)(1+r)^d  \frac{k^d}{q^k} , \qquad \forall\, k \in \N,
\end{equation}
with $q>1$. In turn, this implies that
\begin{equation}
\sum_{k=1}^\infty \sup_{x \in B_{E_k} (r)} 
\|p_k (x)\| < \infty, \qquad \forall\, r>0.
\end{equation}
We can then apply \cref{RegolarSeriePolinomiali} obtaining that $f \in \mathscr{P}_d^m(H)$.

Fix $r>0$. We now prove that $f$ satisfies the hypothesis of \cref{GeneralSardTheorem-intro} with $B = B_{H}(r)$ and polynomial approximating maps $f_n: E_n \to \R^m$
\begin{equation}
f_n(x):= \sum_{k=1}^n p_k(x_1,\dots,x_k), \qquad \forall\, x \in H.
\end{equation}
Note that, by our assumptions, $\sup_{n \in \N} \deg f_n \leq d$. It remains to check the validity of \eqref{ipotesiMainTheorem-intro}. Firstly, we have for all $n \in \N$
\begin{equation}
\sup_{x\in B} 
\|f(x) - f_n \circ \pi_{E_n} (x)\|
\leq \alpha(d,m) (1+r)^d
\sum_{k = n+1}^{\infty} \frac{k^d}{q^k},
\end{equation}
where we used \eqref{eq:upperboundallr}.

We now estimate the derivatives. We have
\begin{align}
\sup_{x \in 
			B }	\norm{\left(D_x f\right)|_{E_{n}} - D_{\pi_{E_n}(x)}f_n }_{\op} 
& \leq \sup_{x \in 
		B }	\norm{D_x f - D_{x}(f_n\circ\pi_{E_n}) }_{\op}  \\
	& \leq 
	\sup_{x \in 
		B }
	\sum_{k=n+1}^{\infty} 
	\norm{ D_{ x }  (p_k \circ \pi_k)  	}_{\op} & \text{by \eqref{DifferentialeSeriePolinomi}}\\
& =	\sup_{x \in 
		B }
	\sum_{k=n+1}^{\infty} 
	\norm{ (D_{ \pi_k(x) }  p_k) \circ \pi_k  	}_{\op} \\
	& \leq 
	\sum_{k=n+1}^{\infty}  \sup_{x \in 
		B_{E_k} (r) }
	\norm{ D_{ x }  p_k	}_{\op}.	
\end{align}
Note that, in the last line, the supremum is taken on a finite-dimensional Euclidean ball. To proceed, by \cref{MarkovLemma} and \eqref{eq:upperboundallr} we have 
\begin{equation}\label{eq:Markovinproof}
\sup_{x \in B_{E_k} (r)} 
\norm{D_x p_k }_{\op} 
\leq  
\frac{d^2 \sqrt{m}}{r}\alpha(d,m) (1+r)^d \frac{k^d}{q^k}.
\end{equation}
Therefore, continuing the previous estimate, we obtain
\begin{align}
\sup_{x \in 
			B }	\norm{\left(D_x f\right)|_{E_{n}} - D_{\pi_{E_n}(x)}f_n }_{\op}  & \leq \frac{d^2\sqrt{m}}{r}\alpha(d,m) (1+r)^d \sum_{k=n+1}^\infty \frac{k^d}{q^k} .
\end{align}
We have therefore proved that for all $n \in \N$, and $r>0$ it holds
\begin{equation}
		\sup_{x \in B}
		\bigg(\|f(x)  - f_n \circ \pi_{E_n} (x) \|
		+
		\|\left(D_x f\right)|_{E_{n}} - D_{\pi_{E_n}(x)}f_n \|_{\op} \bigg)
		\leq 
		\left(1+\frac{d^2\sqrt{m}}{r}\right)\alpha(d,m) (1+r)^d \sum_{k=n+1}^\infty \frac{k^d}{q^k}.
\end{equation}
To conclude, note that by elementary estimates it holds
\begin{equation}
\left(1+\frac{d^2\sqrt{m}}{r}\right)\alpha(d,m) (1+r)^d \sum_{k=n+1}^\infty \frac{k^d}{q^k}\leq c q^{-n},
\end{equation}
for sufficiently large $n$, where $c\geq 1$ is a constant depending only on the fixed parameters $r,d,m$. Therefore, the main assumption \eqref{ipotesiMainTheorem-intro} of \cref{GeneralSardTheorem-intro} is satisfied, yielding
\begin{equation}
		\dim_e\bigg( f\big( \crit_{\nu} (f) \cap B_H(r) \big)\bigg) 
		\leq 
		\nu + \frac{\ln \beta_0}{ \ln q}  , \qquad \forall\, \nu = 0,\dots,m-1.
	\end{equation}
In particular, if $q > \beta_0$ then $f$ satisfies the Sard property on $B_H(r)$, namely
\begin{equation}
\mu\bigg(f\big(\crit (f) \cap B_H(r)\big)\bigg)=0.
\end{equation}
Since $r>0$ is arbitrary, we obtain that $\mu(f(\crit(f)))=0$.
\end{proof}

\subsection{Sard threshold on compacts}

We can now prove \cref{thm:mainintro}, of which we recall the statement.
	\begin{theorem}\label{thm:main}
	For every $ d,m\in \mathbb{N}$, there exists $\omega_0(d,m)\in (0,1]$ such that
	\begin{enumerate}[(i)]
	\item \label{item:1main} for every $f\in \mathscr{P}_{d}^m(H)$ and for every compact set $K\subset H$ with $\omega(K,H)<\omega_0(d,m)$,
	\begin{equation}
	\mu\bigg(f\left(\mathrm{crit}(f)\cap K\right)\bigg)=0;
	\end{equation}
	\item \label{item:2main} for every $\omega>\omega_0(d,m)$, with $\omega \in (0,1]$, there exist $f\in \mathscr{P}_{d}^m(H)$ and a compact set $K\subset H$ with $\omega(K,H)=\omega$ and such that
	\begin{equation}
	\mu\bigg(f\left(\mathrm{crit}(f)\cap K\right)\bigg)>0.
	\end{equation}
	\end{enumerate}
	\end{theorem}
\begin{proof}
By definition, for any compact set $K$, it holds $\omega(K,H) \in [0,1]$, see \eqref{eq:defomegaKH}. We define $\omega_0(d,m)$ as the supremum of the set of $\omega\in [0,1]$ such that for every $f\in \mathscr{P}_{d}^m(H)$ and for every compact set $K\subset H$ with $\omega(K,H)=\omega$ it holds:
	\begin{equation}
	\mu\bigg(f\left(\mathrm{Crit}(f)\cap K\right)\bigg)=0.
	\end{equation}
By \cref{CorollaryOfGeneralSard-intro} this set is non-empty and $\omega_0(d,m)\geq \beta_0(d,m)^{-1}$. Hence $\omega_0(d,m)\in (0,1]$. Both items of \cref{thm:main} follow. Note that by \cref{thm:sintesiexamples-intro} (more specifically, the construction of \cref{ExampleOptimality3}), provided that $d\geq 3$, one must have $\omega_0(d,m)\leq (d-1)^{-1/d} < 1$, so that \cref{item:2main} is non-vacuous.
\end{proof}

\section{Applications to the Endpoint maps of Carnot groups}
We apply  in this last section the results from the previous sections to the study of the Sard property for Endpoint maps for Carnot groups.

	
\begin{definition}[Carnot groups]\label{CarnotGroups}
		An $m$-dimensional Carnot group of step $s \in \N$ is a connected and simply connected Lie group $(\gr, \cdot)$ of dimension $m$, whose Lie algebra of left-invariant vector fields $\alg$ admits a stratification of step $s$, that is 
	\begin{equation}\label{eq:stratification}
		\alg 
		= 
		\alg_1 \oplus \dots \oplus \alg_s,
	\end{equation}
	where 
	$\alg_i \neq \{0\}$, 
	$\alg_{i+1} = [\alg_1,\alg_i]$ 
	for all $i=1,\dots,s-1$ and $[\alg_1,\alg_s]= \{0\}$. We set $k_i=\dim \alg_i$, $a_i = k_1 + \dots +k_i$, and $k=k_1$, which is called the \emph{rank} of the Carnot group.
	\end{definition}
	
	We identify elements of $\alg$ with vectors of $T_e \gr$. Since $\alg$ is nilpotent and $\gr$ is simply connected, the group exponential map $\exp_{\gr} : \alg \to \gr$ 
	is a smooth diffeomorphism. 
	
	The stratification \eqref{eq:stratification} allows the definition of a family of automorphisms 
	$\delta_\lambda : \gr \to \gr$, for $\lambda>0$, called \emph{dilations} such that
	\begin{equation}
	\delta_\lambda(\exp_{\gr} (v)) = \exp_{\gr}(\lambda^i v),\qquad \forall\, v\in \alg_i,\quad i=1,\dots,s.
	\end{equation}	
		
	Fix an \emph{adapted basis} of $\alg$, namely a set of vectors $v_1, \dots , v_m \in \alg$  such that for all $i= 1, \dots, s$ the list of vectors $v_{k_{i-1} + 1}, \dots , v_{k_i}$ is a basis of $\alg_i$, with the convention $k_0=0$.
	For any $j=1, \dots,m$ the \emph{weight} of $v_j$ is the unique 
	$w_j \in \{1,\dots,s\}$ such that $v_j \in \alg_{w_j}$. We define a diffeomorphism
	$\varphi_{\gr}: \R^m \to \gr$, by
	\begin{equation}
		\varphi_{\gr} (x_1 , \dots , x_m) := 
		\exp_{\gr} \left( \sum_{i=1}^m x_i v_i \right).
	\end{equation}
	The map $\varphi_{\gr}$ defines global coordinates on $\gr$, called \emph{exponential coordinates} (of first kind).
 	For any $x \in \gr $ we denote by 
	$\tau_x: \gr \to \gr $ the left translation on $\gr$, that is $\tau_x(y) = x \cdot y$. We denote by $X_1 , \dots , X_m$ the left-invariant vector fields on $\gr$ corresponding to the adapted basis. They are given by
	$    X_j(x)= D_e \tau_x (v_j) $
	for any $x \in \gr$,
	where $D_e \tau_x $ denotes the differential 
	of $\tau_x$
	at the unit element.

We recall the definition of (homogeneous) polynomials on Carnot groups. We say that $P:\gr \to \R$ is a polynomial if for some (and then any) set of exponential coordinates $\varphi_{\gr}$, the map $P\circ \varphi_{\gr} :\R^m \to \R$ is a polynomial. Any polynomial $P:\gr \to \R$ can be written as
	\begin{equation}\label{Poly_linearCombMono}
		P(\varphi_{\gr}(x_1,\dots,x_m)) = \sum_{\alpha \in \N^m} c_{\alpha} x_1^{\alpha_1} \cdots x_m^{\alpha_m} , \qquad \forall\, x\in \R^m,
	\end{equation}
	for some $c_{\alpha} \in \R$, which are non-zero only for a finite set of multi-indices. For $\alpha \in \N^m$, we denote by
	\begin{equation}
	|\alpha|_{\gr}:=\sum_{i=1}^m w_i \alpha_i,
	\end{equation}
	the weighted degree of the multi-index $\alpha$. The weighted degree of a polynomial $P$ is
	\begin{equation}
	\deg_{\gr} (P) :=\max \left\{|\alpha|_{\gr}\,\middle\vert\, c_\alpha \neq 0\right\}.
	\end{equation}
	A polynomial $P$ is homogeneous of weighted degree $w$ if $P (\delta_\lambda x) = \lambda^w P(x)$ for any $x \in \gr$ and $\lambda >0$. 
	\begin{remark}\label{rmk:homdegree}
In other words, $P$ is homogeneous of weighted degree $w$ if and only if
	\begin{equation}
	P(\varphi_{\gr}(x_1,\dots,x_m)) = \sum_{\substack{\alpha \in \N^m \\ |\alpha|_{\gr}=w}} c_{\alpha} x_1^{\alpha_1} \cdots x_m^{\alpha_m}.
	\end{equation}
	In particular any such polynomial depends only on the variables $x_i$ with $w_i \leq w$, namely $x_1,\dots,x_{k_{w}}$.
	\end{remark}
	
	It is easy to see that the concepts we just introduced do not depend on the choice of exponential coordinates. For this reason, in the following, we fix such a choice, and we omit $\varphi_{\gr}$ from the notation, identifying $\gr \simeq \R^m$.	Note that $e\in \gr$ is identified with $0\in \R^m$.

We recall from \cite[Prop.\ 5.18]{Bellaiche} the following result on the structure of the vector fields $X_j$.
\begin{proposition}\label{VectorFields_structure}
	In exponential coordinates $(x_1,\dots,x_m)\in \R^m$, the vector fields $X_j$, for $j=1,\dots,k$ have the following form:
	\begin{equation}\label{struttura_campi}
		X_j (x_1 , \dots , x_m) = 
		\partial_j 
		+
		\sum_{h >k} 
		Q_{jh}(x_1 , \dots, 
		x_{ k_{w_h - 1 }} ) 
		\partial_h  ,
	\end{equation}
	where $Q_{jh}: \gr \simeq \R^m \to \R $ is a homogeneous polynomial of weighted degree $w_h - 1$.
\end{proposition}

\subsection{Endpoint maps}
	In this section we introduce the Endpoint map associated to a Carnot group. We work with the separable Hilbert space
	\begin{equation}
	H:=L^2(I, \R^k),
	\end{equation}
	for some fixed interval $I$, say $I=[0,1]$, and where $k$ is the rank of $\gr$. For any $u \in H$ with $u=(u_1,\dots,u_{k})$ we consider the following Cauchy problem
	for a curve $\gamma: I \to \gr$:
	\begin{equation}\label{system}
			\dot{\gamma}(t)=  \sum_{j=1}^{k}  u_j(t) X_j(\gamma(t)), \qquad 			\gamma(0)=e.
	\end{equation}
It is well--known that \eqref{system} admits a unique absolutely continuous solution, $\gamma_u : I \to \gr$, for any $u \in H$.

\begin{definition}\label{def:endpoint}
Let $\gr$ be a Carnot group. The \emph{Endpoint map} of $\gr$ is the function
\begin{equation}
\End: H \to \gr,
\end{equation}
given by $\End(u):=\gamma_u(1)$.  We call the function $u\in H$ a \emph{control}. We use the same notation to denote the Endpoint map $\End:H\to \R^m\simeq \gr$ with an identification in exponential coordinates. 
\end{definition}

Thanks to the special form of the vector fields $X_j$ given by \cref{VectorFields_structure}, we can rewrite the system \eqref{system} in exponential coordinates. Adopting the notation  $\gamma= (\gamma_{1} , \dots , \gamma_m)$, we have 
\begin{equation}\label{eq:systemcomponents}
\begin{cases}
\gamma_i(0) =0 & i=1,\dots,m, \\
\dot{\gamma}_i(t) = u_i(t) & i=1,\dots,k, \\
\dot{\gamma}_i(t) =\displaystyle \sum_{j=1}^k u_j(t) Q_{ji}(\gamma_1(t),\dots,\gamma_{k_{w_i-1}}(t)) & i=k+1,\dots,m ,
\end{cases}
\end{equation}
for a.e.\ $t\in I$. Similarly, we write $\gamma_u=(\gamma_{u,1},\dots,\gamma_{u,m})$ for the solution to \eqref{eq:systemcomponents}.

The next result (\cref{RestrictionEndpointPolynomial-intro} in the Introduction) connects this framework to the previous sections.

\begin{proposition}\label{RestrictionEndpointPolynomial}
	Let $\gr$ be a Carnot group of topological dimension $m$, step $s$, and rank $k$. Then the Endpoint map $\End\in \mathscr{P}_s^m(H)$, for $H=L^2(I,\R^k)$.
\end{proposition}
\begin{proof}The fact that $\End$ is $\mathcal{C}^1$ (in fact smooth) and that the differential $D\End$ is locally Lipschitz follows from \cite[Prop.\ 8.5]{ABB}. The fact that the $D\End$ is weakly continuous is proved in \cite[Thm.\ 23]{BoarLer17}. It remains to show that for all finite dimensional linear subspaces $E$ of $H$, the restriction $\End|_E:  E \to \R^m$ is a polynomial map of degree at most $s$.

We actually prove a stronger claim: for any $n \in \N$, for any $n$-dimensional linear subspace $E \subset L^2(I, \R^k)$, for any $t \in I$ and $i=1,\dots,m$, the function on $E$ defined by 
$ E\ni u \mapsto \gamma_{u , i } (t)$ is a polynomial of degree $w_i$, and its coefficients (when $u$ is written in terms of some basis of $E\simeq \R^n$) are continuous functions of $t \in I$.
We prove the claim by induction on $w$. We prove the base case, that is $w_i =1$. In this case $i=1,\dots,k$. For these values of $i$, by \eqref{eq:systemcomponents} we obtain
\begin{equation}
\gamma_{u,i}(t) = \int_0^t u_i(s) \,ds,
\end{equation}
which is clearly a polynomial function of $u\in E$, of degree $1$. Moreover its coefficients (when $u$ is written in terms of a basis of $E\simeq \R^n$) are continuous functions of $t\in I$. This concludes the base case.

Let then $1\leq \theta <s$. The induction assumption is that for any $t \in I$ the function $E\ni u \to \gamma_{u , i} (t)$ is a polynomial of degree $w_i$ for all $i$ such that $w_i \leq \theta$, with coefficients that depends continuously on $t$.

Now we prove that the above property holds for any $i$ such that $w_i = \theta +1$. In particular $i > k$. In this case, by \eqref{eq:systemcomponents}, the curve $\gamma_{u , i}:I \to \R$ is determined by
\begin{equation}\label{eq:integrating}
\gamma_{u,i}(t) = \sum_{j=1}^k \int_0^t u_j(s) Q_{ji}(\gamma_{u,1}(s),\dots,\gamma_{u,k_{w_i-1}}(s))\, ds.
\end{equation}
Recall that $Q_{ji}$ is a homogeneous polynomial of weighted degree $w_i-1$ and as already highlighted in the notation, depends only on the functions $\gamma_{u,\ell}(s)$ with $w_{\ell} \leq w_i-1 = \theta$. By the induction assumption, all the functions $E\ni u \mapsto \gamma_{u,\ell}(s)$ are polynomials of degree $w_\ell\leq \theta$, depending continuously on $s$. By using the definition of weighted degree (see \cref{rmk:homdegree}), it follows that the function
\begin{equation}
E\ni u \mapsto Q_{ji}(\gamma_{u,1}(s),\dots,\gamma_{u,k_{w_i-1}}(s)),
\end{equation}
is a polynomial of degree $w_i-1$, with coefficients depending continuously on $s\in I$. Then, the integrand in \eqref{eq:integrating}, namely the function
\begin{equation}
E\ni u \mapsto \sum_{j=1}^k u_j(s) Q_{ji}(\gamma_{u,1}(s),\dots,\gamma_{u,k_{w_i-1}}(s)),
\end{equation}
is a polynomial of degree $w_i=\theta +1$, with coefficients depending continuously on $s\in I$. We conclude easily using \eqref{eq:integrating}.
\end{proof}

Thanks to \cref{RestrictionEndpointPolynomial-intro} we can apply the results of \cref{sec:Sardinfinite} to Endpoint maps of Carnot groups, obtaining \cref{RisultatoEndpointGenerale-intro} in the Introduction.

\subsection{Sard property for real--analytic controls}\label{sec:Carnotrealanal}
We specialize the result in \cref{RisultatoEndpointGenerale-intro} to some concrete set of controls $K$, estimating its $n$-width.

	Fix a closed interval $I\subset \R$ and let 
	$\mathcal{D}(r,I) 
	:= 
	\{z \in \C \mid d(z,I) \leq r \}$.
	The radius of convergence of a real--analytic function $u:I \to \R$ is the largest 
	$r \in (0, \infty]$ such that $u$ extends to a (unique) holomorphic function on $\mathcal{D}(r, I) $.
	For fixed $r>0$ we denote by $\mathcal{C}^\omega(I,\R^k;r)\subset L^2(I, \R^k)$ the set of real--analytic controls whose components have radius of convergence strictly greater than $r$:
	\begin{equation}
	\mathcal{C}^\omega(I,\R^k;r):=\bigg\{u \in L^2(I,\R^k) \,\bigg|\, 
	\textrm{$u$ is real--analytic with radius of convergence 
	strictly greater than } r \bigg\},
	\end{equation}
	which we endow with the supremum norm:
	\begin{equation}
	\|u\|_{\mathcal{C}^\omega(I,\R^k;r)}:=
	\sup_{z\in 	\mathcal{D}(r, I)}\|u(z)\|.
	\end{equation}
Similarly, we consider finite concatenations of real--analytic controls as above. More precisely, for $\ell\in \N$ divide the interval $I$ in equispaced sub-intervals:
\begin{equation}\label{eq:decomposition}
I = I_1 \cup \dots \cup I_\ell, \qquad I_a = \inf I+ \left[\frac{(a-1)|I|}{\ell} ,\frac{a |I|}{\ell}\right],
\end{equation}
and we let $\mathcal{C}^\omega(I,\R^k;r,\ell)\subset L^2(I, \R^k)$ be the set of piecewise analytic controls $u: I\to \R^k$ that are real--analytic on each set $I_1,\dots,I_\ell$, with radius of convergence strictly greater than $r$, namely 
\begin{equation}
\mathcal{C}^\omega(I,\R^k;r,\ell) = \bigg\{u \in L^2(I,\R^k) \,\bigg|\,  u|_{I_a} \in \mathcal{C}^\omega(I_a,\R^k;r) \subset L^2(I_a,\R^k),\,\text{ for all $a=1,\dots,\ell$} \bigg\},
\end{equation}
endowed with the norm
\begin{equation}\label{eq:normComegaell}
\|u\|_{\mathcal{C}^\omega(I,\R^k;r,\ell)} = \max_{a = 1,\dots,\ell} \|u\|_{\mathcal{C}^\omega(I_a,\R^k;r,\ell)}.
\end{equation}

	\begin{proposition}\label{WidthAnalitiche}
Fix $I=[0,1]$, $\ell\in \N$ and $r>1$. Let $Y = \mathcal{C}^\omega(I,\R^k;r,\ell)$. Then the set $B_{Y} := \{u\in \V \mid \|u\|_{Y}\leq 1\}$ is precompact in $L^2(I,\R^k)$, and for its $n$-width it holds:
	 \begin{equation}
			\Omega_n
			( 
			B_{Y}, L^2(I,\R^k) )
			\leq
			\frac{ 
				( k\ell)^{1/2} 
									}
				{  \ln r
									}
			\left(
			\frac{1}{r}
			 \right)^
			 {\floor*{\frac{n}{k\ell}}}.
		\end{equation}
	\end{proposition}
\begin{proof}
We start by proving the case $\ell=1$. Fix $u=(u_1 , \dots , u_k) \in B_{Y}$. Since $r>1$, any $u_j$ for $j=1 , \dots , k$ has an expansion in power series on the complex unit ball centered at the origin, denoted by $B_{\C}(1)$. In particular, we have for any $t \in I$ 
\begin{equation}
	u_j(t)
	=
	\sum_{h=0}^\infty
	\frac
	{u^{(h)}_j (0)}
	{h!}	
	t^h.
\end{equation}
For any $N \geq 1$ we consider  the truncated sum, given for any $t \in I$ as
\begin{equation}
	u_{j,N} (t)
	=
	\sum_{h=0}^{N-1}
	\frac
	{u^{(h)}_j (0) }
	{	h!	} 
	t^h  .
\end{equation}
By the Cauchy integral formula, and since $r>1$, we have
\begin{equation}
	u_j^{(h)}(0)
	=
	\frac{h!}{2\pi i}
	\int_{\partial B_{\C}(r)}
	\frac{u_j(z)}{z^{h+1}}
	\,
	dz,
\end{equation}
where $B_{\C}(r)$ is the complex unit ball of radius one centred at zero. Therefore,
\begin{equation}
	|u_j^{(h)} (0) |
	\leq
	\frac{h!}{r^h}  .
\end{equation}
Hence we have
\begin{equation}
	\|u_j - u_{j,N}\|_{L^2(I)}
	\leq 
	\norm{u_j - u_{j,N}}_{L^{\infty}(I)}
	\leq 
	\sum_{h = N }^{ \infty}
	\frac
	{\vert u_j^{(h)}(0) \vert}
	{h!}
	\leq
	\sum_{h= N }^{ \infty}
	\frac
	{1}{r^h}
	\leq
	\int_{N}^{\infty}
	\frac{1}{r^x}
	\,
	dx
	=
	\frac{r^{-N}}{\ln r}.
\end{equation}
It follows that for all $u\in B_{Y}$ it holds
\begin{equation}\label{EstimateWidth_1}
\|u - u_N\|_{L^2(I, \R^k)} \leq \left( \sum_{j=1}^k \|u_j-u_{j,N}\|^2_{L^2(I)} \right)^{1/2} \leq 
	\frac{k^{1/2} r^{- N}}{\ln r}.
\end{equation}
Thanks to this result we can find an approximating subspace. Consider the standard basis $e_1 , \dots , e_k$ of $\R^k$, and
for any $k \in \N$ we 
denote by $f_h \in L^2(I)$ the function $f_h(t) = t^h$. We have 
\begin{equation}
	u_N
	=
	\sum_{h=0}^{N-1}
	\frac
	{u^{(h)} (0) }
	{	h!	} 
	f_h
	=
	\sum_{j=1}^k
	\sum_{h=0}^{N-1}
	\frac
	{u^{(h)}_j (0) }
	{	h!	} 
	f_h \otimes e_j  .
\end{equation}
In particular, 
$u_N \in V_N
:=
\mathrm{span} \{  f_h \otimes e_j \mid j=1,\dots , k,\;
		h= 0, \dots , N-1\}
$ and $\dim (V_N) = k N$.
Hence from \eqref{EstimateWidth_1} it follows that for any $N\in \N$
\begin{equation}\label{EstimateWidth_2}
	\Omega_{ k N }
	( 
	B_{Y}, L^2(I,\R^k) )
	\leq
	\frac{
		k^{1/2}
	}
	{  \ln r
	}
	\left(
\frac{1}{r}
\right)^{N}.
\end{equation}
For $n \in \N$ we consider $N =\floor*{\frac{n}{k}} $. 
By \cref{def:nwidth}, the $n$-width is non-increasing as a function of $n\in \N$, hence by \eqref{EstimateWidth_2} we have
\begin{equation}
\Omega_n 
( B_{Y}, L^2(I,\R^k) )
\leq 
\Omega_{k N}
( B_{Y}, L^2(I,\R^k) )
\leq 
\frac{
	k^{1/2}
}
{  \ln r
}
\left(
\frac{1}{r}
\right)^{\floor*{\frac{n}{k}}}.
\end{equation}
This proves the estimate on the $n$-width for $\ell=1$.

We sketch the argument for general $\ell$. In this case, for any $u \in \mathcal{C}^\omega(I,\R^k;r,\ell)$ we can write
\begin{equation}
u = \sum_{a=1}^\ell u|_{I_a} \mathbbm{1}_{I_a},
\end{equation}
where $I_a$ are the intervals of the decomposition \eqref{eq:decomposition}, and $\mathbbm{1}_{I_a}$ are the corresponding characteristic functions. By definition, each $u|_{I_a}$ is the restriction to $I_a$ of a real--analytic function with radius of convergence $r_a\geq r>1$. We can expand each $u|_{I_a}$ in Taylor series centered at the lower bound of $I_a$. Then, in order to obtain the finite-dimensional approximation, we truncate the series as in the previous case, repeating analogous estimates for the remainder (taking into account the length of the intervals $|I_a| = 1/\ell$). This concludes the estimate on the $n$-width for general $\ell$.

To prove the compactness assertion, observe that there is a linear immersion (given by the inclusion) $\mathcal{C}^\omega(I,\R^k;r,\ell) \hookrightarrow L^2(I,\R^k)$ and it holds $\|u\|_{L^2(I,\R^k)} \leq \|u\|_{C^\omega(I,\R^k;r,\ell)}$ for all $u\in C^\omega(I,\R^k;r,\ell)$. In particular $B_{Y}$ is bounded in $L^2(I,\R^k)$ and since its $n$-width tends to zero as $n\to \infty$ it is also precompact by \cref{PropoPinkus}.
\end{proof}

We can now prove \cref{thm:main4intro} in the Introduction, of which we recall the statement.
\begin{theorem}\label{thm:main4}
Let $\gr$ be a Carnot group of topological dimension $m$, step $s$, and rank $k$. Let $I=[0,1]$. For $\ell\in\N$, there exists $r=r(m,s,k,\ell)>1$ such that
\begin{equation}
\mu\big(\End(\mathrm{Crit}(\End)\cap \mathcal{C}^\omega(I,\R^k;r,\ell))\big)= \mu\big(\End(\mathrm{Crit}(\End|_{\mathcal{C}^\omega(I,\R^k;r,\ell)})\big) =0.
\end{equation}
Namely, the Sard property holds on the space of piecewise real--analytic controls with radius of convergence $>r$ and with $\ell$ pieces. 
Furthermore it holds
\begin{equation}
\dim_{\mathcal{H}}\big(\End(\mathrm{Crit}(\End)\cap \mathcal{C}^\omega(I,\R^k;\infty,\ell))\big) \leq \dim_{\mathcal{H}}\big(\End(\mathrm{Crit}(\End|_{\mathcal{C}^\omega(I,\R^k;\infty,\ell)})\big) \leq m-1.
\end{equation}
Namely, the strong Sard property holds on the space of piecewise entire controls with $\ell$ pieces.
\end{theorem}
\begin{proof}
Observe that $\End\in \mathscr{P}^m_s(L^2(I,\R^k))$ by \cref{RestrictionEndpointPolynomial-intro}. Let $ Y: =\mathcal{C}^\omega(I,\R^k;r,\ell)$. Let $B_{Y}$ be the unit ball of $Y$ with respect to the norm \eqref{eq:normComegaell}, namely $B_{Y} = \{u\in Y \mid \|u\|_{Y}\leq 1\}$. By \cref{WidthAnalitiche,PropoPinkus}, for any $r>1$ we have that $K:= \overline{B_{Y}}$ is compact (we stress that the closure is taken in the $L^2(I,\R^k)$ topology, and not w.r.t.\ the norm $\|\cdot\|_{Y}$), and it holds
\begin{equation}
\limsup_{n\to \infty}
\Omega_{n}(K, L^2(I, \R^k))^{1/n} = \limsup_{n\to \infty}
\Omega_{n}( B_{Y}, L^2(I, \R^k))^{1/n} \leq r^{-\tfrac{1}{k\ell}},
\end{equation}
where we used \cref{thmPinkus} for the first equality.
 We can apply \cref{i:RisultatoEndpointGenerale-intro_1} of \cref{RisultatoEndpointGenerale-intro}: for $r > r(m,s,k,\ell):=\beta_0(s,m)^{k\ell}$ it holds
\begin{equation}
\mu\big(\End(\mathrm{Crit}(\End)\cap \mathrm{span}(K))\big)= \mu\big(\End(\mathrm{Crit}(\End|_{\mathrm{span}(K)})\big) =0.
\end{equation}
By \cref{lem:restrandclosures}, the above equalities remain valid replacing $X:=\mathrm{span}(K)$ with any linear subspace of the latter, dense in it w.r.t.\ the $L^2(I,\R^k)$ topology. In particular, this is the case for $Y$, observing that $Y=\mathrm{span}(B_Y)$ and thus $\bar{Y} = \overline{\mathrm{span}(B_Y)}\supseteq \mathrm{span}(\overline{B_Y})=\mathrm{span}(K)$.

The same argument via \cref{i:RisultatoEndpointGenerale-intro_2} of \cref{RisultatoEndpointGenerale-intro} yields the statement for the space $\mathcal{C}^\omega(I,\R^k;\infty,\ell)$.
\end{proof}

\subsection{Surjectivity of the Endpoint map on finite-dimensional spaces of controls}\label{sec:surjectivity}

The Endpoint maps of sub-Riemannian manifolds are surjective when restricted to piecewise constant controls: this follows from the proof of the Rashevskii-Chow theorem in \cite{ABB}.
In this section, we prove that in Carnot groups the Endpoint maps are surjective also when they are restricted to the space of controls 
which consists in the set of polynomial maps of some large enough fixed degree (which depends on the Carnot group). The proof is obtained by using a quantitative version of the inverse function theorem, as it can be found in \cite{ClarkeInvFunc}. For completeness we present and prove in our setting the statements we need.

Given $M  \in \mathcal{L} (\R^m , \R^m)$ we denote by $\sigma(M)$ the smallest singular value of $M$, namely
\begin{equation}
	\sigma(M) := \inf_{ \norm{v} = 1} \norm{M v} .
\end{equation}  
From the definition it follows that for all $M_1,M_2 \in \mathcal{L} (\R^m , \R^m)$ it holds
\begin{equation}\label{WeylIneqFinDim}
| \sigma(M_1) - \sigma(M_2) |
\leq 
\norm{M_1 - M_2}_{\op} .
\end{equation}
From \cite[Lemma 3]{ClarkeInvFunc} we obtain the following technical lemma.
\begin{lemma}\label{Lemma3Clarke}
	Let $M_0 \in \mathcal{L}( \R^m , \R^m)$  be invertible. Then, for every $v \in \R^m$ with $\|v\|=1$ there exists $w \in \R^m$ with $\|w\| = 1$ such that 
	\begin{equation}
		\langle w , M v \rangle
		\geq 
		\frac{ \sigma (M_0) }{2},
	\end{equation}
	for all $M$ such that $\|M-M_0\|_{\op} \leq \tfrac{\sigma(M_0)}{2}$.
\end{lemma}
\begin{proof}
Given $v \in \R^m$ with $\|v\|=1$, the set $C_v \subset 
\R^m$
\begin{equation}
	C_v 
	:=
\left \{ 
Mv \in \R^m \,\middle\vert\, 
\|M-M_0\|_{\op} \leq \frac{\sigma(M_0)}{2}
\right \},
\end{equation}
is convex. We now prove that its distance from $0$ is at least $\frac{ \sigma( M_0) }{ 2 } $.
Indeed, by \eqref{WeylIneqFinDim}, for any $M$ with $\|M-M_0\|_{\op} \leq \tfrac{\sigma(M_0)}{2}$	we obtain 
$\sigma(M) > \frac{\sigma(M_0)}{2}$. Therefore,
\begin{equation}
\|Mv\| 
\geq
\inf_{\|v\|=1 } \|Mv\|
=
\sigma(M)
>
\frac{ \sigma( M_0) }{ 2 }  .
\end{equation}
Hence, we have proved that $C_v$ is separated from the ball $B_{\R^m} \left( \frac{\sigma(M_0)}{2} \right)$.
We conclude the proof thanks to 
the usual separation theorem for convex sets. Indeed, it directly provides
$w \in \R^m$ with $\|w\| = 1$ such that for all $M$ with $\|M-M_0\|_{\op} \leq \tfrac{\sigma(M_0)}{2}$ we have
\begin{equation}
	\langle w , M v \rangle
	\geq 
	\frac{ \sigma (M_0) }{2}  ,
\end{equation}
concluding the proof.
\end{proof}
From \cite[Lemma 4]{ClarkeInvFunc} we obtain the following ``quantitative injectivity'' lemma.
\begin{lemma}\label{Lemma4Clarke}
Let $W\subset \R^m$ be an open set, and $f \in \mathcal{C}^1 (W , \R^m)$. Assume $x_0 \in W$ is such that $D_{x_0} f$ is invertible. Set $\sigma_{f , x_0} := \sigma ( D_{x_0} f ) >0$.
There exists $r_{f, x_0} >0$ such that for any $g \in C^1 ( W, \R^m)$ with
\begin{equation}\label{IpotesiInverseFunQuantit}
	\sup_{z \in W} 
	\norm{ D_z f - D_z g}_{\op}
	< 
	\frac{ \sigma_{f,x_0} }{ 4 } ,
\end{equation}
and all $x , y \in B_{\R^m} (x_0 , r_{f, x_0})$ it holds
\begin{equation}\label{InvQuantFun4}
\| g(x) - g(y) \|
	\geq
	\frac{\sigma_{ f , x_0 }}{2} \| x - y \| .
\end{equation}
\end{lemma}
\begin{proof}
Since $f$ is $\mathcal{C}^1$ there exists $r = r_{f , x_0} >0$ such that  $B_{\R^m}(x_0,r) \subset W$ and
\begin{equation}
\sup_{ z \in B_{\R^m} (x_0, r) }
\norm{D_z f - D_{x_0} f }_{\op}
<
\frac{\sigma_{ f , x_0 }}{4}.
\end{equation}
Let $g \in\mathcal{C}^1(W,\R^m)$ as in the statement. By \eqref{IpotesiInverseFunQuantit} we get 
\begin{equation}\label{InvQuantFun1}
	\norm{ D_z g - D_{x_0} f }_{\op}
	<
	\frac{\sigma_{ f , x_0 }}{ 2 } ,
\end{equation}
for all $z \in B_{\R^m} (x_0, r)$.
Now we fix $x \neq y$ in $B_{\R^m} (x_0, r)$, and we have
\begin{equation}
	g(x) - g(y)
	=
\|x - y\|
	\int_{0}^{1}
	D_{(1-t) y + t x}  g 
	\cdot 
	\frac{ x-y}{\|x - y\|}  
	\,
	dt  .
\end{equation}
We define $v_{x,y}:= \frac{x-y}{\|x - y\| }\in\R^m $, and we apply \cref{Lemma3Clarke} to $v=v_{x,y}$ and
$M_0=D_{x_0} f$, which is invertible by hypothesis. We denote by $w_{x,y}\in\R^m$ the provided unit vector, such that
\begin{equation}\label{InvQuantFun2}
	\left\langle  
	w_{x,y }  ,  D_z g 
	\cdot
	\frac{x - y}{\|x-y\|} 
	\right\rangle
	\geq 
	\frac{ \sigma_{ f , x_0 }}{2} ,
\end{equation}
 for all $z \in B_{\R^m} (x_0, r)$. We conclude the proof by the Cauchy-Schwarz inequality, which gives
\begin{equation}\label{InvQuantFun3}
\|g(x) - g(y)\|
	\geq
\|x-y\|
	\left\langle
	w_{x,y}  , 
	\int_{0}^{1}
	D_{(1-t) y + t x}  g 
	\cdot 
	\frac{ x-y}{ \|x - y\|}  
	\,
	dt
	\right\rangle
	\geq
	\frac{\sigma_{ f , x_0 }}{2} \|x-y\| ,
\end{equation}
concluding the proof.
\end{proof}

We conclude with the following inclusion, analogous to \cite[Lemma 5]{ClarkeInvFunc}, providing the ``quantitative surjectivity'' counterpart of the previous statement. We report the proof for completeness.
\begin{lemma}\label{Lemma5Clarke}
In the same setting of \cref{Lemma4Clarke}, the following inclusion holds
\begin{equation}\label{eq:ballscontained}
	g(B_{ \R^m} (x_0 , r_{f, x_0}) ) 
	\supset 
	B_{ \R^m } 
	\left( g(x_0) ,
	 \frac{ r_{f, x_0} \sigma_{f , x_0}}{8} \right ) .
	\end{equation}
\end{lemma}
\begin{proof}
Set $r= r_{f, x_0}$ and $\sigma = \sigma_{f, x_0}$. Let $y\in B_{ \R^m } 
\left( g(x_0) ,
 \frac{ r\sigma}{8} \right )$. Let $x$ be a minimum of $\|y - g(\cdot)\|^2$ on $B_{ \R^n} (x_0 , r)$. We claim that $ x \in \mathrm{int} B_{ \R^n} (x_0, r)$. Otherwise, using \cref{Lemma4Clarke}, we have
\begin{align}
	\frac{r \sigma}{8} 
	&
	\geq \|y - g(x_0)\| \\
	& \geq \|g(x) - g(x_0)\| - \|y - g(x)\| \\
	& \geq 
	\frac{\sigma}{2} \|x - x_0 \| - \|y - g(x)\| \\
	& \geq 
	\frac{\sigma r}{2}  - \|y - g(x_0)\| \\
	& \geq \frac{\sigma r}{2}  - \frac{\sigma r}{8}
	= \frac{3 r \sigma}{8}  ,
\end{align}
which is a contradiction. Thus $x$ yields a local minimum for the function $\|y - g(\cdot)\|^2$, and consequently 
$\nabla_x \|y - g(\cdot)\|^2 = 
- 2 D_x g \cdot (y - g(x)) = 0$. Note that by \eqref{InvQuantFun1} and the fact that the singular values are $1$-Lipschitz (see \cref{SingolarValuesLipschitz}) it follows that $\sigma(D_x g) > \tfrac{\sigma}{2}$, hence $D_x g$ is invertible. Hence $y = g(x)$.
\end{proof}

We can now state the surjectivity property of the Endpoint maps of Carnot groups.
\begin{proposition}\label{thm:polysurj}
	For every Carnot group $\gr$ there exists $d_{\gr} \in \N$ such that 
	\begin{equation}
		\End \left(\{ u \in L^2(I , \R^k) \mid  
		u \text{ is a polynomial map with components of degrees} \leq d_\gr \} \right) =  \gr  .
	\end{equation}
	\end{proposition}
\begin{proof}
The classical proof of the Rashevskii-Chow theorem provides $u_0 \in L^2(I , \R^k) $ such that $D_{u_0} \End$ is surjective and $\End( u_0 )  =  0$  (see e.g.\ \cite[Sec.\ 3.2]{ABB}). 
In particular there exist
$w_1 , \dots , w_m \in \ker (D_{u_0}\End) ^\perp$ 
such that the vectors 
$
D_{u_0} \End (w_1),
\dots,
D_{u_0} \End (w_m)
$ are linearly independent. 
Now we consider the map $f: \R^m \to \gr \simeq \R^m$  defined as
\begin{equation}
f(s_1 , \dots , s_m) 
= 
\End
\left (u_0 + \sum_{i=1}^{m}  s_i w_i \right )  .
\end{equation}
The differential $D_0 f $ is invertible by construction, in particular $f$ covers a neighbourhood of $f(0)=0$. The idea of the proof is to construct a perturbation $g: \R^m \to \gr \simeq \R^m$ of the form
\begin{equation}
	g( s_1 , \dots , s_m) 
	= 
	\End
	\left (q_0 + \sum_{i=1}^{m}  s_i p_i \right )  ,
\end{equation}
where $q_0 , p_1 , \dots , p_m \in L^2(I,\R^k)$ are suitably chosen polynomial maps, in such a way that the image of $g$ stills contains a neighbourhood of $0$. This will be done by applying \cref{Lemma5Clarke}, as we now explain.

Given 
$\varepsilon_1 , \eta_1 >0 $
by the density of polynomial maps in $L^2(I , \R^k)$ and by the continuity of $\End$ and $D\End$,
there exists  a polynomial map
$q_0 \in L^2(I , \R^k)$ such that 
\begin{equation}\label{ControllabilitaPoli2}
	\norm{\End(q_0)} < \varepsilon_1\qquad
	\text{	and }\qquad
	\norm{D_{u_0} \End - D_{q_0} \End  }_{\op} < \eta_1  .
\end{equation}
Now we prove that we can find polynomial maps  
$p_1 , \dots , p_m$ 
and a neighbourhood $W\subset\R^m$ of $0$  such that, setting $\sigma_{f,0} = \sigma(D_{x_0}f)$, it holds
\begin{equation}\label{ControllabilitaPoli1}
	\sup_{ W }
	\norm{D_s f - D_s g}_{\op}
	\leq
	\frac{\sigma_{ f , 0 }}{4}  .
\end{equation}
By the density of polynomial maps, given $\eta_2>0$ 
we consider $p_1 , \dots ,  p_m$ 
such that for any $i= 1 , \dots , m$
\begin{equation}
\norm{w_i - p_i} < \eta_2  .
\end{equation}
We get the following estimate
\begin{align}
	\norm{D_0 f - D_0 g}_{\op}
	& \leq
	\sum_{i=1}^{m}
	\|D_{u_0} \End (w_i) - D_{q_0} \End (p_i) \| \\
	& \leq
	\sum_{i=1}^{m}
	\| D_{u_0} \End (w_i) - D_{q_0} \End (w_i) \| 
	+
	\|D_{q_0} \End (w_i) - D_{q_0} \End (p_i) \| \\
	& \leq
	\norm{D_{u_0} \End - D_{q_0} \End }_{\op}
	\sum_{i=1}^{m} \norm{w_i}
	+ 
	\norm{D_{q_0} \End }_{\op}	
	\sum_{i=1}^{m} \norm{w_i - p_i} \\
	& \leq
	\eta_1
	\sum_{i=1}^{m} \norm{w_i}
	+ 
	\norm{D_{q_0} \End }_{\op}	
	m \eta_2  .
\end{align}
Given $\eta_3 >0$, by the continuity of $D f$ and $D g$ there exists a neighbourhood $W\subset \R^m$ of $0$ such that
\begin{equation}
	\sup_{ s \in W }
	\norm{ D_s f -  D_0 f }_{\op}
	+ 
	\norm{D_s g - D_0 g }_{\op}
	<
	\eta_3   .
\end{equation}
Hence, by the triangle inequality we get the following estimate 
\begin{equation}
\sup_{ s \in W }
\norm{D_s f - D_s g}_{\op}
\leq
\eta_3 
+
\eta_1
\sum_{i=1}^{m} \norm{w_i}
+ 
\norm{D_{q_0} \End }_{\op}
m \eta_2  .
\end{equation}
To obtain \eqref{ControllabilitaPoli1} it is enough to choose $\eta_1 , \eta_2 , \eta_3$ small enough.
We have thus found $g$ in such a way $\|f(0) - g(0)\| < \varepsilon_1$ and $g$ satisfies the hypothesis of \cref{Lemma5Clarke}. By the latter, we find $r_{f,0}>0$ such that
\begin{equation}
	g(B_{\R^m} (r_{f,0}) ) 
	\supset 
	B_{ \R^m } 
	\left( 
	g (0),
	 \frac{ r_{f,0}\sigma_{f , 0}}{8} \right )  .
\end{equation}
Since $\|g(0)\|<\varepsilon_1$ by construction, taking $\varepsilon_1 = \frac{r_{f,0} \sigma_{ f , 0 } }{16}$ we also get
\begin{equation}\label{SuriettivitaPolinomi1}
g ( B_{ \R^m } (r_{f , 0}) ) \supset B_{ \R^m }
	\left ( \frac{r_{f , 0} \sigma_{f, 0} }{ 16 }  \right).
\end{equation}
Let then $d_\gr$ be the maximum degree of the polynomial maps 
$q_0 ,  \dots , p_m$.
We have proved that
\begin{equation}\label{eq:surj1}
	B_{ \R^m }
\left ( \frac{r_{f , 0} \sigma_{f, 0} }{ 16 }  \right)
\subset
\End \left(\{ u \in L^2(I , \R^k) \mid u \text{ is a polynomial map of degree at most } d_\gr \}\right)  .
\end{equation}
In other words, the Endpoint map is surjective on a small ball when restricted to polynomial controls of degree $\leq d_{\gr}$. Now recall that, for Carnot groups, dilations have the following property:
\begin{equation}\label{eq:dilationsF}
\delta_{\lambda}(\End(u)) = \End(\lambda u),\qquad \forall\, \lambda >0, \quad u \in L^2(I,\R^k).
\end{equation}
From \eqref{eq:surj1} and \eqref{eq:dilationsF} it follows that $\End$, when restricted to polynomial controls of degree $\leq d_{\gr}$ is surjective on the whole $\gr \simeq \R^m$. 
\end{proof}

With the same proof we can obtain the following statement, corresponding to \cref{thm:contdensurj-intro} in the Introduction, of which \cref{thm:polysurj} is a special case.

\begin{proposition}\label{thm:contdensurj}
Let $\gr$ be a Carnot group with topological dimension $m$ and rank $k$. Let $S \subset L^2(I,\R^k)$ be a dense set. Then there exist $u_0,u_1,\dots,u_m \in S$ such that
\begin{equation}
\End(\mathrm{span}\{u_0,u_1,\dots,u_m\}) = \gr.
\end{equation}
\end{proposition}

\bibliographystyle{alphaabbr}

	\bibliography{Sard_references}

@article {LRT-selection,
    AUTHOR = {Lerario, Antonio and Rizzi, Luca and Tiberio, Daniele},
     TITLE = {Quantitative approximate definable choices},
   JOURNAL = {Math. Ann.},
  FJOURNAL = {Mathematische Annalen},
    VOLUME = {392},
      YEAR = {2025},
    NUMBER = {1},
     PAGES = {1289--1319},
      ISSN = {0025-5831},
   MRCLASS = {14P10},
  MRNUMBER = {4887789},
       DOI = {10.1007/s00208-025-03128-3},
       URL = {https://doi.org/10.1007/s00208-025-03128-3},
}

@book {Vitbook,
    AUTHOR = {Vitu\v{s}kin, A. G.},
     TITLE = {Theory of the transmission and processing of information},
      NOTE = {Translated from the Russian by Ruth Feinstein; translation
              editor A. D. Booth},
 PUBLISHER = {Pergamon Press, New York-Oxford-London-Paris},
      YEAR = {1961},
     PAGES = {xvi+206},
   MRCLASS = {41.30},
  MRNUMBER = {132342},
MRREVIEWER = {G. G. Lorentz},
}

@article {ZZ-Rigid,
    AUTHOR = {Zelenko, I. and Zhitomirski\u{\i}, M.},
     TITLE = {Rigid paths of generic {$2$}-distributions on {$3$}-manifolds},
   JOURNAL = {Duke Math. J.},
  FJOURNAL = {Duke Mathematical Journal},
    VOLUME = {79},
      YEAR = {1995},
    NUMBER = {2},
     PAGES = {281--307},
      ISSN = {0012-7094},
   MRCLASS = {58A17 (58A30 58E10)},
  MRNUMBER = {1344763},
MRREVIEWER = {Richard W. Montgomery},
       DOI = {10.1215/S0012-7094-95-07907-1},
       URL = {https://doi.org/10.1215/S0012-7094-95-07907-1},
}

@article {agrasmoothness,
    AUTHOR = {Agrachev, Andrei},
     TITLE = {Any sub-{R}iemannian metric has points of smoothness},
   JOURNAL = {Dokl. Akad. Nauk},
  FJOURNAL = {Rossi\u{\i}skaya Akademiya Nauk. Doklady Akademii Nauk},
    VOLUME = {424},
      YEAR = {2009},
    NUMBER = {3},
     PAGES = {295--298},
      ISSN = {0869-5652},
   MRCLASS = {53C17 (49J15)},
  MRNUMBER = {2513150},
MRREVIEWER = {Vladimir Krouglov},
       DOI = {10.1134/S106456240901013X},
       URL = {https://doi.org/10.1134/S106456240901013X},
}

@article {BR-Inv,
    AUTHOR = {Barilari, Davide and Rizzi, Luca},
     TITLE = {Sub-{R}iemannian interpolation inequalities},
   JOURNAL = {Invent. Math.},
  FJOURNAL = {Inventiones Mathematicae},
    VOLUME = {215},
      YEAR = {2019},
    NUMBER = {3},
     PAGES = {977--1038},
      ISSN = {0020-9910},
   MRCLASS = {53C17 (49J15 49Q20)},
  MRNUMBER = {3935035},
MRREVIEWER = {Emmanuel Tr\'{e}lat},
       DOI = {10.1007/s00222-018-0840-y},
       URL = {https://doi.org/10.1007/s00222-018-0840-y},
}

@incollection{Suss-realanal,
  author={Sussmann, Héctor J.},
  booktitle={53rd IEEE Conference on Decision and Control}, 
  title={A regularity theorem for minimizers of real-analytic subriemannian metrics}, 
  year={2014},
  volume={},
  number={},
  pages={4801-4806},
  doi={10.1109/CDC.2014.7040138},
  }

@article {OV-codimensionCarnotSard,
    AUTHOR = {Ottazzi, Alessandro and Vittone, Davide},
     TITLE = {On the codimension of the abnormal set in step two {C}arnot
              groups},
   JOURNAL = {ESAIM Control Optim. Calc. Var.},
  FJOURNAL = {ESAIM. Control, Optimisation and Calculus of Variations},
    VOLUME = {25},
      YEAR = {2019},
     PAGES = {Paper No. 18, 17},
      ISSN = {1292-8119},
   MRCLASS = {53C17 (14M17 22E25)},
  MRNUMBER = {3981990},
MRREVIEWER = {Andrea Pinamonti},
       DOI = {10.1051/cocv/2018002},
       URL = {https://doi.org/10.1051/cocv/2018002},
}

@article{BPR-Sardpreprint3,
  TITLE = {Abnormal subanalytic distributions in sub-{R}iemannian geometry},
    AUTHOR = {Belotto da Silva, Andr\'{e} and Parusi\'{n}ski, Adam and Rifford,
              Ludovic},
  URL = {https://hal.science/hal-04881557},
  YEAR = {2025},
  MONTH = Jan,
  HAL_ID = {hal-04881557},
  HAL_VERSION = {v1},
}

@article{BPR-Sardpreprint2,
      title={The Analytic Minimal Rank Sard Conjecture}, 
    AUTHOR = {Belotto da Silva, Andr\'{e} and Parusi\'{n}ski, Adam and Rifford,
              Ludovic},
      year={2025},
      eprint={2208.01392},
      archivePrefix={arXiv},
      primaryClass={math.DG},
      url={https://arxiv.org/abs/2208.01392}, 
}

@article {BPR-Sardpreprint1,
    AUTHOR = {Belotto da Silva, Andr\'{e} and Parusi\'{n}ski, Adam and Rifford,
              Ludovic},
     TITLE = {Abnormal singular foliations and the {S}ard conjecture for
              generic co-rank one distributions},
   JOURNAL = {Rev. Mat. Iberoam.},
  FJOURNAL = {Revista Matem\'{a}tica Iberoamericana},
    VOLUME = {41},
      YEAR = {2025},
    NUMBER = {5},
     PAGES = {1599--1628},
      ISSN = {0213-2230},
   MRCLASS = {53C17 (58A30)},
  MRNUMBER = {4938554},
       DOI = {10.4171/rmi/1567},
       URL = {https://doi.org/10.4171/rmi/1567},
}

@article {BFPR-StrongSardInventiones,
    AUTHOR = {Belotto da Silva, Andr{é} and Figalli, Alessio and Parusi\'{n}ski, Adam and
              Rifford, Ludovic},
     TITLE = {Strong {S}ard conjecture and regularity of singular minimizing
              geodesics for analytic sub-{R}iemannian structures in
              dimension 3},
   JOURNAL = {Invent. Math.},
  FJOURNAL = {Inventiones Mathematicae},
    VOLUME = {229},
      YEAR = {2022},
    NUMBER = {1},
     PAGES = {395--448},
      ISSN = {0020-9910},
   MRCLASS = {53C17 (53C22)},
  MRNUMBER = {4438357},
MRREVIEWER = {Luca Rizzi},
       DOI = {10.1007/s00222-022-01111-2},
       URL = {https://doi.org/10.1007/s00222-022-01111-2},
}

@article {BR-DukeMartinet,
    AUTHOR = {Belotto da Silva, Andr\'{e} and Rifford, Ludovic},
     TITLE = {The {S}ard conjecture on {M}artinet surfaces},
   JOURNAL = {Duke Math. J.},
  FJOURNAL = {Duke Mathematical Journal},
    VOLUME = {167},
      YEAR = {2018},
    NUMBER = {8},
     PAGES = {1433--1471},
      ISSN = {0012-7094},
   MRCLASS = {53A99 (32S45 34H05)},
  MRNUMBER = {3807314},
MRREVIEWER = {Davide Vittone},
       DOI = {10.1215/00127094-2017-0058},
       URL = {https://doi.org/10.1215/00127094-2017-0058},
}

@article {CJT-generic,
    AUTHOR = {Chitour, Y. and Jean, F. and Tr\'{e}lat, E.},
     TITLE = {Genericity results for singular curves},
   JOURNAL = {J. Differential Geom.},
  FJOURNAL = {Journal of Differential Geometry},
    VOLUME = {73},
      YEAR = {2006},
    NUMBER = {1},
     PAGES = {45--73},
      ISSN = {0022-040X},
   MRCLASS = {58A30 (49K15 53C17)},
  MRNUMBER = {2217519},
MRREVIEWER = {Yuri L. Sachkov},
       URL = {http://projecteuclid.org/euclid.jdg/1146680512},
}

@article {AG-subanalitic,
    AUTHOR = {Agrachev, Andrei and Gauthier, Jean-Paul},
     TITLE = {On the subanalyticity of {C}arnot-{C}aratheodory distances},
   JOURNAL = {Ann. Inst. H. Poincar\'{e} C Anal. Non Lin\'{e}aire},
  FJOURNAL = {Annales de l'Institut Henri Poincar\'{e} C. Analyse Non Lin\'{e}aire},
    VOLUME = {18},
      YEAR = {2001},
    NUMBER = {3},
     PAGES = {359--382},
      ISSN = {0294-1449},
   MRCLASS = {93B27 (53C17 93C10)},
  MRNUMBER = {1831660},
MRREVIEWER = {Hector J. Sussmann},
       DOI = {10.1016/S0294-1449(00)00064-0},
       URL = {https://doi.org/10.1016/S0294-1449(00)00064-0},
}

@article {R-subdiff,
    AUTHOR = {Rifford, Ludovic},
     TITLE = {Subdifferentials and minimizing {S}ard conjecture in
              sub-{R}iemannian geometry},
   JOURNAL = {J. \'{E}c. polytech. Math.},
  FJOURNAL = {Journal de l'\'{E}cole polytechnique. Math\'{e}matiques},
    VOLUME = {10},
      YEAR = {2023},
     PAGES = {1195--1244},
      ISSN = {2429-7100},
   MRCLASS = {53C17 (49J52 49Q15)},
  MRNUMBER = {4645934},
}

@article {AS-Morse,
    AUTHOR = {Agrachev, Andrei and Sarychev, A. V.},
     TITLE = {Abnormal sub-{R}iemannian geodesics: {M}orse index and
              rigidity},
   JOURNAL = {Ann. Inst. H. Poincar\'{e} C Anal. Non Lin\'{e}aire},
  FJOURNAL = {Annales de l'Institut Henri Poincar\'{e} C. Analyse Non Lin\'{e}aire},
    VOLUME = {13},
      YEAR = {1996},
    NUMBER = {6},
     PAGES = {635--690},
      ISSN = {0294-1449},
   MRCLASS = {49K15 (58A30 58E10)},
  MRNUMBER = {1420493},
MRREVIEWER = {Heinz Sch\"{a}ttler},
       DOI = {10.1016/S0294-1449(16)30118-4},
       URL = {https://doi.org/10.1016/S0294-1449(16)30118-4},
}

@article {AAPL-OT,
    AUTHOR = {Agrachev, Andrei and Lee, Paul},
     TITLE = {Optimal transportation under nonholonomic constraints},
   JOURNAL = {Trans. Amer. Math. Soc.},
  FJOURNAL = {Transactions of the American Mathematical Society},
    VOLUME = {361},
      YEAR = {2009},
    NUMBER = {11},
     PAGES = {6019--6047},
      ISSN = {0002-9947},
   MRCLASS = {49J15 (37J60 49J30 53C17)},
  MRNUMBER = {2529923},
MRREVIEWER = {Filippo Santambrogio},
       DOI = {10.1090/S0002-9947-09-04813-2},
       URL = {https://doi.org/10.1090/S0002-9947-09-04813-2},
}

@article {BNV-Filiform,
    AUTHOR = {Boarotto, Francesco and Nalon, Luca and Vittone, Davide},
     TITLE = {The {S}ard problem in step 2 and in filiform {C}arnot groups},
   JOURNAL = {ESAIM Control Optim. Calc. Var.},
  FJOURNAL = {ESAIM. Control, Optimisation and Calculus of Variations},
    VOLUME = {28},
      YEAR = {2022},
     PAGES = {Paper No. 75, 20},
      ISSN = {1292-8119},
   MRCLASS = {53C17 (22E25 58K05)},
  MRNUMBER = {4524416},
MRREVIEWER = {Scott Robert Zimmerman},
       DOI = {10.1051/cocv/2022074},
       URL = {https://doi.org/10.1051/cocv/2022074},
}

@article {BV-Dynamical,
    AUTHOR = {Boarotto, Francesco and Vittone, Davide},
     TITLE = {A dynamical approach to the {S}ard problem in {C}arnot groups},
   JOURNAL = {J. Differential Equations},
  FJOURNAL = {Journal of Differential Equations},
    VOLUME = {269},
      YEAR = {2020},
    NUMBER = {6},
     PAGES = {4998--5033},
      ISSN = {0022-0396},
   MRCLASS = {53C17 (37N35 58K05)},
  MRNUMBER = {4104464},
MRREVIEWER = {Emmanuel Tr\'{e}lat},
       DOI = {10.1016/j.jde.2020.03.050},
       URL = {https://doi.org/10.1016/j.jde.2020.03.050},
}

@article {RT-MorseSard,
    AUTHOR = {Rifford, L. and Tr\'{e}lat, E.},
     TITLE = {Morse-{S}ard type results in sub-{R}iemannian geometry},
   JOURNAL = {Math. Ann.},
  FJOURNAL = {Mathematische Annalen},
    VOLUME = {332},
      YEAR = {2005},
    NUMBER = {1},
     PAGES = {145--159},
      ISSN = {0025-5831},
   MRCLASS = {53C17 (49J52)},
  MRNUMBER = {2139255},
       DOI = {10.1007/s00208-004-0622-2},
       URL = {https://doi.org/10.1007/s00208-004-0622-2},
}

@article {Bates,
    AUTHOR = {Bates, S. M.},
     TITLE = {On smooth rank-{$1$} mappings of {B}anach spaces onto the
              plane},
   JOURNAL = {J. Differential Geom.},
  FJOURNAL = {Journal of Differential Geometry},
    VOLUME = {37},
      YEAR = {1993},
    NUMBER = {3},
     PAGES = {729--733},
      ISSN = {0022-040X},
   MRCLASS = {58C25 (58C27)},
  MRNUMBER = {1217168},
MRREVIEWER = {P. T. Church},
       URL = {http://projecteuclid.org/euclid.jdg/1214453907},
}

@incollection {A-openproblems,
    AUTHOR = {Agrachev, Andrei},
     TITLE = {Some open problems},
 BOOKTITLE = {Geometric control theory and sub-{R}iemannian geometry},
    SERIES = {Springer INdAM Ser.},
    VOLUME = {5},
     PAGES = {1--13},
 PUBLISHER = {Springer, Cham},
      YEAR = {2014},
   MRCLASS = {53-02 (49-02 53C17 93B27)},
  MRNUMBER = {3205092},
       DOI = {10.1007/978-3-319-02132-4\_1},
       URL = {https://doi.org/10.1007/978-3-319-02132-4_1},
}

@book {Montgomerybook,
    AUTHOR = {Montgomery, Richard},
     TITLE = {A tour of subriemannian geometries, their geodesics and
              applications},
    SERIES = {Mathematical Surveys and Monographs},
    VOLUME = {91},
 PUBLISHER = {American Mathematical Society, Providence, RI},
      YEAR = {2002},
     PAGES = {xx+259},
      ISBN = {0-8218-1391-9},
   MRCLASS = {53C17 (37J99 53C60 58E10 70G45 70H05)},
  MRNUMBER = {1867362},
MRREVIEWER = {Andrey V. Sarychev},
       DOI = {10.1090/surv/091},
       URL = {https://doi.org/10.1090/surv/091},
}

@article {AGL-pathspace,
    AUTHOR = {Agrachev, Andrei and Gentile, Alessandro and Lerario,
              Antonio},
     TITLE = {Geodesics and horizontal-path spaces in {C}arnot groups},
   JOURNAL = {Geom. Topol.},
  FJOURNAL = {Geometry \& Topology},
    VOLUME = {19},
      YEAR = {2015},
    NUMBER = {3},
     PAGES = {1569--1630},
      ISSN = {1465-3060},
   MRCLASS = {53C17 (37J60 58E10)},
  MRNUMBER = {3352244},
MRREVIEWER = {Mauricio Godoy},
       DOI = {10.2140/gt.2015.19.1569},
       URL = {https://doi.org/10.2140/gt.2015.19.1569},
}

@article {LDMOPV-Sardprop,
    AUTHOR = {Le Donne, Enrico and Montgomery, Richard and Ottazzi,
              Alessandro and Pansu, Pierre and Vittone, Davide},
     TITLE = {Sard property for the endpoint map on some {C}arnot groups},
   JOURNAL = {Ann. Inst. H. Poincar\'{e} C Anal. Non Lin\'{e}aire},
  FJOURNAL = {Annales de l'Institut Henri Poincar\'{e} C. Analyse Non Lin\'{e}aire},
    VOLUME = {33},
      YEAR = {2016},
    NUMBER = {6},
     PAGES = {1639--1666},
      ISSN = {0294-1449},
   MRCLASS = {53C17 (14M17 22E25)},
  MRNUMBER = {3569245},
MRREVIEWER = {Jeanne Nielsen Clelland},
       DOI = {10.1016/j.anihpc.2015.07.004},
       URL = {https://doi.org/10.1016/j.anihpc.2015.07.004},
}

@incollection {AAA-Rendiconti,
    AUTHOR = {Agrachev, A.},
     TITLE = {Compactness for sub-{R}iemannian length-minimizers and
              subanalyticity},
      NOTE = {Control theory and its applications (Grado, 1998)},
   JOURNAL = {Rend. Sem. Mat. Univ. Politec. Torino},
  FJOURNAL = {Universit\`a e Politecnico di Torino. Seminario Matematico.
              Rendiconti},
    VOLUME = {56},
      YEAR = {1998},
    NUMBER = {4},
     PAGES = {1--12 (2001)},
      ISSN = {0373-1243},
   MRCLASS = {93B29 (49J15 53C17 58E10)},
  MRNUMBER = {1845741},
}

@article {ClarkeInvFunc,
    AUTHOR = {Clarke, F. H.},
     TITLE = {On the inverse function theorem},
   JOURNAL = {Pacific J. Math.},
  FJOURNAL = {Pacific Journal of Mathematics},
    VOLUME = {64},
      YEAR = {1976},
    NUMBER = {1},
     PAGES = {97--102},
      ISSN = {0030-8730},
   MRCLASS = {26A57},
  MRNUMBER = {425047},
MRREVIEWER = {B. Rodr\'{\i}guez-Salinas},
       URL = {http://projecteuclid.org/euclid.pjm/1102867214},
}

@book {Vit1,
    AUTHOR = {Vitu\v{s}kin, A. G.},
     TITLE = {O mnogomernyh variaciyah},
 PUBLISHER = {Gosudarstv. Izdat. Tehn.-Teor. Lit., Moscow},
      YEAR = {1955},
     PAGES = {220},
   MRCLASS = {27.2X},
  MRNUMBER = {0075267},
MRREVIEWER = {L. C. Young},
}

@book {BCR,
    AUTHOR = {Bochnak, J. and Coste, M. and Roy, M.-F.},
     TITLE = {G\'eom\'etrie alg\'ebrique r\'eelle (Second edition in english: Real Algebraic Geometry)},
    SERIES = {Ergebnisse der Mathematik und ihrer Grenzgebiete [Results in Mathematics and Related Areas]},
    VOLUME = {12 (36)},
 PUBLISHER = {Springer-Verlag},
   ADDRESS = {Berlin},
      YEAR = {1987 (1998)},
     PAGES = {x+373},
      ISBN = {3-540-16951-2},
   MRCLASS = {14G30 (11E25 12D15 32C05 58A07)},
  MRNUMBER = {949442 (90b:14030)},
MRREVIEWER = {Jes{\'u}s M. Ruiz},
}

@article {FedererTransAmerMathSoc,
    AUTHOR = {Federer, Herbert},
     TITLE = {The {$(\varphi,k)$} rectifiable subsets of {$n$}-space},
   JOURNAL = {Trans. Amer. Math. Soc.},
  FJOURNAL = {Transactions of the American Mathematical Society},
    VOLUME = {62},
      YEAR = {1947},
     PAGES = {114--192},
      ISSN = {0002-9947},
   MRCLASS = {27.2X},
  MRNUMBER = {22594},
MRREVIEWER = {L. Cesari},
       DOI = {10.2307/1990632},
       URL = {https://doi.org/10.2307/1990632},
}

@incollection {Bellaiche,
    AUTHOR = {Bella\"{\i}che, Andr\'{e}},
     TITLE = {The tangent space in sub-{R}iemannian geometry},
 BOOKTITLE = {Sub-{R}iemannian geometry},
    SERIES = {Progr. Math.},
    VOLUME = {144},
     PAGES = {1--78},
 PUBLISHER = {Birkh\"{a}user, Basel},
      YEAR = {1996},
      ISBN = {3-7643-5476-3},
   MRCLASS = {53C99 (57R27)},
  MRNUMBER = {1421822},
MRREVIEWER = {Claudio\ Gorodski},
       DOI = {10.1007/978-3-0348-9210-0\_1},
       URL = {https://doi.org/10.1007/978-3-0348-9210-0_1},
}

@book {ABB,
    AUTHOR = {Agrachev, Andrei and Barilari, Davide and Boscain, Ugo},
     TITLE = {A comprehensive introduction to sub-{R}iemannian geometry},
    SERIES = {Cambridge Studies in Advanced Mathematics},
    VOLUME = {181},
      NOTE = {From the Hamiltonian viewpoint,
              With an appendix by Igor Zelenko},
 PUBLISHER = {Cambridge University Press, Cambridge},
      YEAR = {2020},
     PAGES = {xviii+745},
      ISBN = {978-1-108-47635-5},
   MRCLASS = {53C17},
  MRNUMBER = {3971262},
MRREVIEWER = {Luca\ Rizzi},
}

@book {TaoRMT2012,
    AUTHOR = {Tao, Terence},
     TITLE = {Topics in random matrix theory},
    SERIES = {Graduate Studies in Mathematics},
    VOLUME = {132},
 PUBLISHER = {American Mathematical Society, Providence, RI},
      YEAR = {2012},
     PAGES = {x+282},
      ISBN = {978-0-8218-7430-1},
   MRCLASS = {60B20 (15B52)},
  MRNUMBER = {2906465},
MRREVIEWER = {Steven\ Joel\ Miller},
       DOI = {10.1090/gsm/132},
       URL = {https://doi.org/10.1090/gsm/132},
}

@article {Lev,
    AUTHOR = {Lokutsievskiy, Lev and Zelikin, Mikhail},
     TITLE = {Derivatives of sub-{R}iemannian geodesics are
              {$L_{p}$}-{H}\"{o}lder continuous},
   JOURNAL = {ESAIM Control Optim. Calc. Var.},
  FJOURNAL = {ESAIM. Control, Optimisation and Calculus of Variations},
    VOLUME = {29},
      YEAR = {2023},
     PAGES = {Paper No. 70, 30},
      ISSN = {1292-8119,1262-3377},
   MRCLASS = {53C17 (49J15 49N60)},
  MRNUMBER = {4629512},
       DOI = {10.1051/cocv/2023055},
       URL = {https://doi.org/10.1051/cocv/2023055},
}

@article {realpoly2,
    AUTHOR = {Kuhn, Harald},
     TITLE = {Interpolation vorgeschriebener {E}xtremwerte},
   JOURNAL = {J. Reine Angew. Math.},
  FJOURNAL = {Journal f\"{u}r die Reine und Angewandte Mathematik. [Crelle's
              Journal]},
    VOLUME = {238},
      YEAR = {1969},
     PAGES = {24--31},
      ISSN = {0075-4102,1435-5345},
   MRCLASS = {41.10},
  MRNUMBER = {247332},
MRREVIEWER = {J.\ Hersch},
       DOI = {10.1515/crll.1969.238.24},
       URL = {https://doi.org/10.1515/crll.1969.238.24},
}

@article {realpoly,
    AUTHOR = {Kristiansen, G. K.},
     TITLE = {Characterization of polynomials by means of their stationary
              values},
   JOURNAL = {Arch. Math. (Basel)},
  FJOURNAL = {Archiv der Mathematik},
    VOLUME = {43},
      YEAR = {1984},
    NUMBER = {1},
     PAGES = {44--48},
      ISSN = {0003-889X,1420-8938},
   MRCLASS = {12D05 (26C10 30C10 42A05)},
  MRNUMBER = {758339},
MRREVIEWER = {Mahfooz\ Alam},
       DOI = {10.1007/BF01193610},
       URL = {https://doi.org/10.1007/BF01193610},
}

@book {LangRealandFunct,
    AUTHOR = {Lang, Serge},
     TITLE = {Real and functional analysis},
    SERIES = {Graduate Texts in Mathematics},
    VOLUME = {142},
   EDITION = {Third},
 PUBLISHER = {Springer-Verlag, New York},
      YEAR = {1993},
     PAGES = {xiv+580},
      ISBN = {0-387-94001-4},
   MRCLASS = {00A05 (26-01 28-01 46-01 47-01 58-01)},
  MRNUMBER = {1216137},
       DOI = {10.1007/978-1-4612-0897-6},
       URL = {https://doi.org/10.1007/978-1-4612-0897-6},
}

@book {BasuPoRoyBook,
    AUTHOR = {Basu, Saugata and Pollack, Richard and Roy,
              Marie-Fran\c{c}oise},
     TITLE = {Algorithms in real algebraic geometry},
    SERIES = {Algorithms and Computation in Mathematics},
    VOLUME = {10},
   EDITION = {Second},
 PUBLISHER = {Springer-Verlag, Berlin},
      YEAR = {2006},
     PAGES = {x+662},
      ISBN = {978-3-540-33098-1; 3-540-33098-4},
   MRCLASS = {14P10 (03C10 52C45 68Q25 68W30)},
  MRNUMBER = {2248869},
       URL = {https://perso.univ-rennes1.fr/marie-francoise.roy/bpr-ed2-posted3.html},
}

@article {Yomdinnearly,
    AUTHOR = {Yomdin, Y.},
     TITLE = {The geometry of critical and near-critical values of
              differentiable mappings},
   JOURNAL = {Math. Ann.},
  FJOURNAL = {Mathematische Annalen},
    VOLUME = {264},
      YEAR = {1983},
    NUMBER = {4},
     PAGES = {495--515},
      ISSN = {0025-5831,1432-1807},
   MRCLASS = {58C25 (41A99)},
  MRNUMBER = {716263},
MRREVIEWER = {S.\ \L ojasiewicz},
       DOI = {10.1007/BF01456957},
       URL = {https://doi.org/10.1007/BF01456957},
}

@article {Kupka,
    AUTHOR = {Kupka, Ivan},
     TITLE = {Counterexample to the {M}orse-{S}ard theorem in the case of
              infinite-dimensional manifolds},
   JOURNAL = {Proc. Amer. Math. Soc.},
  FJOURNAL = {Proceedings of the American Mathematical Society},
    VOLUME = {16},
      YEAR = {1965},
     PAGES = {954--957},
      ISSN = {0002-9939,1088-6826},
   MRCLASS = {57.55 (57.50)},
  MRNUMBER = {182024},
MRREVIEWER = {S.\ Rolewicz},
       DOI = {10.2307/2035591},
       URL = {https://doi.org/10.2307/2035591},
}

@article {Morse,
    AUTHOR = {Morse, Anthony P.},
     TITLE = {The behavior of a function on its critical set},
   JOURNAL = {Ann. of Math. (2)},
  FJOURNAL = {Annals of Mathematics. Second Series},
    VOLUME = {40},
      YEAR = {1939},
    NUMBER = {1},
     PAGES = {62--70},
      ISSN = {0003-486X,1939-8980},
   MRCLASS = {99-04},
  MRNUMBER = {1503449},
       DOI = {10.2307/1968544},
       URL = {https://doi.org/10.2307/1968544},
}

@article {Sard,
    AUTHOR = {Sard, Arthur},
     TITLE = {The measure of the critical values of differentiable maps},
   JOURNAL = {Bull. Amer. Math. Soc.},
  FJOURNAL = {Bulletin of the American Mathematical Society},
    VOLUME = {48},
      YEAR = {1942},
     PAGES = {883--890},
      ISSN = {0002-9904},
   MRCLASS = {27.2X},
  MRNUMBER = {7523},
MRREVIEWER = {H.\ Blumberg},
       DOI = {10.1090/S0002-9904-1942-07811-6},
       URL = {https://doi.org/10.1090/S0002-9904-1942-07811-6},
}

@article {smale,
    AUTHOR = {Smale, S.},
     TITLE = {An infinite dimensional version of {S}ard's theorem},
   JOURNAL = {Amer. J. Math.},
  FJOURNAL = {American Journal of Mathematics},
    VOLUME = {87},
      YEAR = {1965},
     PAGES = {861--866},
      ISSN = {0002-9327,1080-6377},
   MRCLASS = {57.55 (57.50)},
  MRNUMBER = {185604},
MRREVIEWER = {Richard\ Beals},
       DOI = {10.2307/2373250},
       URL = {https://doi.org/10.2307/2373250},
}

@article {Kellogg1928,
    AUTHOR = {Kellogg, O. D.},
     TITLE = {On bounded polynomials in several variables},
   JOURNAL = {Math. Z.},
  FJOURNAL = {Mathematische Zeitschrift},
    VOLUME = {27},
      YEAR = {1928},
    NUMBER = {1},
     PAGES = {55--64},
      ISSN = {0025-5874},
   MRCLASS = {DML},
  MRNUMBER = {1544896},
       DOI = {10.1007/BF01171085},
       URL = {https://doi.org/10.1007/BF01171085},
}

@incollection {YomdinApproxCompl,
    AUTHOR = {Yomdin, Y.},
     TITLE = {Approximational complexity of functions},
 BOOKTITLE = {Geometric aspects of functional analysis (1986/87)},
    SERIES = {Lecture Notes in Math.},
    VOLUME = {1317},
     PAGES = {21--43},
 PUBLISHER = {Springer, Berlin},
      YEAR = {1988},
      ISBN = {3-540-19353-7},
   MRCLASS = {58C99 (26B99 68Q30)},
  MRNUMBER = {950974},
MRREVIEWER = {Pierre\ D.\ Milman},
       DOI = {10.1007/BFb0081734},
       URL = {https://doi.org/10.1007/BFb0081734},
}

@article {BoarLer17,
    AUTHOR = {Boarotto, Francesco and Lerario, Antonio},
     TITLE = {Homotopy properties of horizontal path spaces and a theorem of
              {S}erre in subriemannian geometry},
   JOURNAL = {Comm. Anal. Geom.},
  FJOURNAL = {Communications in Analysis and Geometry},
    VOLUME = {25},
      YEAR = {2017},
    NUMBER = {2},
     PAGES = {269--301},
      ISSN = {1019-8385,1944-9992},
   MRCLASS = {53C17 (93B27)},
  MRNUMBER = {3690242},
MRREVIEWER = {Francesco\ Rossi},
       DOI = {10.4310/CAG.2017.v25.n2.a1},
       URL = {https://doi.org/10.4310/CAG.2017.v25.n2.a1},
}

@book {Pinkus,
    AUTHOR = {Pinkus, Allan},
     TITLE = {{$n$}-widths in approximation theory},
    SERIES = {Ergebnisse der Mathematik und ihrer Grenzgebiete (3) [Results
              in Mathematics and Related Areas (3)]},
    VOLUME = {7},
 PUBLISHER = {Springer-Verlag, Berlin},
      YEAR = {1985},
     PAGES = {x+291},
      ISBN = {3-540-13638-X},
   MRCLASS = {41-02 (41A46 46E35)},
  MRNUMBER = {774404},
MRREVIEWER = {Klaus H\"{o}llig},
       DOI = {10.1007/978-3-642-69894-1},
       URL = {https://doi.org/10.1007/978-3-642-69894-1},
}

@book {ComteYomdin,
    AUTHOR = {Yomdin, Yosef and Comte, Georges},
     TITLE = {Tame geometry with application in smooth analysis},
    SERIES = {Lecture Notes in Mathematics},
    VOLUME = {1834},
 PUBLISHER = {Springer-Verlag, Berlin},
      YEAR = {2004},
     PAGES = {viii+186},
      ISBN = {3-540-20612-4},
   MRCLASS = {14P10 (26B15 32S15 37C99)},
  MRNUMBER = {2041428},
MRREVIEWER = {Lev Birbrair},
       DOI = {10.1007/b94624},
       URL = {https://doi.org/10.1007/b94624},
}
\end{document}